\documentclass[letterpaper,12pt,reqno]{amsart}
\usepackage{amssymb}
\usepackage{cases}
\usepackage{amsmath}
\usepackage{mathrsfs}
\usepackage{amsfonts}
\usepackage{amsfonts}
\usepackage[all]{xy}
\usepackage{color}
\DeclareFontFamily{OT1}{pzc}{}
\DeclareFontShape{OT1}{pzc}{m}{it}{<-> [1.15] rpzcmi}{}
\DeclareMathAlphabet{\mathzc}{OT1}{pzc}{m}{it}

\def\End{\operatorname{End}\kern-.5pt}

\def\KK{{\mathzc K\kern0pt}}
\def\vv{{\mathzc v\kern.5pt}}

\def\aq{/\kern-2pt/}
\def\Hom{\operatorname{Hom}}
\def\diag{\operatorname{diag}}

\def\La{{\Lambda}}

\def\fS{{\mathfrak S}}
\def\End{{\text{\rm End}}}
\def\Hom{{\text{\rm Hom}}}
\def\row{{\text{\rm row}}}
\def\col{{\text{\rm col}}}

\def\bsi{{\boldsymbol i}}
\def\bsj{{\boldsymbol j}}

\def\bsq{{\boldsymbol q}}
\def\bse{{\boldsymbol e}}
\def\up{{\boldsymbol{\upsilon}}}
\def\sH{{\mathcal H}}
\def\sS{{\mathcal S}}
\def\sZ{{\mathcal Z}}
\def\sD{{\mathcal D}}

\def\sA{{\mathcal A}}
\def\sE{{\mathcal E}}

\def\la{{\lambda}}
\def\al{{\alpha}}

\def\whd{{\widehat{d}}}

\def\sL{{\mathcal L}}

\def\vep{{\varepsilon}}

\def\xib{{\xi}}

\def\ugk{{{K}}}
\def\sfK{{\mathsf K}}
\def\sfE{{\mathsf E}}

\def\sfM{{\mathsf M}}

\def\fA{{\boldsymbol{\mathfrak A}}}
\def\bfU{{\mathbf U}}
\def\bsS{{\boldsymbol{\sS}}}

\newtheorem{theorem}{Theorem}[section]
\newtheorem{lemma}[theorem]{Lemma}
\newtheorem{proposition}[theorem]{Proposition}
\newtheorem{corollary}[theorem]{Corollary}
\theoremstyle{definition}
\newtheorem{definition}[theorem]{Definition}

\newtheorem{algorithm}[theorem]{Algorithm}

\newtheorem{remark}[theorem]{Remark}

\numberwithin{equation}{theorem}

\def\co{{\rm co}}
\def\ro{{\rm ro}}
\def\Hq{{\sH_\bsq(r)}}
\def\Sq{{\sS_\bsq(m|n,r)}}
\def\vep{{\varepsilon}}
\def\scrT{{\mathscr T}}
\def\hat{\widehat}
\def\fw{{\mathfrak v}}

\begin{document}

\baselineskip16pt
\def\hei{\relax}

 \title[A Realisation of the quantum superalgebra $\bfU(\mathfrak{gl}_{m|n})$ ]{A Realisation of the quantum supergroup $\bfU(\mathfrak{gl}_{m|n})$ }
\author{Jie Du and Haixia Gu}
\address{J.D., School of Mathematics and Statistics,
University of New South Wales, Sydney NSW 2052, Australia}
\email{j.du@unsw.edu.au}
\address{H.G., Department of Mathematics,
East China Normal University (Minhang Campus), Shanghai 200241, China}
\email{alla0824@126.com}
\date{\today}
\thanks{The first author gratefully acknowledge support from ARC under grant DP120101436. The work was completed while the second author was visiting UNSW.
She would like to thank the China Scholarship Council and ARC for
financial support and UNSW for its hospitality during her visit.}


\begin{abstract}We reconstruct the quantum enveloping superalgebra $\bfU(\mathfrak{gl}_{m|n})$ over $\mathbb Q(\up)$ via (finite dimensional) quantum Schur superalgebras.
In particular, we obtain a new basis containing the standard generators for $\bfU(\mathfrak{gl}_{m|n})$ and explicit multiplication formulas between the generators and an arbitrary basis element.
\end{abstract}

 \maketitle

When an algebra is defined by generators and relations, the realisation problem of the algebra is to reconstruct the algebra via a certain vector space---the realisation space---together with some explicit multiplication rules on a basis of the space.
For example, for Kac--Moody algebras, the affine case is constructed in \cite{Kac2}
via loop algebras and central extensions while, for symmetrizable Kac--Moody algebras, the problem was solved by Peng and Xiao \cite{PX} via root categories and Hall multiplication.
If the realisation space happens to be the Grothendieck group of a certain representation category, such a realisation is also known as categorification.

Since the introduction of quantum groups in late eighties, their realisation and applications have achieved significant progress. For example,
 Ringel \cite{R90} gave a realisation for the $\pm$-parts of quantum enveloping algebras of finite type. This work has also been generalised by Lusztig \cite{Lu91,Lu93}, Green \cite{Gr95}
 to the symmetrizable case. The realisation problem for the entire quantum enveloping algebras has also attracted much attention. For example, Beilinson, Lusztig and MacPherson
solved the problem for quantum $\mathfrak{gl}_n$ via endomorphism algebras of certain representations, namely, the quantum Schur algebras. This approach has also been investigated in the affine case by Deng, Fu and first author; see \cite{DF09} and \cite{DDF}. Recently,
Bridgeland \cite{Br} has solved the problem for all quantum enveloping algebras of finite type, generalising the construction in \cite{PX}. In this paper, we will tackle the problem for the quantum groups associated with the Lie superalgebra $\mathfrak{gl}_{m|n}$.

The work was motivated by several recent developments about the quantum Schur or $\up$-Schur (if the parameter needs to be specified) superalgebras $\bsS(m|n,r)$. Firstly, Rui and the first author \cite{DR} introduced these algebras and established several properties in analogy to those of quantum Schur algebras. For example, the construction of standard bases via super double cosets gives the base change property, the Kazhdan--Lusztig combinatorics via super cells and super Robinson--Schensted-Knuth correspondence gives a construction of simple modules in terms cells modules. Moreover, these algebras have also alternative characterisations as a homomorphic image of the quantum supergroup $\bfU(\mathfrak{gl}_{m|n})$ and as the dual algebras
of the homogeneous components of the super matrix coordinate bialgebra.
Secondly,  a presentation for these algebras have been given by Turkey and Kujawa \cite{TK} via the epimorphism from $\bfU(\mathfrak{gl}_{m|n})$ to $\bsS(m|n,r)$. Thirdly, by applying the relative norm method of Hoefsmit and Scott developed in 1977 (see \cite{J}) and its application to quantum Schur algebras \cite{Du91-2}, the authors together with Wang \cite{DGW} determined a classification of their irreducible representations at a root of unity when $m+n\geq r$.

All these works suggest that the realisation of $\bfU(\mathfrak{gl}_{m|n})$ via quantum Schur superalgebras should exist. However, since the BLM realisation for quantum $\mathfrak{gl}_{n}$ relies on a geometric setting for quantum Schur algebras, it is inevitable that one has to develop a new and algebraic approach in the super case. Motivated from a large amount of computation with relative norms in \cite{DGW},
 we discover in this paper that the relative norm method of Hoefsmit and Scott can replace the geometric method to establish two key multiplication formulas and, thus, to solve the realisation problem for the quantum supergroup $\bfU(\mathfrak{gl}_{m|n})$.

 We organise the paper as follows. The basics of symmetric groups, Hecke algebras and quantum Schur superalgebras are briefly discussed in \S1. In particular, we use the Hecke algebra action on the tensor superspace and display the relative norm basis for the $\bsq$-Schur superalgebra. In \S 2, we develop some algorithms for computing a reduced expression of a distinguished double coset representative associated with a matrix and computing the corresponding Hecke algebra action on various tensors.
With these algorithms, we are able to derive in \S3 the super version of two key multiplication formulas which in the non super or affine case were obtained by analysis on (partial) flag varieties.

Now, like \cite{BLM}, we are in position to get more multiplication formulas first with respect to the parameter $\bsq=\up^2$ in \S 4 and then their normalised version with respect to parameter $\up$ in \S5. Uniform spanning sets for $\up$-Schur superalgebras $\bsS(m|n,r)$ for all $r\geq0$ are constructed in \S6. The uniformness allows to define a $\mathbb Q(\up)$-subspace $\fA(m|n)$ inside the direct product $\bsS(m|n)$ of $\bsS(m|n,r)$ over all $r\geq0$. We then prove that, by deriving explicit multiplication formulas between (candidates of) generators and basis elements of $\fA(m|n)$, $\fA(m|n)$ is closed under multiplication by these generators. By producing the defining relations for $\bfU(\mathfrak{gl}_{m|n})$ in $\bsS(m|n)$, we find in \S7 an algebra homomorphism $\eta$ from $\bfU(\mathfrak{gl}_{m|n})$ to $\bsS(m|n)$. Finally, as a further application of the multiplication formulas, we establish a super triangular relation in \S8. This relation gives us a monomial basis which is crucial to prove that $\fA(m|n)$ is the image of $\eta$ and $\eta$ is injective.

Like the non super case, this realisation will have many applications such as a super version of the quantum Schur--Weyl duality at a root of unity and realisation of the finite dimensional quantum supergroups. We plan to complete these tasks in forthcoming papers.

Throughout the paper,
let $\up$ be an indeterminate and let $\bsq=\up^2$. Let $\sZ=\mathbb Z[\up,\up^{-1}]$ and let $\sA=\mathbb Z[\bsq,\bsq^{-1}]$.
For any integers $0\leq t\leq s$, define Gaussian polynomials in $\sA$ by
$$\left[\!\!\left[t\atop s\right]\!\!\right]=\frac{[\![s]\!]^!}{[\![t]\!]^![\![s-t]\!]^!}$$
where $[\![r]\!]^{!}:=[\![1]\!][\![2]\!]\cdots[\![r]\!]$ with $[\![i]\!]=1+\bsq+\cdots+\bsq^{i-1}.$ By introducing
$[i]=\frac{\up^i-\up^{-i}}{\up-\up^{-1}}$, we define the symmetric Gaussian polynomials $[r]^!$ and $\left[t\atop s\right]$ in $\sZ$
similarly. Clearly, $\left[\!\!\left[t\atop s\right]\!\!\right]=\up^{s(t-s)}\left[t\atop s\right]$.

The letters $m,n$ denote two arbitrary fixed nonnegative integers (not both zero) and, for $h\in[1,m+n]:=\{1,2,\ldots,m+n\}$, define the parity function
\begin{equation}\label{parity}
\hat{h}=\begin{cases}0,&\text{if }1\leq h\leq m;\\1,&\text{if }m+1\leq h\leq m+n.
\end{cases}\end{equation}
Let $\bsq_h=\bsq^{(-1)^{\hat{h}}}$ and $\up_h=\up^{(-1)^{\hat{h}}}$.

 \section{Preliminaries}


 Let $(W,S)$ be the symmetric group on $r$ letters where $W=\fS_r=\fS_{\{1,2,\ldots,r\}}$ and $S=\{s_k\mid 1\leq k<r\}$ is the set of basic transpositions $s_k=(k,k+1),$ and let
$\ell:W\to\mathbb{N}$ be the length function with respect to $S$.

An $N$-tuple
$\lambda=(\lambda_1,\lambda_2,\cdots,\lambda_N)$ of non-negative integers is
called a composition of $r$ into $N$ parts if
$|\la|:=\sum_i\la_i=r$. Let $\La(N,r)$ denote the set of compositions of $r$ into $N$-parts.

The parabolic (or standard Young)
subgroup $W_{\lambda}$ of $W$ associated with a composition $\la$
consists of the permutations of $\{1,2,\cdots,r\}$ which leave
invariant the following sets of integers
$$R^\la_i=\{\tilde\lambda_{i-1}+1,\tilde\lambda_{i-1}+2,\ldots,\tilde\lambda_{i-1}+\lambda_i\}\quad(1\leq i\leq N).$$
Here $\tilde\la_0=0$ and $\tilde\la_i=\sum_{j=1}^i\la_j$ so that $\tilde\la$ is the partial sum sequence associated with $\la$.

We will also denote by $\sD_\la:=\mathcal{D}_{W_\la}$ the set of all
distinguished (or shortest) coset representatives of the right
cosets of $W_\la$ in $W$. 
Let
$\mathcal{D}_{\rho\la}=\mathcal{D}_\rho\cap\mathcal{D}^{-1}_{\la}$, where $\rho\in\La(N,r)$.
Then $\mathcal{D}_{\rho\la}$ is the set of distinguished
$W_\rho$-$W_\la$ double coset representatives. For
$d\in\mathcal{D}_{\rho\la}$, the subgroup $W_\rho^d\cap
W_\la=d^{-1}W_\rho d\cap W_\la$ is a parabolic subgroup associated
with a composition which is denoted by $\rho d\cap\la$. In other
words, we define
\begin{equation}\label{ladmu}
W_{\rho d\cap\la}=W_\rho^d\cap W_\la.
\end{equation}
The composition $\rho d\cap\la$ can be easily described in terms of the following matrix.
For $\rho,\la\in\Lambda(N,r)$, $d\in
\mathcal{D}_{\rho\la}$, let
\begin{equation}\label{jmath}
A=(a_{i,j})=\jmath(\rho,d,\la),\qquad\text{where }a_{i,j}=|R^\rho_i\cap d(R^\la_j)|,
\end{equation}
be the $N\times N$ matrix associated to the
 double coset $W_\rho d W_\la$. Then
 \begin{equation}\label{ladmu}
 \rho d\cap\la=(\nu^1,\nu^2,\ldots,\nu^N),
  \end{equation}
where $\nu^j=(a_{1,j},a_{2,j},\ldots,a_{N,j})$ is the $j$th column of $A$.
In this way, the matrix set
$$M(N,r)=\{\jmath(\rho,d,\la)\mid\rho,\la\in\La(N,r),d\in\sD_{\rho\la}\}$$
is the set of all $N\times N$ matrices over $\mathbb N$ whose entries sum to $r$.
For $A\in M(N,r)$, let
 $$\ro(A):=(\sum_{j=1}^Na_{1,j},\ldots,\sum_{j=1}^Na_{N,j})=\rho\,\text{ and }\,\co(A):=(\sum_{i=1}^Na_{i,1},\ldots,\sum_{i=1}^Na_{i,N})=\la.$$
We also let row$_j(A)$ and col$_j(A)$ denote the $j$th row and $j$th column of $A$, respectively.

For a composition $\la=(\la_1,\ldots,\la_{m+n})\in\La(m+n,r)$, we often rewrite
 $$\la=(\lambda^{(0)}|\lambda^{(1)})=(\lambda^{(0)}_1,\lambda^{(0)}_2,\cdots,\lambda^{(0)}_m|\lambda^{(1)}_1,
\lambda^{(1)}_2,\cdots,\lambda^{(1)}_n)$$
to indicate the``even''
and ``odd'' parts of $\la$ and identify $\Lambda(m+n,r)$ with the set
\begin{equation*}
\begin{aligned}
\Lambda(m|n,r)&=\{\lambda=(\lambda^{(0)}|\lambda^{(1)}) \mid\la\in\La(m+n,r)\}\\
&=\bigcup_{r_1+r_2=r}(\La(m,r_1)\times\La(n,r_2))\\
\end{aligned}
\end{equation*}
and identify the set $\La(m+n,1)$ with the set of standard basis
\begin{equation}\label{bse}
\{\bse_1,\bse_2,\ldots,\bse_{m+n}\}.
\end{equation}
Thus, a parabolic subgroup $W_\la$ associated with  $\lambda=(\lambda^{(0)}|\lambda^{(1)})$ has even part
$W_{\la^{(0)}}$ and odd part $W_{\la^{(1)}}$.

For $d\in\mathcal{D}_{\rho\la}$ with $\rho,\la\in\La(m|n,r)$,  the parabolic subgroup $W_{\rho d\cap \la}$  decomposes into four parabolic subgroups
 \begin{equation}\label{ladmu00-11}
W_{\rho d\cap\la}=W_{\rho d\cap\la}^{00} \times W_{\rho
d\cap\la}^{11}\times W_{\rho d\cap\la}^{01}\times W_{\rho
d\cap\la}^{10},
\end{equation}
where $W_{\rho d\cap\la}^{ij}=W_{\rho^{(i)}}^d\cap W_{\la^{(j)}}$.
In this case, the composition $\nu=\rho d\cap \la$ has the form
$\nu=(\nu^1,\ldots,\nu^{m+n})$ with $\nu^i\in\La(m+n,\la_i)$.

We say that $d$ satisfies the {\it even-odd trivial intersection
property} if $W_{\rho d\cap\la}^{01}=1=W_{\rho d\cap\la}^{10}$. For
$\rho,\la\in\Lambda(m|n,r)$, let
\begin{equation}\label{Dcirc}
\mathcal{D}^\circ_{\rho\la}=\{d\in\mathcal{D}_{\rho\la}\mid
W^d_{\rho^{(0)}}\cap W_{\la^{(1)}}=1,W^d_{\rho^{(1)}}\cap
W_{\la^{(0)}}=1\}.
\end{equation}
This set is the super version of the usual $\sD_{\rho\la}$.
Let
\begin{equation}\label{M(m|n)}
\aligned
M(m|n,r)&=\{\jmath(\rho,d,\la)\mid\rho,\la\in\La(m|n,r),d\in\sD_{\rho\la}^\circ\}\text{ and }\\
M(m|n)&=\bigcup_{r\geq0}M(m|n,r).\endaligned
\end{equation}
Write a matrix $A\in M(m|n)$ in the form $\biggl(\begin{matrix}A_{11}&A_{12}\\A_{21}&A_{22}\end{matrix}\biggr)$, where $A_{11}$ is $m\times m$, and write $|A|=r$ if $A\in M(m|n,r)$. Then the even-odd trivial intersection property implies that the entries in $A_{12}$ and $A_{21}$ are either 0 or 1. In fact, $M(m|n)$ is the set of all $(m+n)\times(m+n)$ matrices over $\mathbb N$ satisfying this property. We also define the parity
$\hat A\in\{0,1\}$ of $A$ by setting
\begin{equation}\label{hatA}
\hat A\equiv |A_{12}|+|A_{21}|\;(\text{mod}\; 2)
\end{equation}
and denote the transpose of $A$ by $A'$. Clearly, $\hat A=\hat A'$.

For
$d\in\mathcal{D}^\circ_{\rho\la}$, if we put
$W_{\nu^{(0)}}=W^d_{\rho^{(0)}}\cap
W_{\la^{(0)}},W_{\nu^{(1)}}=W^d_{\rho^{(1)}}\cap
W_{\la^{(1)}},\nu=(\nu^{(0)}|\nu^{(1)})$, then
$W_\nu=W_{\nu^{(0)}}\times W_{\nu^{(1)}}$. We have, in general,
$\nu\in\La(m'|n',r)$ where $m'=m(m+n)$ and $n'=n(m+n)$.

The Hecke algebra $\mathcal{H}_\bsq(r)=\sH_\bsq(W)$ corresponding to
$W=\fS_r$ is a free $\sA$-module with basis $\{T_w; w\in
W\}$ and the multiplication is defined by the rules: for $s\in S$,
\begin{equation}
{T}_w{T}_s=\left\{\begin{aligned} &{T}_{ws},
&\mbox{if } \ell(ws)>\ell(w);\\
&(\bsq-1)T_w+\bsq{T}_{ws}, &\mbox{otherwise}.
\end{aligned}
\right.
\end{equation}
Let $\sH_\up(r)=\sH_\bsq(r)\otimes_\sA\sZ$.


Let $V(m|n)$ be a free $\sA$-module of rank $m+n$
with basis $v_1,v_2,\cdots,v_{m+n}$. Its tensor product $V(m|n)^{\otimes r}$ has basis
$\{v_\bsi\}_{\bsi\in I(m|n,r)}$ where
$$I(m|n,r)=\{\bsi=(i_1,i_2,\cdots,i_r)\mid 1\leq
i_j\leq m+n,\forall j\}$$
and, for $\bsi=(i_1,i_2,\cdots,i_r)\in I(m| n,r)$, let
$$v_\bsi=v_{i_1}\otimes v_{i_2}\otimes\cdots \otimes v_{i_r}=v_{i_1}v_{i_2}\cdots v_{i_r}.$$
The symmetric group $W$ acts on $I(m|n, r)$ by place
permutation: for $w\in W, \bsi\in I(m|n,r)$
$$\bsi w=(i_{w(1)},i_{w(2)},\cdots,i_{w(r)}).$$
Thus, following \cite{M}, $V(m|n)^{\otimes r}$ is a right
$\sH_\bsq(r)$-module with the following action:\footnote{Replacing the basis
$\{v_\bsi\}_{\bsi\in I(m|n,r)}$ by the basis $\{\up^{-|\bsi|}v_\bsi\}_{\bsi\in I(m|n,r)}$, where $|\bsi|$ is the number of inversions in $\bsi$,
the action of $\up^{-1}T_{s_k}$ on the new basis coincides with the action given in \cite[(2.6.1)]{DGW}.}
\begin{equation}\label{T action}
v_\bsi T_{s_k}=
\begin{cases}
{(-1)}^{\widehat{i_k}\widehat{i_{k+1}}}v_{\bsi s_k}, &\mbox{ if } i_k<i_{k+1};\\
\bsq v_\bsi, &\mbox{ if }1\leq i_k=i_{k+1}\leq m;\\
-v_\bsi, &\mbox{ if }m+1\leq i_k=i_{k+1};\\
(-1)^{\widehat{i_k}\widehat{i_{k+1}}}\bsq v_{\bsi s_k}+(\bsq-1)v_\bsi,&
\mbox{ if } i_k>i_{k+1}
\end{cases}
\end{equation}
Note that when $n=0$, this action coincides with the action on the
usual tensor space as given in \cite[(9.1.1)]{DDPW}.
Clearly, we have the following decomposition into $\sH_\bsq(r)$-modules:
$$V(m|n)^{\otimes r}=\bigoplus_{\la\in\La(m|n,r)}v_\la\sH_\bsq(r).$$

\begin{definition}\label{S(m|n,r)}  The algebra
$$\sS_\bsq(m|n,r):=
\End_{\mathcal{H}_\bsq(r)}(V(m|n)^{\otimes r})$$ is called a $\bsq$-{\it Schur
superalgebra} over $\sA$ with a $\mathbb Z_2$-grading
$$\sS_\bsq(m|n,r)_i=\bigoplus_{|\la^{(1)}|+|\mu^{(1)}|\equiv i(\text{mod}2)}\Hom_{\sH_\bsq(r)}(v_\la\sH_\bsq(r),v_\mu\sH_\bsq(r)).$$ Let  $\sS_\up(m|n,r)=\sS_\bsq(m|n,r)\otimes\sZ$ and
$\sS_R(m|n,r)=\sS_\bsq(m|n,r)\otimes R$ for any commutative ring $R$ which is an $\sA$-module.
\end{definition}

Note that $\sE=\End_{\sA}(V(m|n)^{\otimes r})$ is an $\sH_\bsq(r)$-$\Hq$-bimodule and $\Sq$
is the set of all $\Hq$-fixed points:
$$Z_\sE(\Hq)=\{x\in \sE\mid hx=xh,\,\forall h\in\Hq\}.$$
We will use a right-handed function notation for the elements $e$ in $\sE$: $v\mapsto (v)e$ for all
$v\in V(m|n)^{\otimes r}$. This notation is convenient for the $\sH_\bsq(r)$-$\Hq$-bimodule structure on $\sE$.\footnote{However,
we will turn it to a left action from \S6 onwards via the anti-involution $\tau$ below.}

For $\bsi,\bsj\in I(m| n,r)$, we define $e_{\bsi,\bsj}\in
\End_\sA(V(m| n)^{\otimes r})$ to be the $\sA$-linear map
\begin{equation}
(v_{\bsi'})e_{\bsi,\bsj}=\left\{
\begin{aligned}
&v_\bsj,\mbox{ if }\bsi'=\bsi,\\
&0,\mbox{otherwise}.
\end{aligned}
\right.
\end{equation}

For each $\bsi\in I(m|n,r)$, define $\lambda\in \Lambda (m|n,r)$ to
be the weight $wt(\bsi)$ of $v_\bsi$ by setting $\lambda_k=\#\{j\mid
i_j=k,1\leq j\leq
r\}$.
For each $\lambda\in \Lambda(m|n,r)$, define $\bsi_\lambda\in I(m+n,r)$ by
$$
\bsi_\lambda=(\underbrace{1,\ldots,1}_{\la_1},\underbrace{2,\ldots,2}_{\la_2},\ldots,\underbrace{m+n,\ldots,m+n}_{\la_{m+n}})=(1^{\la_1},2^{\la_2},\ldots,(m+n)^{\la_{m+n}}).$$
Thus, for every $\bsi$, there is a unique $d\in\sD_\la$ such that $\bsi=\bsi_\la d$ ($\la=wt(\bsi)$).

For $\rho,\la\in\Lambda(m|n,r)$ and
$d\in\mathcal{D}^\circ_{\rho\la}$, let $A=\jmath(\rho,d,\la)$ be as defined in \eqref{jmath} and define the relative norm
\begin{equation}\label{norm basis element}
N_{A'}=N_{\la\rho}^d=N_{W,W_\rho^d\cap W_\la}(e_{\la,\rho d}):=\sum_{w\in\mathcal{D}_{\rho
d\cap\la}}\bsq^{-\ell(w)}T_{w^{-1}}e_{\la,\rho d}T_w\in\sE,
\end{equation}
where $e_{\la,\rho d}=e_{\bsi_\la,\bsi_\rho d}$ and $A'=\jmath(\la,d^{-1},\rho)$ is the transpose of $A$.
Clearly, the grading degree of $N_{A'}$ is $\hat A$.
The first assertion of the following is given in {\cite[5.1]{DGW}}.
\begin{lemma}\label{norm basis} (1) The set $\{N^d_{\la\rho}\mid
\rho,\la\in\Lambda(m|n, r), d\in\mathcal{D}^\circ_{\rho\la}\}$
forms an $\sA$-basis of homogeneous elements for $\sS_\bsq(m|n,r )$. Moreover, $N_{A'}N_\la=\delta_{\la,\ro(A)}N_{A'}$ and
$N_\la N_{A'}=\delta_{\la,\co(A)}N_{A'}$,
where $N_\la=N^1_{\la\la}$.

(2) Matrix transposing induces a superalgebra anti-automorphism
$$\tau:\sS_\bsq(m|n,r)\longrightarrow\sS_\bsq(m|n,r),\quad N_A\longmapsto N_{A'}.$$
\end{lemma}

Sign manipulating is an important issue in the super theory. The number of ``odd inversions'' in $\bsi_\la d$ will be frequently encountered later on.

\begin{definition}\label{dhat}
For $\lambda\in\Lambda(m|n,r), d\in \mathcal{D}_\lambda$, if
$\bsi_\lambda d=(i_1,i_2,\cdots,i_r)$, define a number in $\mathbb N$ (not in $\mathbb Z_2$)
$$(\la,d)^\wedge =\sum_{k=1}^{r-1}\sum_{k<l,i_k>i_l}\widehat{i_k}\widehat{i_l}.$$
When $\la$ is clear from the context, we write
$\whd=(\la,d)^\wedge$. \end{definition}
Note that $(-1)^{\whd}(-1)^{\widehat{i_k}\widehat{i_{k+1}}}=(-1)^{\widehat{ds_k}}\text{
for all } d\in\mathcal{D}_\lambda, s_k\in S\text{ with
}ds_k\in\mathcal{D}_\lambda.$

\section{Some Algorithms}

Recall from \eqref{jmath} the map $\jmath$ taking a double coset $W_\rho dW_\la$ to a matrix
$\jmath(\rho,d,\la)$. We now look at the inverse map. Given a matrix $A\in M(N,r)$, we want to construct $d\in\sD_{\rho\la}$, where $\rho=\ro(A)$ and $\la=\co(A)$ such that $A=\jmath(\rho,d,\la)$. In \cite{Du96}, there is an algorithm to compute $d$ as a permutation function. For the proof of the key lemma in next section, we develop in this section another algorithm to compute $d$ as a reduced product of basic transpositions. Moreover, we will modify this algorithm to compute the action of $T_d$ on the tensor vector $v_\rho=v_{\bsi_\rho}$ and on its variations.

Given a matrix $A=(a_{i,j})\in M(N,r)$ with $\rho:=\ro(A)$ and $\la:=\co(A)$, $d\in\sD_{\rho\la}$ is the element satisfying
$$\aligned
 \bsi_\rho d&=(1^{\rho_1},2^{\rho_2},\ldots,N^{\rho_{N}})d\\
 &=(1^{a_{1,1}},2^{a_{2,1}},\ldots,N^{a_{N,1}},1^{a_{1,2}},2^{a_{2,2}},\ldots,N^{a_{N,2}},\ldots,1^{a_{1,N}},2^{a_{2,N}},\ldots,N^{a_{N,N}}).\\
 \endaligned$$

We may contruct $d$ by the following algorithm.
As before, for a sequence $\mu=(\mu_1,\ldots,\mu_N)$, let $\tilde\mu$ be the associated partial sum sequence with $\tilde \mu_i=\mu_1+\cdots+\mu_i$. Similarly, let $\tilde a_{i,j}=a_{1,j}+\cdots+a_{i,j}$ and write $s_{(i_1,i_2,\ldots,i_b)}=s_{i_1}s_{i_2}\cdots s_{i_b}$.

\begin{algorithm}[Computing $d$ from $A$] \label{Algo1} Let $A=(a_{i,j})\in M(N,r)$.

 {\tt
 \noindent
 Set $\rho:=\ro(A)$, $\la:=\co(A)$, $d=1$,  $\col_0(A)=\bf 0$, $\theta=\bf0$, $\la_0=0$;

 \noindent
for $j$ from 1 to $N-1$, do  $w_j:=1$, $\rho:=\rho-\col_{j-1}(A)$, $\theta:=\theta+(\la_{j-1},\ldots,\la_{j-1})$;

         for $i$ from 2 to $N$, do
         $$\aligned w_{ij}&:= s_{\theta+(\tilde\rho_{i-1},\tilde\rho_{i-1}-1,\ldots, \tilde a_{i-1,j}+1)}s_{\theta+(\tilde\rho_{i-1}+1,\tilde\rho_{i-1},\ldots,\tilde a_{i-1,j}+2)} \cdots
 s_{\theta+(\tilde\rho_{i-1}+a_{i,j}-1,\tilde\rho_{i-1}+a_{i,j}-2,\ldots,\tilde a_{i,j})}, \\
 w_j&:=w_j*w_{ij};\endaligned$$

         end do  $d:=d*w_j$; end do;  \qquad OUTPUT($d$);
         }

Here, the inner loops build the subsequence $1^{a_{1,j}},2^{a_{2,j}},\ldots,N^{a_{N,j}}$, while the outer loops
 indicate the construction is done from column 1 to column $N-1$ of $A$.\footnote{Once the construction is done for column $N-1$, that for column $N$ is automatic.} The element $w_{ij}$
   moves the block $i^{a_{i,j}}$ to the left of  the sequence
$1^{\rho_1-\sum_{l=1}^j a_{1,l}},2^{\rho_2-\sum_{l=1}^j a_{2,l}},\ldots, \linebreak(i-1)^{\rho_{i-1}-\sum_{l=1}^j a_{i-1,l}}$.
         \end{algorithm}

By the Algorithm, we see that $w_j$ is a product of $N-1$ permutations $w_{2,j},\ldots, w_{N,j}$ which move $2^{a_{2,j}}$ leftwards by $w_{2,j}$ to form $1^{a_{1,j}},2^{a_{2,j}}$, etc., and move $N^{a_{N,j}}$ leftwards by $w_{N,j}$ to form $1^{a_{1,j}},2^{a_{2,j}},\ldots,N^{a_{N,j}}$. For example,
$$\aligned
 w_{2,1}&=s_{(\tilde\rho_1,\tilde\rho_1-1,\ldots,\tilde a_{1,1}+1)}s_{(\tilde\rho_1+1,\tilde\rho_1,\ldots,\tilde a_{1,1}+2)} \cdots
 s_{(\tilde\rho_1+a_{2,1}-1,\tilde\rho_1+a_{2,1}-2,\ldots,\tilde a_{2,1})}\cdot \\
w_{3,1} &=s_{(\tilde\rho_2,\tilde\rho_2-1,\ldots,\tilde a_{2,1}+1)}s_{(\tilde\rho_2+1,\tilde\rho_2,\ldots, \tilde a_{2,1}+2)} \cdots
 s_{(\tilde\rho_2+a_{3,1}-1,\tilde\rho_2+a_{3,1}-2,\ldots,\tilde a_{3,1})}\cdot \\
 &\quad\,\cdots\\
w_{N,1} &= s_{(\tilde\rho_{N-1},\tilde\rho_{N-1}-1,\ldots, \tilde a_{N-1,1}+1)} \cdots
 s_{(\tilde\rho_{N-1}+a_{N,1}-1,\tilde\rho_{N-1}+a_{N,1}-2,\ldots,\tilde a_{N,1})}. \\ \endaligned $$
Moreover, we have
 $$
\aligned &\ell(w_1)=(\tilde\rho_1-\tilde a_{1,1})a_{2,1}+(\tilde \rho_2-\tilde a_{2,1})a_{3,1}+\cdots+
( \tilde\rho_{N-1}-\tilde a_{N-1,1})a_{N,1}\\
&\ell(w_2)=(\tilde\rho_1-\tilde a_{1,1}-\tilde a_{1,2})a_{2,2}+(\tilde \rho_2-\tilde a_{2,1}-\tilde a_{2,2})a_{3,2}+\cdots+
( \tilde\rho_{N-1}-\tilde a_{N-1,1}-\tilde a_{N-1,2})a_{N,2}\\
&\cdots \cdots\\
&\ell(w_{N-1})=a_{1,N}a_{2,N-1}+(a_{1,N}+a_{2,N})a_{3,N-1}+\cdots+(a_{1,N}+a_{2,N}+\cdots+a_{N-1,N})a_{N,N-1}.
\endaligned$$
Hence, we obtain from the algorithm
$$
 d=w_1w_2\cdots w_{N-1}\;\,\text{ and }\,\;\ell(d)=\sum_{i<k,j>l}a_{i,j}a_{k,l}.$$

We now use the Hecke algebra action on $V(m|n)^{\otimes r}$ defined in \eqref{T action} to compute $v_\rho\cdot T_d$ by modifying Algorithm \ref{Algo1}.

 For $\lambda,\rho\in\Lambda(m|n,r)$, $d\in
\mathcal{D}^\circ_{\rho\lambda}$, set $A=(a_{i,j})\in M(m|n,r)$ to
be the matrix $\jmath(\rho,d,\la)$ associated to the
 double coset $W_\rho d W_\lambda$ and let $\nu=\rho d\cap\la$.

 Since only the $i_k<i_{k+1}$ case occurs when computing $v_\rho\cdot T_d$, by applying the first formula in \eqref{T action} repeatedly, the algorithm above continues to hold with the sign recorded. Hence, we have
 $v_\rho\cdot T_d=(-1)^{\hat d}v_{\bsi_\nu}=(-1)^{\hat d}v^A$, where
 $$v^A=v_1^{a_{1,1}}v_2^{a_{2,1}}\cdots v_N^{a_{N,1}}v_1^{a_{1,2}}v_2^{a_{2,2}}\cdots v_N^{a_{N,2}} \cdots
 v_1^{a_{1,N}}v_2^{a_{2,N}}\cdots v_N^{a_{N,N}}\;(N=m+n).$$
 More precisely, putting $d_i=w_{i+1}\cdots w_{N-1}$,
 $$\aligned
 v_\rho T_d&=(-1)^{\widehat{w_1}}v_1^{a_{1,1}}v_2^{a_{2,1}}\cdots v_N^{a_{N,1}}v_1^{\rho_1-a_{1,1}}v_2^{\rho_2-a_{2,1}}\cdots v_N^{\rho_N-a_{N,1}} \cdot T_{d_1}\\
 &=(-1)^{\widehat{w_1w_2}}v_1^{a_{1,1}}\cdots v_N^{a_{N,1}}v_1^{a_{1,2}}\cdots v_N^{a_{N,2}}
 v_1^{\rho_1-\tilde a_{1,2}}v_2^{\rho_2-\tilde a_{2,2}}\cdots v_N^{\rho_N-\tilde a_{N,2}} \cdot T_{d_2} \\
 &\quad\,\cdots\\
 &=(-1)^{\widehat d}v_1^{a_{1,1}}v_2^{a_{2,1}}\cdots v_N^{a_{N,1}}v_1^{a_{1,2}}v_2^{a_{2,2}}\cdots v_N^{a_{N,2}} \cdots
 v_1^{a_{1,N}}v_2^{a_{2,N}}\cdots v_N^{a_{N,N}},\\\endaligned $$
where $\tilde a_{i,j}=a_{i,1}+\cdots+a_{i,j}$.
In other words, to compute $v_\rho\cdot T_d$, we use the following algorithm:

\vspace{.3cm}

 \noindent
 \begin{algorithm}[Computing $v_1^{\rho_1}v_2^{\rho_2}\cdots v_{m+n}^{\rho_{m+n}}\cdot T_d$] \label{Algo2} Let $A=(a_{i,j})$ be the matrix associated with $(\rho,d,\la)$.

 {\tt
 \noindent
 for $j$ from 1 to $m+n$, do

         for $i$ from 2 to $m+n$,
         move the factor $v_i^{a_{i,j}}$ to the left of  the block\linebreak $v_1^{\rho_1-\sum_{l=1}^j a_{1,l}}\cdots v_{i-1}^{\rho_{i-1}-\sum_{l=1}^j a_{i-1,l}}$ to form $v_1^{a_{i,1}}\cdots v_{i-1}^{a_{i-1,j}}v_i^{a_{i,j}}$
         end move; end do;}
 \end{algorithm}

\begin{lemma}\label{sign} In the algorithm of computing $v_\rho\cdot T_d$, moving $v_i^{a_{i,j}}$ for $i>1$ to the left of  the block $v_1^{\rho_1-\sum_{l=1}^j a_{1,l}}\cdots v_{i-1}^{\rho_{i-1}-\sum_{l=1}^j a_{i-1,l}}$
 contributes the sign
$$(-1)^{\sum_{k=1}^{i-1}
a_{i,j}(\rho_k-\sum_{l=1}^j a_{k,l})\hat i\hat k}$$
Hence, $v_\rho\cdot T_d=(-1)^{\hat d}v_\rho$ and
$$\hat d=\sum_{j=1}^{m+n-1}\sum_{i=2}^{m+n}(\sum_{k=1}^{i-1}
a_{i,j}(\rho_k-\sum_{l=1}^j a_{k,l})\hat i\hat k).$$
\end{lemma}

For $1\leq k\leq m+n-1$ and $0\leq t\leq a_{h+1,k}-1$, let
\begin{equation}\label{fw}
\aligned
\fw_{h+1}^{k,t}&=v_{h+1}^{\tilde a_{h+1,k-1}+t}v_h
v^{\rho_{h+1}-\tilde a_{h+1,k-1}-t-1}_{h+1}\\
&=v_{h+1}^{a_{h+1,1}}\cdots v_{h+1}^{a_{h+1,k-1}}b_{h+1,k}^{(t)}v_{h+1}^{a_{h+1,k+1}}\cdots v_{h+1}^{a_{h+1,m+n}},
\endaligned
\end{equation}
where $b_{h+1,k}^{(t)}=v_{h+1}^tv_h
{v_{h+1}}^{a_{h+1,k}-t-1}$ and $\tilde a_{i,j}=a_{i,1}+\cdots+a_{i,j}$.
Applying Algorithm \ref{Algo2} to the element
$$v_\rho(\fw_{h+1}^{k,t})=v^{\rho_1}_1 \cdots v^{\rho_h}_h \fw_{h+1}^{k,t}v^{\rho_{h+2}}_{h+2} \cdots
v^{\rho_{m+n}}_{m+n}$$
yields the following.

\begin{lemma}\label{sign2}
Maintain the notation above and assume $a_{h+1,k}\geq1$. We have
$$v_\rho(\fw_{h+1}^{k,t})\cdot T_d=\!\begin{cases}(-1)^{\hat d}\bsq^{\sum_{j>k}a_{h,j}} v^A(b_{h+1,k}^{(t)}),&h<m;\\
(-1)^{\hat d}(-1)^{\sum_{i> m+1,j<k}a_{i,j}}\bsq^{\sum_{j>k}a_{m,j}}v^A(b_{m+1,k}^{(t)}),&h=m;\\
(-1)^{\hat d}v^A(b_{h+1,k}^{(t)}),&h>m,\end{cases}$$
where $v^A(b_{h+1,k}^{(t)})$ is the element obtained from $v^A$ by replacing the factor $v_{h+1}^{a_{h+1,k}}$ by $b_{h+1,k}^{(t)}$
for $0\leq t\leq a_{h+1,k}-1$.
\end{lemma}
\begin{proof} We first observe that the sign contribution occurs only when moving $v_i^{a_{i,j}}$  to the left
of $v_{i'}^{\rho_{i'}-\sum_{l=1}^{j-1}a_{i',l}}$, where $m<i'<i$. We now have three cases to consider.

If $h<m$, then it is clear that the sign is the same as the sign $(-1)^{\hat d}$ when computing $v_\rho\cdot T_d$. In this case, moving the factor $b_{h+1,k}^{(t)}$ to the left of  the block $v_1^{\rho_1-\sum_{l=1}^k a_{1,l}}\cdots v_h^{\rho_{h}-\sum_{l=1}^k a_{h,l}}$ requires applying the second formula in \eqref{T action} $\sum_{j>k}a_{h,j}=\linebreak\rho_{h}-\sum_{l=1}^k a_{h,l}$ times. This gives the power $\bsq^{\sum_{j>k}a_{h,j}}$.

If $h=m$, then one factor of $v_{m+1}^{a_{m+1,k}}$ is replaced by $v_m$. Thus, for $i>m+1$
and $j\leq k-1$, moving $v_i^{a_{i,j}}$ to the left of $v_{m+1}^{\tilde a_{m+1,k-1}-\tilde a_{m+1,j}+t}v_m
v^{\rho_{m+1}-\tilde a_{m+1,k-1}-t-1}_{m+1}$ loses the sign contribution $(-1)^{a_{i,j}}$.
The power of $\bsq$ is produced similarly by taking $h=m$ in the argument above.

Finally, assume $h>m$. In this case, moving $v_i^{a_{i,j}}$ ($i>h+1$ and $j<k$) to the left
of $v_{h+1}^{\tilde a_{h+1,k-1}-\tilde a_{h+1,j}+t}v_h
v^{\rho_{h+1}-\tilde a_{h+1,k-1}-t-1}_{h+1}$ contributes the same sign as in computing $v_\rho\cdot T_d$, while moving $b_{h+1,k}^{(t)}$ to the left of  the block $v_1^{\rho_1-\sum_{l=1}^k a_{1,l}}\cdots v_h^{\rho_{h}-\sum_{l=1}^k a_{h,l}}$ requires applying the third formula in \eqref{T action} $\rho_{h}-\sum_{l=1}^k a_{h,l}=\sum_{j>k}a_{h,j}$ times, resulting again in the same sign.
\end{proof}

\section{Super version of two key multiplication formulas}
The realisation of quantum $\mathfrak{gl}_n$ requires several important multiplication formulas between certain standard basis elements and between certain elements of a spanning set---the BLM spanning set---for the $\up$-Schur algebra. These formulas were built on two key formulas, see  \cite[Lem~3.2]{BLM}, whose proof used the geometry of (partial) flag varieties.\footnote{The affine analogue of the two key formulas has also been discovered by Lusztig (see
\cite[Lem.~3.5]{Lu}) and the proof is also geometric in nature.}  In this section, we will use an algebraic method to derive the two corresponding formulas in the super case. The algorithm discussed in \S2 is helpful to the success of the method.

 Let $\rho,\la\in\La(m|n,r)$, $d\in\sD^\circ_{\rho\la}$, and $A=\jmath(\rho,d,\la)$.
 Fix $1\leq h<m+n$, and let
 $$\mu=(\rho_1,\rho_2,\cdots,\rho_h+1,\rho_{h+1}-1,\cdots,\rho_{m+n})=\rho+\bse_h-\bse_{h+1}.$$
 It is easy to show $\mathcal{D}_{\rho\mu}=\{1\}$. Set $B=\jmath(\rho,1,\mu)$ to be the
 matrix associated to the double coset $W_\rho 1W_\mu$.

Notice also that, if $W_\nu=W_\rho^d\cap
W_\lambda$, then $\nu=(\nu^1,\nu^2,\cdots,\nu^{m+n})$ where
$\nu^i=\col_i(A)=(a_{1,i},a_{2,i},\cdots,a_{m+n,i})\in\Lambda(m|n,\lambda_i)$.
If we set $W_{\rho'}=W_\mu\cap W_\rho$, then
$\rho'=(\rho_1,\cdots,\rho_h,1,\rho_{h+1}-1,\rho_{h+2},\cdots,\rho_{m+n})$.
Observing the compositions $\rho$ and $\rho'$, we have
$$\mathcal{D}_{\rho'}\cap
W_\rho=\{1,s_{\rho_1+\cdots+\rho_h+1},s_{\rho_1+\cdots+\rho_h+1}s_{\rho_1+\cdots+\rho_h+2},\cdots,s_{\rho_1+\cdots+\rho_h+1}\cdots
s_{\rho_1+\cdots+\rho_{h+1}-1}\}.$$

For $A=(a_{i,j})\in M(m|n,r)$ and $h,k\in[1,m+n]$ with $h<m+n$, define
 \begin{equation}\label{fqAh}
f_k(\bsq;A,h)=\begin{cases}\bsq^{\sum_{j>k}a_{h,j}},&\text{if }h<m;\\
(-1)^{\sum_{i>m,j<k}a_{i,j}}\bsq^{-\sum_{j<k}a_{m+1  ,j}}\bsq^{\sum_{j>k}a_{m,j}},&\text{if }h=m;\\
(\bsq^{-1})^{\sum_{j<k}a_{h+1,j}},&\text{if }h>m,\end{cases}
\end{equation}
and
 \begin{equation}\label{fqAh+1}
g_k(\bsq;A,h+1)=\begin{cases}\bsq^{\sum_{j<k}a_{h+1,j}},&\text{if }h<m;\\
(-1)^{\sum_{i>m,j<k}a_{i,j}},&\text{if }h=m;\\
(\bsq^{-1})^{\sum_{j>k}a_{h,j}},&\text{if }h>m.\end{cases}
\end{equation}
Here is the key lemma of the paper.
\begin{lemma}\label{main lemma}Maintain the notation above. For $A=(a_{i,j})\in M(m|n,r)$ and $h\in[1,m+n)$,
let $D=\diag(\ro(A))$,
$B=D+E_{h,h+1}-E_{h+1,h+1}$ and $C=D-E_{h,h}+E_{h+1,h}$. Then, in the $\bsq$-Schur superalgebra
$\sS_\bsq(m|n,r)$ over $\sA$, the following multiplication formulas
hold (and are homogeneous):
$$\aligned
(1)\qquad N_{A'}N_{B'}&=\sum_{k\in[1,m+n],a_{h+1,k}\geq
1}f_k(\bsq;A,h)[\![a_{h,k}+1]\!]_{\bsq_h}
N_{(A+E_{h,k}-E_{h+1,k})'},\\
(2)\qquad N_{A'}N_{C'}&=\sum_{k\in[1,m+n],a_{h,k}\geq
1}g_k(\bsq;A,h+1)[\![a_{h+1,k}+1]\!]_{\bsq_{h+1}}
N_{(A-E_{h,k}+E_{h+1,k})'}.\endaligned
$$
where $N_{(A\pm E_{h,k}\mp E_{h+1,k})'}=0$ whenever $A\pm E_{h,k}\mp E_{h+1,k}\not\in M(m|n,r)$.
\end{lemma}

\begin{proof} We first prove (1).
Since $N_{A'}=N_{\la\rho}^d,N_{B'}=N_{\rho\mu}^1$ and $N_{(A+E_{h,k}-E_{h+1,k})'}=N_{\la\mu}^{d'}$ for some $d'\in\mathcal{D}^\circ_{\mu\lambda}$, it suffices to prove that
the actions of both sides on $v_\la$ is the same.
By \eqref{norm basis element},
\begin{equation*}
\begin{aligned}
(v_\lambda) N^d_{\lambda\rho}N_{\rho\mu}^1&=((v_\lambda) N_{W,W_\rho^d\cap
W_\lambda}(e_{\la,\rho d}))N_{W,W_\rho\cap W_\mu}(e_{\rho,\mu})\\
&=((v_\lambda) \sum_{x\in\mathcal{D}_\nu\cap
W_\lambda}\bsq^{-\ell(x)}T_{x^{-1}}e_{\lambda,\rho
d}T_x )\sum_{w\in\mathcal{D}_{\rho'}}\bsq^{-\ell(w)}T_{w^{-1}}e_{\rho,\mu}T_w.\\
\end{aligned}
\end{equation*}
Write $x=x_0x_1$ with $x_i\in W_{\la^{(i)}}$. By \eqref{T action}, $v_\la\cdot T_x=x^{\ell(x_0)}(-1)^{\ell(x_1)}v_\la$. Thus,
\begin{equation*}
\begin{aligned}
(v_\lambda) N^d_{\lambda\rho}N_{\rho\mu}^1
&=(\sum_{x\in\mathcal{D}_\nu\cap
W_\lambda}(-\bsq^{-1})^{\ell(x_1)}v_{\rho
d}\cdot T_x)\sum_{w\in\mathcal{D}_{\rho'}}\bsq^{-\ell(w)}T_{w^{-1}}e_{\rho,\mu}T_w\\
&=\sum_{x\in\mathcal{D}_\nu\cap
W_\lambda,w\in\mathcal{D}_{\rho'}}(-\bsq^{-1})^{\ell(x_1)}\bsq^{-\ell(w)}(v_{\rho
d}\cdot T_xT_{w^{-1}})e_{\rho,\mu}T_w,\quad\text{by \eqref{T action}}.\\
\end{aligned}
\end{equation*}
Now, $(v_{\rho
d}\cdot T_xT_{w^{-1}})e_{\rho,\mu}\neq0$ if and only if $w=w'dx$ for some $w'\in\sD_{\rho'}\cap W_\rho$. Write $w'=w'_0w'_1$ with $w'_i\in W_{\rho^{(i)}}$. Then, in this case, $\bsq^{-\ell(w)}(v_{\rho
d}\cdot T_xT_{w^{-1}})e_{\rho,\mu}=\bsq^{-\ell(w)}(-1)^{\hat d}(v_{\rho}\cdot T_{dx}T{_{(dx)^{-1}}}T_{w^{\prime-1}})e_{\rho,\mu}=(-1)^{\hat d}\bsq^{-\ell(w')}(v_\rho T_{w^{\prime-1}})e_{\rho,\mu}$.
Hence,
\begin{equation}
\begin{aligned}
(v_\lambda) N^d_{\lambda\rho}N_{\rho\mu}^1&=\sum_{\substack {x\in\mathcal{D}_\nu\cap W_\lambda\\
w'\in\mathcal{D}_{\rho'}\cap
W_\rho}}(-\bsq^{-1})^{\ell(x_1)}(-1)^{\hat{d}}\bsq^{-\ell(w')}\bsq^{\ell(w'_0)}(-1)^{\ell(w'_1)}v_\mu
T_{w'dx}\\
&=\sum_{\substack{ x\in\mathcal{D}_\nu\cap W_\lambda\\
w'\in\mathcal{D}_{\rho'}\cap
W_\rho}}(-1)^{\hat{d}}(-\bsq^{-1})^{\ell(x_1)}(-\bsq^{-1})^{\ell(w'_1)}v_\mu
T_{w'dx}\\
&=(-1)^{\hat{d}}\sum_{x\in\mathcal{D}_\nu\cap W_\lambda
}(-\bsq^{-1})^{\ell(x_1)}
(\Upsilon\cdot T_d)T_x,\label{LHS}
\end{aligned}
\end{equation}
where
$\Upsilon:=\sum_{w'\in\mathcal{D}_{\rho'}\cap
W_\rho}(-\bsq^{-1})^{\ell(w'_1)}v_\mu\cdot
T_{w'}$ and $\hat d=(\rho, d)^\wedge=\sum_{i>k>m,j<l}a_{i,j}a_{k,l}$.
Thus, by observing
$$\mathcal{D}_{\rho'}\cap W_\rho=\begin{cases}\mathcal{D}_{\rho'}\cap W_{\rho^{(0)}},&\text{if }h<m;\\
\mathcal{D}_{\rho'}\cap W_{\rho^{(1)}},&\text{if }h\geq m.\end{cases}$$
we see that
$$\Upsilon=\begin{cases}\sum_{w'\in\mathcal{D}_{\rho'}\cap W_\rho}v_\mu \cdot T_{w'},&\text{if }h<m;\\
\sum _{w'\in\mathcal{D}_{\rho'}\cap W_\rho}(-q^{-1})^{\ell(w')}v_\mu\cdot
T_{w'},&\text{if }h\geq m.\end{cases}$$
Since
$\mathcal{D}_{\rho'}\cap
W_\rho=\{1,s_{\tilde\rho_h+1},s_{\tilde\rho_h+1}s_{\tilde\rho_h+2},\cdots,s_{\tilde\rho_h+1}\cdots
s_{\tilde\rho_h+\rho_{h+1}-1}\},$
it follows from \eqref{T action} that,
$$\Upsilon=\begin{cases}
v^{\rho_1}_1 \cdots v^{\rho_h}_h( \sum_{i=1}^{\rho_{h+1}} v_{h+1}^{i-1}v_hv_{h+1}^{\rho_{h+1}-i})v^{\rho_{h+2}}_{h+2} \cdots v^{\rho_{m+n}}_{m+n},&\text{if }h<m;\\
v^{\rho_1}_1 \cdots v^{\rho_h}_h ( \sum_{i=1}^{\rho_{h+1}}(-\bsq^{-1})^{i-1} v_{h+1}^{i-1}v_hv_{h+1}^{\rho_{h+1}-i}) v^{\rho_{h+2}}_{h+2} \cdots v^{\rho_{m+n}}_{m+n},&\text{if }h= m;\\
v^{\rho_1}_1 \cdots v^{\rho_h}_h ( \sum_{i=1}^{\rho_{h+1}}\bsq^{-(i-1)} v_{h+1}^{i-1}v_hv_{h+1}^{\rho_{h+1}-i}) v^{\rho_{h+2}}_{h+2} \cdots v^{\rho_{m+n}}_{m+n},&\text{if }h> m.\\
\end{cases}
$$
Let 
\begin{equation}\label{vAb}
v^A(b_{h+1,k}^\bullet)=\begin{cases}
\sum_{t=0}^{a_{h+1,k}-1}v^A(b_{h+1,k}^{(t)}),&\text{ if }h<m;\\
\sum_{t=0}^{a_{m+1,k}-1}(-\bsq^{-1})^tv^A(b_{h+1,k}^{(t)}),&\text{ if }h=m;\\
\sum_{t=0}^{a_{h+1,k}-1}\bsq^{-t}v^A(b_{h+1,k}^{(t)}),&\text{ if }h>m;
\end{cases}
\end{equation}
cf. Lemma \ref{sign2}.

If $h<m$, then $ \sum_{i=1}^{\rho_{h+1}} v_{h+1}^{i-1}v_hv_{h+1}^{\rho_{h+1}-i}=\sum_{\substack{k\in[1,m+n]\\a_{h+1,k}\geq1}}\sum_{t=0}^{a_{h+1,k}-1}\fw_{h+1}^{k,t}$, and so
$$\aligned
\Upsilon\cdot T_d&=\sum_{\substack{k\in[1,m+n]\\a_{h+1},k\geq1}}\sum_{t=0}^{a_{h+1,k}-1}
v^{\rho_1}_1 \cdots v^{\rho_h}_h
\fw_{h+1}^{k,t}v^{\rho_{h+2}}_{h+2} \cdots v^{\rho_{m+n}}_{m+n}\cdot T_d,\\
&=(-1)^{\hat{d}}\sum_{\substack{k\in[1,m+n]\\ a_{h+1,k}\geq
1}}\bsq^{\sum_{j>k}a_{hj}}v^A(b_{h+1,k}^\bullet),\;\text{
by Lemma \ref{sign2}.}
\endaligned
$$
If $h\geq m$ and $\star:={\sum_{i> m+1,j<k}a_{i,j}}$, then
$$\aligned
\Upsilon\cdot T_d&=\sum_{\substack{k\in[1,m+n]\\a_{h+1,k}\geq1}}\sum_{t=0}^{a_{h+1,k}-1}v^{\rho_1}_1 \cdots v^{\rho_h}_h
((-1)^{\delta_{h,m}}\bsq^{-1})^{\tilde a_{h+1,k-1}+t}\fw_{h+1}^{k,t}v^{\rho_{h+2}}_{h+2} \cdots v^{\rho_{m+n}}_{m+n}\cdot T_d\\
&=\begin{cases}(-1)^{\hat{d}}
\sum_{\substack{k\in[1,m+n]\\a_{m+1  ,k}\geq
1}}(-1)^\star(-\bsq^{-1})^{\sum_{j<k}a_{m+1  ,j}}\bsq^{\sum_{j>k}a_{m,j}}v^A(b_{m+1,k}^\bullet),&\text{if }h=m\\
(-1)^{\hat{d}}\sum_{k\in[1,m+n],a_{h+1,k}\geq
1}(\bsq^{-1})^{\sum_{j<k}a_{h+1,j}}v^A(b_{h+1,k}^\bullet),&\text{if }h>m.\\\end{cases}
\endaligned$$
Thus, with the notation given in \eqref{fqAh},
the RHS of \eqref{LHS} becomes
$$\sum_{\substack{k\in[1,m+n]\\ a_{h+1,k}\geq
1}}f_k(\bsq;A,h)
\sum_{x\in\mathcal{D}_\nu\cap W_\lambda
}(-q^{-1})^{\ell(x_1)}
v^A(b_{h+1,k}^\bullet)\cdot T_x=\sum_{\substack{k\in[1,m+n]\\ a_{h+1,k}\geq
1}}f_k(\bsq;A,h)\Upsilon'(k;A,h)
,$$
where
$$\Upsilon'=\Upsilon'(k;A,h):=\sum_{x\in\mathcal{D}_\nu\cap W_\lambda
}(-\bsq^{-1})^{\ell(x_1)}
v^A(b_{h+1,k}^\bullet)\cdot T_x.$$

By comparing this with the formulas in the theorem, it  remains to prove that, for all $k\in[1,m+n]$ with $a_{h+1,k}\ge1$,
\begin{equation}\label{last}
\Upsilon'=\begin{cases}
[\![a_{h,k}+1]\!]_{\bsq_h}(v_\la)N_{(A+E_{h,k}-E_{h+1,k})'},&\text{ if }A+E_{h,k}-E_{h+1,k}\in M(m|n,r),\\
0,&\text{otherwise.}
\end{cases}
\end{equation}

We now compute the left hand side of \eqref{last}.
 Recall the composition $\nu=\nu(A)$ defined by the columns of $A$.
 Take a note of the right hand side of \eqref{last}:   for $a_{h+1,k}\geq 1$,
$$\boxed{(v_\la)N_{(A+E_{h,k}-E_{h+1,k})'}=\sum_{x\in\sD_\sigma\cap W_\la}(-\bsq^{-1})^{\ell(x_1)}v_\sigma\cdot T_x,}$$
where
  $$\sigma=\sigma(h,k)=\nu(A+E_{h,k}-E_{h+1,k})=(\sigma^1,\sigma^2,\ldots,\sigma^{m+n}).$$
 Clearly, $\sigma$ is obtained from $\nu$ by replacing
  $\nu^k=(a_{1k},a_{2k},\cdots,a,b,\cdots,a_{m+n,k}),
  $ where $a=a_{h,k}$, $b=a_{h+1,k}$,
   by
  $${\sigma}^{k}=(a_{1k},a_{2k},\cdots,a+1,b-1,\cdots,a_{m+n,k}).$$
We also define $\tau=\tau(h,k)$ to be the composition obtained from $\nu$ by replacing $\nu^k$ by
    $$\tau^{k}=(a_{1,k},a_{2,k},\cdots,a,1,b-1,\cdots,a_{m+n,k}).$$
Note that the relationship between $\nu^k,\sigma^k,$ and $\tau^k$  are similar to the relationship between $\rho,\mu$, and $\rho'$. Note also that $W_{\nu}\cap W_{\sigma}=W_{\tau}$ and so
$W_\tau\leq W_\sigma$ and $W_\tau\leq W_\nu$.
In particular, we have
\begin{equation}\label{MM'}
\begin{cases}(1)& b=1\implies W_\tau=W_\nu;\\
(2)& a=0\implies W_\tau=W_\sigma;\\
(3)& a=1\implies W_\sigma=W_\tau\cup W_\tau s_{\tilde a},
\end{cases}
\end{equation}
where  $$\tilde a=|\nu^1|+\cdots+|\nu^{k-1}|+a_{1,k}+\cdots+a_{h-1,k}+a_{h,k}=\tilde a'+a\,(a'=a_{h-1,k}).$$
Since $\sigma^{k'}=\nu^{k'}=\tau^{k'}$ for all $k'\neq k$, it follows that, for each $\iota\in\{\nu,\sigma,\tau\}$, there is a commuting set decomposition
\begin{equation}\label{DD'}
\mathcal{D}_{\iota}\cap W_{\lambda }=\sD(\mathcal{D}_{\iota^k }\cap
W_{\lambda_k })\sD',
\end{equation}
where
$\mathcal{D}=\mathcal{D}_{\nu^1}\cap
W_{\lambda_1}\cdots\mathcal{D}_{\nu^{k-1} }\cap
W_{\lambda_{k-1} }$ and $\mathcal{D}'=\mathcal{D}_{\nu^{k+1} }\cap
W_{\lambda_{k+1} } \cdots\mathcal{D}_{\nu^{m+n}}\cap
W_{\lambda_{m+n}}.$
Moreover,
\begin{equation}\label{break}
\mathcal{D}_{\tau^k}\cap
W_{\lambda_k}=(\mathcal{D}_{\tau^k}\cap
W_{\nu^k})\times(\mathcal{D}_{\nu^k}\cap
W_{\lambda_k})=(\mathcal{D}_{\tau^k}\cap
W_{{\sigma^k}})\times(\mathcal{D}_{{\sigma^k}}\cap
W_{\lambda_k}),
\end{equation}
where
\begin{equation*}
\aligned
\mathcal{D}_{\tau^k}\cap
W_{\nu^k}&=\{1,s_{\tilde a+1},s_{\tilde a+1}s_{\tilde a+2},\cdots,s_{\tilde a+1}\cdots
s_{\tilde a+b-1}\}\text{ and }\\
\mathcal{D}_{\tau^k}\cap
W_{{\sigma^k}}&=\{1,s_{\tilde a},s_{\tilde a}s_{\tilde a-1},\cdots,s_{\tilde a}\cdots
s_{\tilde a'+1}\}.
\endaligned
\end{equation*}

We now complete our computation in two cases.

\textbf{Case 1.  $\boldsymbol{a_{h+1,k}\geq 1}$ and $\boldsymbol{\hat{k}=0}$.} In this case, $\mathcal{D}_{\nu^k }\cap
W_{\lambda_k }\subseteq W_{\la^{(0)}}$. So $x=x_0$ for all $x\in \mathcal{D}_{\nu^k }\cap W_{\lambda_k }$ (or $x_1=1$).

If $h<m$, by the decomposition \eqref{DD'} for $\iota=\nu$, we have
\begin{equation*}
\begin{aligned}\Upsilon'&=\sum_{w\in\mathcal{D}_{\nu}\cap
W_{\la}}(-\bsq^{-1})^{\ell(w_1)} \sum_{x\in}v^A(v_h
v_{h+1}^{a_{h+1,k}-1}+\cdots+v_{h+1}^{a_{h+1,k}-1}  v_h)\cdot T_w\\
&=\sum_{d\in\sD,d'\in\sD'}(-\bsq^{-1})^{\ell(d_1d'_1)} \sum_{x\in\mathcal{D}_{\nu^k }\cap
W_{\la_k  }}v^A(v_h
v_{h+1}^{a_{h+1,k}-1}+\cdots+v_{h+1}^{a_{h+1,k}-1}  v_h)\cdot T_xT_{dd'}\\
\end{aligned}
\end{equation*}
However,
\begin{equation*}
\begin{aligned}
&\sum_{x\in\mathcal{D}_{\nu^k }\cap
W_{\la_k  }}v^A(v_h
v_{h+1}^{a_{h+1,k}-1}+\cdots+v_{h+1}^{a_{h+1,k}-1}  v_h)\cdot T_x\\
&=\sum_{\substack{y\in\mathcal{D}_{\tau^k }\cap
W_{\nu^k }\\ y'\in\mathcal{D}_{\nu^k }\cap
W_{\la_k  }}}\!\!(\cdots v^{a_{1,k}}_1 \cdots v_{h-1}^{a_{h-1,k}}
(v^{a_{h,k}+1}_h  v_{h+1}^{a_{h+1,k}-1})
v^{a_{h+2,k}}_{h+2} \cdots
v^{a_{m+n,k}}_{m+n}\cdots) \cdot T_yT_{y'}\\
&=\sum_{\substack{z\in\mathcal{D}_{\tau^k }\cap
W_{{\sigma^k} }\\ z'\in\mathcal{D}_{{\sigma}^k }\cap
W_{\la_k }}}(\cdots v^{a_{1,k}}_1 \cdots
v^{a+1}_h  v^{a_{h+1,k}-1}_{h+1}
v^{a_{h+2,k}}_{h+2} \cdots
v^{a_{m+n,k}}_{m+n}\cdots )\cdot T_{zz'}\,\;\;(a=a_{h,k})\\
&=[\![a+1]\!]\sum_{z'\in\mathcal{D}_{{\sigma}^k }\cap
W_{\la_k }}(\cdots v^{a_{1,k}}_1 \cdots
v^{a+1}_h  v^{a_{h+1,k}-1}_{h+1}
v^{a_{h+2,k}}_{h+2} \cdots
v^{a_{m+n,k}}_{m+n}\cdots )\cdot T_{z'}
\end{aligned}
\end{equation*}
(Here, and in the sequel, only the part associated with column $k$ is displayed. Other parts unchanged from $v^A$ are omitted.)
By \eqref{break}, substituting gives
$$\Upsilon'=[\![a_{hk}+1]\!]\sum_{z'\in\mathcal{D}_{{\sigma} }\cap
W_{\la  }}(-\bsq^{-1})^{\ell(z'_1)}v_\sigma\cdot T_{z'}=[\![a_{hk}+1]\!](v_\la)N_{(A+E_{h,k}-E_{h+1,k})'}.
$$

Assume now $h\geq m$, then $b=a_{h+1,k}=1$ and so $W_\nu=W_\tau$ as seen in \eqref{MM'}(1)
and $\sum_{i=1}^{a_{m+1,m}}(-\bsq^{-1})^{i-1}v_{m+1} ^{i-1}  v_m
v_{m+1}^{a_{m+1,k}-i}=v_m$.

 If $h=m$, by a similar argument as above. we have
\begin{equation*}
\begin{aligned}
\Upsilon'=&\sum_{x\in\mathcal{D}_{\nu }\cap W_{\lambda}}(-\bsq^{-1})^{\ell(x_1)}
v^A(\sum_{i=1}^{a_{m+1,m}}(-\bsq^{-1})^{i-1}v_{m+1} ^{i-1}  v_m
v_{m+1}^{a_{m+1,k}-i})
\cdot T_x\\
&=\sum_{x\in\mathcal{D}_{\tau }\cap
W_{\lambda}}(-\bsq^{-1})^{\ell(x_1)}(\cdots v_1^{a_{1k}} \cdots  v_m^{a_{m,k}+1}
v_{m+2}^{a_{m+2,k}} \cdots  v_{m+n}^{a_{m+n,k}}\cdots)\cdot T_x\\
&=[\![a_{m,k}+1]\!]\sum_{z\in\mathcal{D}_{{\sigma} }\cap
W_{\lambda}}(-\bsq^{-1})^{\ell(z_1)}v_\sigma\cdot T_z.
\end{aligned}
\end{equation*}

If $h>m$, then $a=a_{h,k}=0,1$. For $a=0$, by \eqref{MM'}, $W_\nu=W_\sigma=W_\tau$ and so
$$\aligned\Upsilon'&=\sum_{x\in\mathcal{D}_{\nu }\cap W_{\lambda}}
(-\bsq^{-1})^{\ell(x_1)}(\cdots v_1^{a_{1k}} \cdots  v_h
v_{h+2}^{a_{h+2,k}} \cdots  v_{m+n}^{a_{m+n,k}}\cdots)
\cdot T_x\\
&=[\![a_{h,k}+1]\!]\sum_{x\in\mathcal{D}_{\sigma }\cap W_{\lambda}}(-\bsq^{-1})^{\ell(x_1)}
v_\sigma\cdot T_x.\endaligned$$
For $a=1$,  $W_\nu=W_\tau$, $\sD_\tau\cap W_\sigma=\{1,s_{\tilde a}\}$ and $A+E_{h,k}-E_{h+1,k}\not\in M(m|n,r)$. But
$$\aligned
\Upsilon'&=\sum_{x\in\mathcal{D}_{\nu }\cap W_{\lambda}}
(-\bsq^{-1})^{\ell(x_1)}(\cdots v_1^{a_{1k}} \cdots  v_h^2
v_{h+2}^{a_{h+2,k}} \cdots  v_{m+n}^{a_{m+n,k}}\cdots)
\cdot T_x\\
&=\sum_{z'\in\mathcal{D}_{\sigma }\cap W_{\lambda}}(-\bsq^{-1})^{\ell(z'_1)}
\sum_{z\in\sD_\tau\cap W_\sigma}(\cdots v_1^{a_{1k}} \cdots  v_h^2
v_{h+2}^{a_{h+2,k}} \cdots  v_{m+n}^{a_{m+n,k}}\cdots)
\cdot T_zT_{z'}=0.\endaligned$$

{\bf Case 2. $\boldsymbol{a_{h+1,k}\geq1}$ and $\boldsymbol{\hat{k}=1}$.} In this case, $\mathcal{D}_{\nu^k}\cap W_{\la_k}\subseteq W_{\la^{(1)}}$. So
 $x=x_1$ for $x\in\mathcal{D}_{\nu^k }\cap W_{\la_k}$.

If $h<m$ then $b=a_{h+1,k}=1$ and $a=a_{h,k}\in\{0,1\}$. Repeatedly applying \eqref{MM'} yields
$$\aligned
\Upsilon'&=\sum_{x\in\mathcal{D}_{\tau }\cap
W_{\la  }}(-\bsq^{-1})^{\ell(x_1)}(\cdots v_1^{a_{1k}} \cdots
v_h^{a+1}  v_{h+2}^{a_{h+2,k}} \cdots
v_{m+n}^{a_{m+n,k}}\cdots)\cdot T_x\\
&=
\begin{cases}
0,&\text{ if }a=1;\\
\sum_{x\in\mathcal{D}_{{\sigma} }\cap
W_{\la  }}(-\bsq^{-1})^{\ell(x_1)}v_\sigma\cdot T_x,&\text{ if }a=0.
\end{cases}
\endaligned$$

Assume now $h=m$. Then $a=a_{h,k}\in\{0,1\}$ (but $a_{h+1,k}\geq1$). By the first formula of \eqref{T action},
$$\Upsilon'=\sum_{x\in \sD_\tau\cap W_\la}(-\bsq^{-1})^{\ell(x_1)}(\cdots v_1^{a_{1k}} \cdots
v_h^{a+1}v_{h+1}^{a_{h+1,k}}  v_{h+2}^{a_{h+2,k}} \cdots
v_{m+n}^{a_{m+n,k}}\cdots)\cdot T_x.$$
By \eqref{break} and \eqref{MM'}(3), a similar argument shows that
$$\Upsilon'=\begin{cases} 0,&\text{ if }a=1;\\
\sum_{x\in\sD_\sigma\cap W_\la}(-\bsq^{-1})^{\ell(x_1)}v_\sigma\cdot T_x,&\text{ if }a=0.\end{cases}$$

Finally, assume $h>m$. Then, by the decomposition \eqref{DD'},
\begin{equation*}
\begin{aligned}
\Upsilon'&=\sum_{w\in\mathcal{D}_{\nu}\cap
W_{\la}}(-\bsq^{-1})^{\ell(w_1)}[\cdots v_1^{a_{1,k}}  \cdots
v_h^{a_{h,k}}  (v_h
v_{h+1}^{a_{h+1,k}-1}\!\!+\bsq^{-1}v_{h+1}  v_h
v_{h+1}^{a_{h+1,k}-2}\\
&\quad\,\qquad\qquad\qquad\qquad\;\;+\cdots+(\bsq^{-1})^{a_{h+1,k}-1}v_{h+1}^{a_{h+1,k}-1}
v_h)
v_{h+2}^{a_{h+2,k}} \cdots  v_{m+n}^{a_{m+n,k}}\cdots]T_{w}\\
&=\sum_{d\in\sD,d'\in\sD'}(-\bsq^{-1})^{\ell(d_1d'_1)} \sum_{x\in\mathcal{D}_{\nu^k }\cap
W_{\la_k  }}(-\bsq^{-1})^{\ell(x)}[\cdots v_1^{a_{1,k}}  \cdots
v_h^{a_{h,k}}  (v_h
v_{h+1}^{a_{h+1,k}-1}+\!\!\\
&\quad\,\bsq^{-1}v_{h+1}  v_h
v_{h+1}^{a_{h+1,k}-2}+\cdots+(\bsq^{-1})^{a_{h+1,k}-1}v_{h+1}^{a_{h+1,k}-1}
v_h)
v_{h+2}^{a_{h+2,k}} \cdots  v_{m+n}^{a_{m+n,k}}\cdots]T_xT_{dd'}.\\
\end{aligned}
\end{equation*}
By \eqref{break} and repeatedly applying the third formula in \eqref{T action}, the inner summation becomes
\begin{equation*}
\begin{aligned}
&\sum_{\substack{y\in\mathcal{D}_{\tau^k }\cap
W_{\nu^k }\\x\in\mathcal{D}_{\nu^k }\cap
W_{\la_k }}}(-\bsq^{-1})^{\ell(yx)}(\cdots v_1^{a_{1,k}} \cdots
v_h^{a_{h,k}+1}  v_{h+1}^{a_{h+1}-1}
v_{h+2}^{a_{h+2,k}} \cdots  v_{m+n}^{a_{m+n,k}}\cdots)T_{yx}\\
&=\sum_{\substack{z\in\mathcal{D}_{\tau^k }\cap
W_{{\sigma}^k }\\z'\in\mathcal{D}_{{\sigma^k} }\cap
W_{\la_k }}}(-\bsq^{-1})^{\ell(zz')}(\cdots v_1^{a_{1,k}} \cdots
v_h^{a_{h,k}+1}  v_{h+1}^{a_{h+1}-1}
v_{h+2}^{a_{h+2,k}} \cdots  v_{m+n}^{a_{m+n,k}}\cdots)T_{zz'}\\
&=\sum_{z'\in\mathcal{D}_{{\sigma^k} }\cap
W_{\la_k }}\sum_{i=0}^{a_{h,k}}(\bsq^{-1})^i(-\bsq^{-1})^{\ell(z')}(\cdots v_1^{a_{1k}} \cdots
v_h^{a_{h,k}+1}  v_{h+1}^{a_{h+1}-1}
v_{h+2}^{a_{h+2,k}} \cdots  v_{m+n}^{a_{m+n,k}}\cdots)T_{z'}.\\
\end{aligned}
\end{equation*}
Hence, substituting gives
$$\Upsilon'=[\![a_{h,k}+1]\!]_{\bsq^{-1}}\sum_{z'\in\mathcal{D}_{{\sigma} }\cap
W_{\la }}(-\bsq^{-1})^{\ell(z_1')}v_\sigma\cdot T_{z'}.$$
So we have proved all cases for the first formula.

To prove the second formula, we simply reverse the roles of $h$ and $h+1$. Note that
$$\aligned
\mu&=(\rho_1,\ldots,\rho_{h-1},\rho_h-1,\rho_{h+1}+1,\rho_{h+2},\ldots),\\
\rho'&=(\rho_1,\ldots,\rho_{h-1},\rho_h-1,1,\rho_{h+1},\rho_{h+2},\ldots),\\
\endaligned
$$ and so
$$\mathcal{D}_{\rho'}\cap W_\rho=\begin{cases}\mathcal{D}_{\rho'}\cap W_{\rho^{(0)}},&\text{if }h\leq m;\\
\mathcal{D}_{\rho'}\cap W_{\rho^{(1)}},&\text{if }h> m.\end{cases}$$
Thus, $\Upsilon\cdot T_d$ can be computed similarly.
By noting
\begin{equation*}
\begin{aligned}
&{\nu}^{k}=(a_{1,k},a_{2,k},\ldots,\;\; a,\qquad b,\;\; a_{h+2,k},\ldots),\\
&{\sigma}^{k}=(a_{1,k},a_{2,k},\ldots,a-1,b+1,a_{h+2,k},\ldots),\\
&\tau^{k}=(a_{1,k},a_{2,k},\ldots,a-1,1,\;\;\,b,a_{h+2,k},\ldots),
\end{aligned}
\end{equation*}
where $a=a_{h,k}$ and $b=a_{h+1,k}$, the element $\Upsilon'$ is also computed similarly.
We leave the details to the reader.
\end{proof}

\begin{remark}\label{$N_{A'}=0$} As seen in Lemma \ref{main lemma}, we will make the convention $N_{A'}=0$ for any matrix $A\in M_{m+n}(\mathbb Z)$ but $A\notin M(m|n,r)$.
\end{remark}

\section{Some multiplication formulas in $\sS_\bsq(m|n,r)$}

We now use the key formulas given in Lemma \ref{main lemma} to derive more multiplication formulas which will be used to establish the realisation of the supergroup $\bfU(\mathfrak{gl}_{m|n})$.

For $\lambda=(\lambda_1,\cdots,\lambda_{m+n})\in \Lambda(m|n,r)$ and fixed $1\leq
h<m+n$, let
\begin{equation*}
\begin{aligned}
&B_p=B_p(h,\la)=\diag(\la+p\bse_h-(p+1)\bse_{h+1})+E_{h,h+1},\;\text{ resp.,}\\
&C_{p}=C_p(h,\la)=\diag(\la-(p+1)\bse_h+p\bse_{h+1})+E_{h+1,h},
\end{aligned}
\end{equation*}
 for all $0\leq p<\la_{h+1}$, resp., $0\leq p<\la_{h}$, and let
\begin{equation}\label{$U_p,L_p$}
\begin{aligned}
&U_p=U_p(h,\la)=\diag(\la-p\bse_{h+1})+pE_{h,h+1},\text{ resp.,}\\
&L_{p}=L_{p}(h,\la)=\diag(\la-p\bse_{h})+pE_{h+1,h},
\end{aligned}
\end{equation}
 for all $0\leq p\leq \la_{h+1}$, resp., $0\leq p\leq\la_{h}$.
Clearly, $\co(U_p)=\co(L_p)=\la$, $\ro(U_p)=\co(B_p)$ and $\ro(L_p)=\co(C_p)$. We need the following relations for a proof by induction.

\begin{lemma}\label{induction corollary}
Maintain the notation above and assume $1\leq h<m+n$ . We have
\begin{itemize}
\item[(1)]
$N_{U_p'}N_{B_p'}=[\![p+1]\!]_{\bsq_h}N_{U_{p+1}'}$ for all $p\geq 1$;
\item[(2)]
$N_{L_{p}'}N_{C_{p}'}= [\![p+1]\!]_{\bsq_h}N_{L_{p+1}'}$ for all $p\geq 1$.
\end{itemize}
(Note that, for $h=m$, the RHS in both cases is 0 and, for $p\geq\lambda_{h+1}$, resp.,
$p\geq\la_h$, we also have RHS$\,=0$ in (1), resp., in (2).)
\end{lemma}
\begin{proof}Take $A$ to be $U_p$ in Lemma \ref{main lemma}(1). If
$a_{h+1,k}\geq 1$, then $k=h+1$ and $a_{h,h+1}=p$. Thus, $f(\bsq;A,h)=1$ for $h\neq m$ and, hence,
$N_{U_p'}N_{B_p'}=[\![p+1]\!]_{\bsq_h}N_{U_{p+1}'}$.
If $h=m$, then $U_{p+1}\notin M(m|n,r)$ and, hence,
$N_{U_{p+1}'}=0$. The second formula can be proved similarly.
\end{proof}

For $A\in M(m|n,r)$, set $\row_h(A)=(a_{h,1},a_{h,2},\cdots,a_{h,m+n})$ to be the $h$-th row of $A$.
Then $\row_h(A)\in\mathbb N^{m+n}$. For $\nu,\nu'\in\mathbb{N}^{m+n}$, set
$$\nu\leq \nu'\iff \nu_j\leq \nu_j'\text{ for all }j\in[1,m+n].$$
Like the polynomials defined in \eqref{fqAh} and \eqref{fqAh+1}, we define, for $A=(a_{i,j})\in M(m|n,r)$, $h\in[1,m+n)$, and $\nu \in\mathbb N^{m+n}$ with $|\nu|=p\geq 2$,
 \begin{equation}\label{fnuqAh}
f_\nu(\bsq;A,h)=\begin{cases}\bsq^{\sum_{j>t}a_{h,j}\nu_t},&\text{if }h<m;\\
0,&\text{if }h=m;\\
\bsq^{\sum_{l<t}\nu_l\nu_t}(\bsq^{-1})^{\sum_{j<t}a_{h+1,j}\nu_t},&\text{if }h>m,\end{cases}
\end{equation}
and
 \begin{equation}\label{fnuqAh+1}
g_\nu(\bsq;A,h+1)=\begin{cases}\bsq^{\sum_{j<t}a_{h+1,j}\nu_t},&\text{if }h<m;\\
0,&\text{if }h=m;\\
\bsq^{\sum_{l<t}\nu_l\nu_t}
(\bsq^{-1})^{\sum_{j>t}a_{h,j}\nu_t},&\text{if }h>m.\end{cases}
\end{equation}
Note that, if $|\nu|=1$, we may extend the definition by setting
$f_\nu(\bsq;A,h)=f_k(\bsq;A,h)$ and $g_\nu(\bsq;A,h+1)=g_k(\bsq;A,h+1)$ as in \eqref{fqAh} and \eqref{fqAh+1}, if $\nu_k=1$.

Now the multiplication formulas in Lemma \ref{main lemma}
can be generalised from $U_1=B$ and $L_1=C$ to
$U_p$ and $L_p$, respectively.

\begin{theorem}\label{U_i and L_i}
Let $A=(a_{i,j})\in M( m|n,r)$. Assume $1\leq
h<m+n$. Define $U_p=U_{p}(h,\ro(A)),L_p=L_{p}(h,\ro(A))\in M(m|n,r)$ as in \eqref{$U_p,L_p$}. The following
multiplication formulas hold in the $\bsq$-Schur superalgebra
$\sS_\bsq(m|n,r)$ over $\sA$:
\begin{equation*}
\begin{aligned}
(1)\quad &N_{A'}N_{U_p'}=\sum_{\substack{\nu\in\Lambda(m|n,p)\\\nu\leq
 \row_{h+1}(A)}}f_\nu(\bsq;A,h)\prod_{k=1}^{m+n}\left[\!\!\left[a_{h,k}+\nu_k\atop\nu_k\right]\!\!\right]_{\bsq_h}N_{(A+\sum_l\nu_l(E_{h,l}-E_{h+1,l}))'};\\
(2)\quad &N_{A'}N_{L_{p}'}=\sum_{\substack{\nu\in\Lambda(m|n,p)\\\nu\leq
 \row_{h}(A)}}g_\nu(\bsq;A,h+1)\prod_{k=1}^{m+n}\left[\!\!\left[a_{h+1,k}+\nu_k\atop\nu_k\right]\!\!\right]_{\bsq_{h+1}}N_{(A-\sum_l\nu_l(E_{h,l}-E_{h+1,l}))'}.
 \end{aligned}
 \end{equation*}
 Here, as before, $N_X=0$ for a matrix
 $X\not\in
 M(m|n,r)$.
\end{theorem}
\begin{proof} We only prove (1). The proof of (2) is similar.

By the definition of the coefficients $f_\nu(\bsq;A,h)$, (1) is true for $p=1$. Assume now $p>1$
and $h=m$, then $f_\nu(\bsq;A,h)=0$ for all $\nu$ and $U_p\notin M(m|n,r)$. So $N_{U_p}=0$ by the convention
 in Remark \ref{$N_{A'}=0$}. Hence,
$N_{A'}N_{U_p'}=0$, proving the formula in this case. We now assume $h\neq m$ and let
$\bse_k=(0,\ldots,0,\underset{(k)}1,0\ldots)$; see \eqref{bse}. The proof is parallel to that in \cite{BLM} by
applying induction on $p$.

By Corollary \ref{induction corollary},
 $N_{U_p'}N_{B_p'}=[\![p+1]\!]_{\bsq_h} N_{U_{p+1}'}$. We assume $U_{p+1}\in M(m|n,r)$ (so $p+1\leq
 \sum_{k=1}^{m+n}a_{h+1,k}$).  Then
 \begin{equation}\label{$AU$}
\begin{aligned}
&N_{A'}N_{U_{p+1}'}=\frac1{[\![p+1]\!]_{\bsq_h}}N_{A'}N_{U_p'}N_{B_p'}\\
&=\frac1{[\![p+1]\!]_{\bsq_h}}\biggl(\sum_{\substack{\nu\in\Lambda(m+n,p)\\\nu\leq
\row_{h+1}(A)}}f_\nu(\bsq;A,h)\prod_{k=1}^{m+n}\left[\!\!\left[a_{h,k}+\nu_k\atop
\nu_k\right]\!\!\right]_{\bsq_h} N_{A'+\sum_l\nu_l(E_{h,l}-E_{h+1,l})'}\biggr)N_{B_p'}.
\end{aligned}
\end{equation}
Putting $A_\nu=A+\sum_l\nu_l(E_{h,l}-E_{h+1,l})$, we have, by Lemma \ref{main lemma}(1),
\begin{equation}\label{$AB$}
N_{A_\nu'}N_{B_p'}=\sum_{\substack{s\in[1,m+n]\\a_{h+1,s}-\nu_s\geq
1}}f_s(\bsq;A_\nu,h)[\![a_{h,s}+\nu_s+1]\!]_{\bsq_h}
N_{(A_\nu+E_{h,s}-E_{h+1,s})'}.
\end{equation}

For fixed $\eta\in\Lambda(m+n,p+1)$, if
$A+\sum_l\eta_l(E_{h,l}-E_{h+1,l})\in M(m|n,r)$, then
$\eta\leq\row_{h+1}(A)$. By inserting
\eqref{$AB$} into \eqref{$AU$}, we see that the coefficient $f'_\eta(\bsq;A,h)$ of
$N_{A_\eta'}=N_{(A+\sum_l\eta_l(E_{h,l}-E_{h+1,l}))'}$ is a sum over pairs $(s,\nu)$ with $\eta=\nu+\bse_s$. Hence, $f'_\eta(\bsq;A,h)$
 is equal to
\begin{equation*}
\begin{aligned}
\frac1{[\![p+1]\!]_{\bsq_h}}\!\!\!\!\sum_{\substack{s\in[1,m+n]\\\eta-\bse_s\in\Lambda(m+n,p)}}\!\!\!\!
f_{\eta-\bse_s}(\bsq;A,h)\prod_{k=1}^{m+n}
\left[\!\!\left[a_{h,k}+(\eta-\bse_s)_k\atop
(\eta-\bse_s)_k\right]\!\!\right]_{\bsq_h}f_s(\bsq;A_{\eta-\bse_s},h)
[\![a_{h,s}+\eta_s]\!]_{\bsq_h}.
\end{aligned}
\end{equation*}

If $h<m$, then $\bsq_h=\bsq$, $f_s(\bsq;A_{\eta-\bse_s},h)=\bsq^{\sum_{j>s}a_{h,j}+(\eta-\bse_s)_j}$,
 and $f_{\eta-\bse_s}(\bsq;A,h)=\bsq^{\sum_{j>t}a_{h,j}(\eta-\bse_s)_t}$. Since
\begin{equation*}
\begin{aligned}
&\prod_{k=1}^{m+n}
\left[\!\!\left[a_{h,k}+(\eta-\bse_s)_k\atop(\eta-\bse_s)_k\right]\!\!\right]\cdot [\![a_{h,s}+\eta_s]\!]
=\prod_{k=1}^{m+n}\left[\!\!\left[a_{h,k}+\eta_k\atop\eta_k\right]\!\!\right]\cdot[\![\eta_s]\!],\\
\sum_{j>t}a_{h,j}(\eta&-\bse_s)_t+\sum_{j>s}(a_{h,j}+(\eta-\bse_s)_j)=\sum_{j>t}a_{h,j}\eta_t-\sum_{j>s}a_{h,j}+\sum_{j>s}a_{h,j}+\sum_{j>s}\eta_j\\
&=\sum_{j>t}a_{h,j}\eta_t+\sum_{j>s}\eta_j,
\end{aligned}
\end{equation*}
and
$$[\![p+1]\!]=\sum_{\substack{s\in[1,m+n]\\\eta-e_s\in\Lambda(m+n,p)}}
\bsq^{\sum_{j>s}\eta_j}[\![\eta_s]\!],$$
it follows that
$$\aligned
f'_\eta(\bsq;A,h)&=\frac1{[\![p+1]\!]}\bsq^{\sum_{j>t}a_{h,j}\eta_t}\prod_{k=1}^{m+n}\left[\!\!\left[a_{h,k}+\eta_k\atop
\eta_k\right]\!\!\right]\sum_{\substack{s\in[1,m+n]\\\eta-e_s\in\Lambda(m+n,p)}}
\bsq^{\sum_{j>s}\eta_j}[\![\eta_s]\!]\\
&=\bsq^{\sum_{j>t}a_{h,j}\eta_t}\prod_{k=1}^{m+n}\left[\!\!\left[a_{h,k}+\eta_k\atop
\eta_k\right]\!\!\right]=f_\eta(\bsq;A,h)\prod_{k=1}^{m+n}\left[\!\!\left[a_{h,k}+\eta_k\atop
\eta_k\right]\!\!\right],\endaligned
$$
as desired in this case.

If $h>m$, then $\bsq_h=\bsq^{-1}$. We first observe that $\bsq^{\frac{1}{2}p(p-1)}=\frac{[\![p]\!]^!}{[\![p]\!]^!_{{\bsq^{-1}}}}$ and
$\frac12p(p-1)-\sum_{k=1}^{m+n}\frac12\nu_k(\nu_k-1)=\sum_{l<t}\nu_l\nu_t$.
It follows that
$$\aligned
\frac{[\![p]\!]^!}{[\![p]\!]^!_{{\bsq^{-1}}}}(\bsq^{-1})^{\sum_{j<t}a_{h+1,j}\nu_t}\prod_{k=1}^{m+n}
 \frac{[\![a_{h,k}+\nu_k]\!]^!_{{\bsq^{-1}}}}{[\![\nu_k]\!]^![\![a_{h,k}]\!]^!_{\bsq^{-1}}}&=
 f_\nu(\bsq;A,h)\prod_{k=1}^{m+n}\left[\!\!\left[a_{h,k}+\nu_k\atop
\nu_k\right]\!\!\right]_{\bsq^{-1}}.\endaligned
 $$
 In this case, $f_s(\bsq;A_{\eta-\bse_s},h)=(\bsq^{-1})^{\sum_{j<s}a_{h+1,j}-(\eta-\bse_s)_j}$. Thus,
\begin{equation*}
\begin{aligned}
f'_{\eta}(\bsq;A,h)&=\frac{[\![p]\!]^!}{[\![p+1]\!]^!_{\bsq^{-1}}}\!\!\!\sum_{\substack{s\in[1,m+n],\\
\eta-\bse_s\in\Lambda(m+n,p)}}(\bsq^{-1})^{\sum_{j<t}a_{h+1,j}(\eta-\bse_s)_t}\prod_{k=1}^{m+n}
\frac{[\![a_{h,k}+(\eta-\bse_s)_k]\!]_{\bsq^{-1}}}{[\![(\eta-\bse_s)_k]\!]^![\![a_{h,k}]\!]_{\bsq^{-1}}}\cdot\\
&\qquad\qquad\cdot(\bsq^{-1})^{\sum_{j<s}a_{h+1,j}-(\eta-\bse_s)_j}[\![a_{h,s}+\eta_s]\!]_{\bsq^{-1}}.
\end{aligned}
\end{equation*}
 Since $$\aligned&\prod_{k=1}^{m+n}
\frac{[\![a_{h,k}+(\eta-\bse_s)_k]\!]^!_{\bsq^{-1}}}
{[\![(\eta-\bse_s)_k]\!]^![\![a_{h,k}]\!]_{\bsq^{-1}}}\cdot {[\![a_{h,s}+\eta_s]\!]_{\bsq^{-1}}}=\prod_{k=1}^{m+n}
\frac{[\![a_{h,k}+\eta_k]\!]^!_{\bsq^{-1}}}{[\![\eta_k]\!]^![\![a_{h,k}]\!]^!_{\bsq^{-1}}}
\cdot{[\![\eta_s]\!]},
\endaligned$$ and since
\begin{equation*}
\begin{aligned}
&\sum_{j<t}a_{h+1,j}(\eta-\bse_s)_t+\sum_{j<s}(a_{h+1,j}-(\eta-\bse_s)_j)\\
&=\sum_{j<t}a_{h+1,j}\eta_t-\sum_{j<s}a_{h+1,j}+
\sum_{j<s}a_{h+1,j}-\sum_{j<s}\eta_j\\
&=\sum_{j<t}a_{h+1,j}\eta_t-\sum_{j<s}\eta_{j},
\end{aligned}
\end{equation*}
it follows that
$$
\aligned
f'_{\eta}(\bsq;A,h)&=
\frac{[\![p]\!]!}{[\![p+1]\!]^!_{\bsq^{-1}}}(\bsq^{-1})^{\sum_{j<t}a_{h+1,j}\eta_t}
\prod_{k=1}^{m+n}\frac{[\![a_{h,k}+\eta_k]\!]_{\bsq^{-1}}}{[\![\eta_k]\!]^![\![a_{h,k}]\!]_{\bsq^{-1}}}
\!\!\!\sum_{\substack{s\in[1,m+n],\\\eta-\bse_s\in\Lambda(m+n,p)}}
\bsq^{\sum_{j<s}\eta_{j}}[\![\eta_s]\!]\\
&=\frac{[\![p+1]\!]^!}{[\![p+1]\!]^!_{\bsq^{-1}}}(\bsq^{-1})^{\sum_{j<t}a_{h+1,j}\eta_t}\prod_{k=1}^{m+n}
\frac{[\![a_{h,k}+\eta_k]\!]^!_{\bsq^{-1}}}{[\![\eta_k]\!]^![\![a_{h,k}]\!]^!_{\bsq^{-1}}}\\
&=f_\eta(\bsq;A,h)\prod_{k=1}^{m+n}\left[\!\!\left[a_{h,k}+\eta_k\atop\eta_k\right]\!\!\right]_{\bsq^{-1}},
\endaligned$$
as desired.
\end{proof}

\section{Normalised multiplication formulas in $\sS_\up(m|n,r)$}
The multiplication formulas discovered in the last section are derived through a $T$-basis action on the tensor superspace. It is somewhat symmetric relative to $\up^2$ and $-1$, the eigenvalues of $T_s$.
For example, the $h=m$ case in \eqref{fqAh} and \eqref{fqAh+1} reflects such symmetry.
In order to obtain a realisation for quantum super $\mathfrak {gl}_{m|n}$, we need to recover a  symmetry relative to $\up$ and $-\up^{-1}$, the eigenvalues of the normalised generators $\up^{-1}T_s$. Thus, we need to normalise the basis $\{N_{A'}\}_A$ and twist the right action to a left action via the anti-involution $\tau$ given in Lemma \ref{norm basis}(2).

For $A=(a_{i,j})\in M(m|n,r)$, define
\begin{equation}\label{d(A)}
d(A)=\sum_{i>k,j<l}a_{i,j}a_{k,l}+\sum_{j<l}(-1)^{\hat{i}}a_{i,j}a_{i,l},\qquad\xib_A=\bsq^{-d(A)}\tau(N_{A'}).
\end{equation}

\begin{remark}
Actually, if $A$ is corresponding to the triple
$(\lambda|\mu,d,\xi|\eta)$ where $\lambda|\mu, \xi|\eta\in
\Lambda(m|n,r)$ and $d\in\mathcal{D}^\circ_{\lambda|\mu,\xi|\eta}$,
using the notation in \cite[\S6]{DR},
$d(A)=\ell(d^*)-\ell(^*d)+\ell(d)-\ell(w_{0,\xi})+\ell(w_{0,\eta})$.  \end{remark}

Recall $\bsq=\up^2$ and $\up_h=\up^{(-1)^{\hat{h}}}$. The multiplication formulas in
Theorem \ref{U_i and L_i} turn out to be quite symmetric if we use the normalised basis
$\{\xib_A\}_{A\in M(m|n,r)}$ for $\sS_\up(m|n,r)$. We first look at the even generator case.

\begin{proposition}[\bf The $\boldsymbol{h\neq m}$ case]\label{integral basis multiplication}
Let $A=(a_{i,j})\in M( m|n,r)$. Assume $1\leq
h<m+n$ and $h\neq m$. Define $U_p=U_{p}(h,\ro(A)),L_p=L_{p}(h,\ro(A))\in M(m|n,r)$ as in \eqref{$U_p,L_p$}. The
following multiplication formulas hold in the $\up$-Schur
superalgebra $\mathcal{S}_\up(m|n,r)$ over $\mathcal{Z}$.
\begin{itemize}
\item[(1)] $\xib_{U_p}\xib_A=\sum_{\substack{\nu\in\Lambda(m|n,p)\\\nu\leq
 row_{h+1}(A)}}\up_h^{f_h(\nu,A)}\prod_{k=1}^{m+n}\overline{\left[\!\!\left[a_{h,k}+\nu_k\atop\nu_k\right]\!\!\right]}_{\up_h^2}
 \xib_{A+\sum_l\nu_l(E_{h,l}-E_{h+1,l})},$\\
 where
 \begin{equation}\label{beta_h1}
 f_h(\nu,A)=\sum_{j\geq
 t}a_{h,j}\nu_t-\sum_{j>t}a_{h+1,j}\nu_t+\sum_{t<t'}\nu_t\nu_{t'}.\end{equation}

\item[(2)]$\xib_{L_p}\xib_A=\sum_{\substack{\nu\in\Lambda(m|n,p)\\\nu\leq
 row_{h}(A)}}\up_{h+1}^{f'_h(\nu,A)}\prod_{k=1}^{m+n}\overline{\left[\!\!\left[a_{h+1,k}+\nu_k\atop\nu_k\right]\!\!\right]}_{\up_{h+1}^2}
 \xib_{A-\sum_l\nu_l(E_{h,l}-E_{h+1,l})},$\\
 where
 \begin{equation}\label{gamma_h1}
 f'_h(\nu,A)=\sum_{j\leq
 t}a_{h+1,j}\nu_t-\sum_{j<t}a_{h,j}\nu_t+\sum_{t<t'}\nu_t\nu_{t'}.\end{equation}
\end{itemize}
 \end{proposition}
 Note that $\up_h=\up_{h+1}$ if $h\neq m$. Note also that, in this $h\neq m$ case, these formulas are the super version of the formulas (a2) and (b2) in \cite[Lem.~3.4]{BLM}. Formerly, they can be obtained by replacing $\up$ there by $\up_h$.

\begin{proof}
We just give the proof of (1). The proof of (2) is
similar.

For $\nu\in\Lambda(m|n,p)$ with $\nu\leq \row_{h+1}(A)$, set
$X=A+\sum_l\nu_l(E_{h,l}-E_{h+1,l})$. Since
$$\left[\!\!\left[N\atop s\right]\!\!\right]_{\up^2}=\up^{2s(N-s)}\overline{\left[\!\!\left[N\atop
s\right]\!\!\right]}_{\up^2},\quad\text{ for all }\quad N\geq s\geq0,$$
it follows that
$$\prod_{k=1}^{m+n}\left[\!\!\left[a_{h,k}+\nu_k\atop
\nu_k\right]\!\!\right]_{\up^2}=\up^{2\sum_{k=1}^{m+n}a_{hk}\nu_k}\prod_{k=1}^{m+n}\overline{\left[\!\!\left[a_{h,k}+\nu_k\atop
\nu_k\right]\!\!\right]}_{\up^2}.$$
By a comparison with Theorem \ref{U_i and L_i}(1),
 the proof of (1) is equivalent to the identity.
\begin{equation}\label{goal}
\up^{d(X)-d(A)-d(U_p)}f_\nu(\bsq;A,h)\up^{2\sum_{k=1}^{m+n}a_{hk}\nu_k}=
\up_h^{f_h(A,\nu)}.
\end{equation}

Let
$$\aligned
J&=\{(i,j,k,l)\mid 1\leq i,j,k,l\leq m+n, i>k,j<l\},\\
 I&=\{(i,j,k,l)\in J\mid i,k\not\in\{h,h+1\}\},\\
J'&=\{(i,j,i,l)\mid 1\leq i,j,l\leq m+n, j<l\}, \text{ and }\\
I'&=\{(i,j,i,l)\in J'\mid i\not\in\{h,h+1\}\}.
 \endaligned
 $$
By definition, $x_{i,j}=a_{i,j}$ for all $i,j$ with $i\not\in\{h,h+1\}$. Hence,
$$d(X)=\sum_{I}a_{i,j}a_{k,l}+\sum_{J\backslash I}x_{i,j}x_{k,l}+\sum_{I'}(-1)^{\hat i}a_{i,j}a_{i,l}+\sum_{J'\backslash I'}(-1)^{\hat{i}}x_{i,j}x_{il},$$
where
\begin{equation*}
\begin{aligned}
&\sum_{J\backslash I}x_{i,j}x_{k,l}\\
&=\sum_{i=h>k,j<l}(a_{h,j}+\nu_j)a_{k,l}+\sum_{i=h+1,h>k,j<l}(a_{h+1,j}-\nu_j)a_{k,l}+\\
&\quad+\sum_{i=h+1,k=h,j<l}(a_{h+1,j}-\nu_j)(a_{h,l}+\nu_l)+\\
&\quad+\sum_{i>h+1,k=h,j<l}a_{i,j}(a_{h,l}+\nu_l)+\sum_{i>h+1=k,j<l}a_{i,j}(a_{h+1,l}-\nu_l)\\
&=\sum_{J\backslash I}a_{i,j}a_{k,l}+\sum_{h>k,j<l}a_{k,l}\nu_j -\sum_{h>k,j<l}a_{k,l}\nu_j-\sum_{j<l}a_{h,l}\nu_j+\sum_{j<l}a_{h+1,j}\nu_l\\
&\quad -\sum_{j<l}\nu_j\nu_l
+\sum_{i>h+1,j<l}a_{i,j}\nu_l-\sum_{i>h+1,j<l}a_{i,j}\nu_l\\
&=\sum_{J\backslash I}a_{i,j}a_{k,l}-\sum_{j<l}a_{h,l}\nu_j
+\sum_{j<l}a_{h+1,j}\nu_l-\sum_{j<l}\nu_j\nu_l,\;
\end{aligned}
\end{equation*}
and
\begin{equation*}
\begin{aligned}
&\sum_{J'\backslash I'}(-1)^{\hat{i}}x_{i,j}x_{i,l}\\
&=\sum_{j<l}(-1)^{\hat{h}}(a_{h,j}+\nu_j)(a_{h,l}+\nu_l)+\sum_{j<l}(-1)^{\widehat{h+1}}(a_{h+1,j}-\nu_j)(a_{h+1,l}-\nu_l)\\
&=\sum_{J'\backslash I'}(-1)^{\hat{i}}a_{i,j}a_{i,l}+\biggl(\sum_{j<l}(-1)^{\hat{h}}a_{h,j}\nu_l
+\sum_{j<l}(-1)^{\hat{h}}a_{h,l}\nu_j\biggr)-\sum_{j<l}(-1)^{\widehat{h+1}}a_{h+1,j}\nu_l\\
&\quad-\sum_{j<l}(-1)^{\widehat{h+1}}a_{h+1,l}\nu_j+\sum_{j<l}((-1)^{\hat{h}}+(-1)^{\widehat{h+1}})\nu_j\nu_l.
\end{aligned}
\end{equation*}
Thus, since $d(U_p)=(-1)^{\hat{h}}p\sum_{j=1}^{m+n}a_{h,j}=(-1)^{\hat{h}}\sum_{1\leq j,k\leq m+n}a_{h,j}\nu_k$ by \eqref{$U_p,L_p$}, it follows that
\begin{equation*}
\begin{aligned}
d(X)=&d(A)-\sum_{j<l}a_{h,l}\nu_j+\sum_{j<l}a_{h+1,j}\nu_l-\sum_{j<l}\nu_j\nu_l+
\biggl(d(U_p) -(-1)^{\hat{h}}\sum_{j}a_{h,j}\nu_j\biggr)\\
&-\sum_{j<l}(-1)^{\widehat{h+1}}a_{h+1,j}\nu_l-\sum_{j<l}(-1)^{\widehat{h+1}}a_{h+1,l}\nu_j+\sum_{j<l}((-1)^{\hat{h}}+(-1)^{\widehat{h+1}})\nu_j\nu_l.
\end{aligned}
\end{equation*}
In other words,
\begin{equation}\label{$d(X)-d(A)-d(U_p)$}
\begin{aligned}
&d(X)-d(A)-d(U_p)\\
&=-\sum_{j<l}a_{h,l}\nu_j+\sum_{j<l}a_{h+1,j}\nu_l -(-1)^{\hat{h}}\sum_{j}a_{h,j}\nu_j-\sum_{j<l}(-1)^{\widehat{h+1}}a_{h+1,j}\nu_l+\\
&\quad-\sum_{j<l}(-1)^{\widehat{h+1}}a_{h+1,l}\nu_j+\sum_{j<l}((-1)^{\hat{h}}+(-1)^{\widehat{h+1}}-1)\nu_j\nu_l\\
&=\begin{cases}-\sum_{j\leq
l}a_{h,l}\nu_j-\sum_{j<l}a_{h+1,l}\nu_j+\sum_{j<l}\nu_j\nu_l,\;\;\,\qquad\quad\;\text{ if }h<m;\\
-\sum_{j<l}a_{h,l}\nu_j+2\sum_{j<l}a_{h+1,j}\nu_l+\sum_ja_{h,j}\nu_j\\
\qquad\qquad\qquad+\sum_{j<l}a_{h+1,l}\nu_j -3\sum_{j<l}\nu_j\nu_l,\qquad\;\quad\text{ if }h>m;\\
-\sum_{j>k}a_{m,j}+2\sum_{j<k}a_{m+1,j}+\sum_{j>k}a_{m+1,j}-a_{m,k},\text{ if }h=m,\nu=\bse_k.
\end{cases}
\end{aligned}
\end{equation}
Thus, for $h<m$, the exponent of the LHS of \eqref{goal} becomes
\begin{equation*}
\begin{aligned}
&d(X)-d(A)-d(U_p)+\sum_{j>l}2a_{h,j}\nu_l+2\sum_{k=1}^{m+n}a_{h,j}\nu_j\\
 &=-\sum_{j\leq
l}a_{h,l}\nu_j-\sum_{j<l}a_{h+1,l}\nu_j+\sum_{j<l}\nu_j\nu_l+\sum_{j>l}2a_{h,j}\nu_l+2\sum_{k=1}^{m+n}a_{h,k}\nu_k\\
&=\sum_{j>l}a_{h,j}\nu_l+\sum_{k=1}^{m+n}a_{h,j}\nu_j-\sum_{j<l}a_{h+1,l}\nu_j+\sum_{j<l}\nu_j\nu_l\\
&=\sum_{j\geq
l}a_{h,j}\nu_l-\sum_{j<l}a_{h+1,l}\nu_j+\sum_{j<l}\nu_j\nu_l=f_h(\nu,A),
\end{aligned}
\end{equation*}
while, for $h>m$, it has the form
\begin{equation*}
\begin{aligned}
&d(X)-d(A)-d(U_p)+2\sum_{j<l}\nu_j\nu_l-2\sum_{j<t}a_{h+1,j}\nu_t-2\sum_{k=1}^{m+n}a_{h,k}\nu_k\\
&=-\sum_{j<l}a_{h,l}\nu_j+2\sum_{j<l}a_{h+1,j}\nu_l+\sum_ja_{h,j}\nu_j+
\sum_{j<l}a_{h+1,l}\nu_j-3\sum_{j<l}\nu_j\nu_l\\
&\quad-2\sum_{j<t}a_{h+1,j}\nu_t-2\sum_{k=1}^{m+n}a_{h,k}\nu_k+2\sum_{j<l}\nu_j\nu_l\\
&=-\sum_{j\leq l}a_{h,l}\nu_j
+\sum_{j<l}a_{h+1,l}\nu_j-\sum_{j<l}\nu_j\nu_l=-f_h(\nu,A),
\end{aligned}
\end{equation*}
as desired.
 \end{proof}
Let
\begin{equation}\label{beta_m gamma_m}
\aligned
f_m(\bse_k,A)&=\sum_{j\geq k}a_{m,j}+\sum_{j>k}a_{m+1,j},\\
f'_m(\bse_k,A)&=\sum_{j\leq
k}a_{m+1,j}+\sum_{j<k}a_{m,j}.\endaligned
\end{equation}
We now look at the odd generator case.
\begin{proposition}[\bf The $\boldsymbol{h=m}$ case]\label{integral basis multiplication $h=m$}
Let $A=(a_{i,j})\in M( m|n,r)$. Assume $h= m$.
Define $U_1=U_{1}(h,\ro(A)),L_1=L_{1}(h,\ro(A))\in M(m|n,r)$ as in \eqref{$U_p,L_p$}. The following multiplication
formulas hold in the $\up$-Schur superalgebra $\mathcal{S}_\up(m|n,r)$
over $\mathcal{Z}$:
$$\aligned
(1)\; \xib_{U_1}\xib_A&=\sum_{\substack{k\in[1,m+n]\\ a_{m+1,k}\geq
1}}(-1)^{\sum_{i>m,j<k}a_{i,j}}\up_m^{f_m(\bse_k,A)}\overline{[\![a_{m,k}+1]\!]}_{\up^2_m}\xib_{A+E_{m,k}-E_{m+1,k}};\\
(2)\; \xib_{L_1}\xib_A&=\sum_{\substack{k\in[1,m+n]\\ a_{m,k}\geq
1}}(-1)^{\sum_{i>m,j<k}a_{i,j}}\up_{m+1}^{f'_m(\bse_k,A)}\overline{[\![a_{m+1,k}+1]\!]}_{\up^2_{m+1}}\xib_{A-E_{m,k}+E_{m+1,k}}.\endaligned$$
\end{proposition}
\begin{proof} We only prove (1); the proof of (2) is similar.
By Lemma \ref{main lemma} and noting $U_1=B$, we have
$$N_{A'}N_{U_1'}=\sum_{\substack{k\in[1,m+n]\\ a_{h+1,k}\geq1}}(-1)^{\sum_{i> m,j<k}a_{i,j}}\bsq^{-\sum_{j<k}a_{m+1,j}}\bsq^{\sum_{j>k}a_{m,j}}[\![a_{m,k}+1]\!]
N_{(A+E_{h,k}-E_{h+1,k})'}.$$ Recalling $\bsq=\up^2$ and $[\![a_{m,k}+1]\!]=\up^{2a_{m,k}}\overline{[\![a_{m,k}+1]\!]}$, assertion (1) is
equivalent to
$$\up^{d(X)-d(A)-d(U_1)}
\up^{-2\sum_{j<k}a_{m+1,j}}\up^{2\sum_{j>k}a_{m,j}}\up^{2a_{m,k}}=\up^{\sum_{j\geq
k}a_{m,j}+\sum_{j>k}a_{m+1,j}}.$$
 By
(\ref{$d(X)-d(A)-d(U_p)$}),
\begin{equation*}
\begin{aligned}
d(X)-d(A)-d(U_1)=-\sum_{j>k}a_{m,j}+2\sum_{j<k}a_{m+1,j}+\sum_{j>k}a_{m+1,j}-a_{m,k}.
\end{aligned}
\end{equation*}
Hence,
\begin{equation*}
\begin{aligned}
d(X)-d(A)-d(U_1)&-2\sum_{j<k}a_{m+1,j}+2\sum_{j>k}a_{m,j}+2a_{m,k}\\
&=\sum_{j>k}a_{m,j}+\sum_{j>k}a_{m+1,j}+a_{m,k},
\end{aligned}
\end{equation*}
as desired.
\end{proof}

\section{Uniform spanning sets over $\mathbb Q(\up)$}

We now use the multiplication formulas to show the structure of quantum group type for $\up$-Schur superalgebras. First, we describe a spanning set for $\bsS(m|n,r)$ whose members
are linear combination of $\xib_A$'s with the same off-diagonal entries. To achieve this goal, we need to extend further the multiplication formulas above to these elements, generalising \cite[Lem.~5.3]{BLM} to the super case.
 As is seen above, manipulating the sign is the key to the success of such generalisation.

For $A\in M(m|n)$, let
\begin{equation}\label{signAbar}
\bar{A}=\sum_{\substack{m+n\geq i> m\geq k\geq1
\\m<j<l\leq m+n}}a_{i,j}a_{k,l}.
\end{equation}
We will see in \S8 that this number occurs naturally in a triangular relation.

By the definition, we have the following.
\begin{lemma}\label{Abar} Let $\mu\in\mathbb N^{m+n}$ and $B=(b_{i,j})\in M(m|n)$.
\begin{enumerate}
\item If $A=\diag{(\mu)}$ or $A=E_{h,h+1}+\mu$ or
$A=E_{h+1,h}+\mu$, where $X+\mu:=X+\diag(\mu)$, then $\overline{A}=0$;
\item If $h\neq m$, the
$\overline{B+E_{h,k}-E_{h+1,k}}=\overline{B}$;
\item If $h=m$, then
$$\overline{B+E_{m,k}-E_{m+1,k}}=\begin{cases}
\overline{B},&\text{ if }k\leq m;\\
\overline{B}-\sum_{i\leq
m,j>m+1}b_{i,j},&\text{ if }m+1=k;\\
\overline{B}+\sum_{i>m,m<j<k}b_{i,j}-\sum_{i\leq
m,j>k}b_{i,j},&\text{ if }m+1<k.
\end{cases}$$
\end{enumerate}
\end{lemma}
\begin{proof} Write a matrix $A\in M(m|n)$ in the form $\biggl(\begin{matrix}A_{11}&A_{12}\\A_{21}&A_{22}\end{matrix}\biggr)$, where $A_{11}$ is $m\times m$. Then $\bar A$ is a sum of $a_{i,j}a_{k,l}$ with $a_{i,j}\in A_{22}$
and $a_{k,l}\in A_{12}$ and $j<l$. So (1) is clear from the definition and (2) for $h<m$ and (3) for $k\leq m$ are clear. For the case $h>m$ of (2), $(a_{h+1,k}-1)a_{i,j}$ and $(a_{h,k}-1)a_{i,j}$ occur in the sum. The remaining cases can be seen similarly.
\end{proof}
Consider now the set of matrices with zero diagonal:
$$M(m|n)^\pm=\{A=(a_{i,j})\in M(m|n)\mid a_{i,i}=0,
1\leq i\leq m+n\}.$$
For $A\in M(m|n)^{\pm}$ and
$\bsj=(j_1,j_2,\cdots,j_{m+n})\in\mathbb{Z}^{m+n}$, define
\begin{equation}\label{Ajr}
A(\bsj,r)=\begin{cases}\sum_{\substack{\lambda\in\Lambda(m|n,r-|A|)}}(-1)^{\overline{A+\lambda}}\up^{\lambda\centerdot\bsj}\xib_{A+\lambda},&\text{ if }|A|\leq r;\\
0,&\text{ otherwise,}\end{cases}
\end{equation}
where $\centerdot=\centerdot_s$ denotes the super (or signed) ``dot product'':
\begin{equation}\label{dot product}
\lambda\centerdot\bsj=\sum_{i=1}^{m+n}(-1)^{\hat{i}}\lambda_ij_i=\la_1j_1+\cdots+\la_mj_m-\la_{m+1}j_{m+1}-\cdots-\la_{m+n}j_{m+n}.
\end{equation}
Also, for notational simplicity in the multiplication formulas below, we define $A(\bsj,r)$ to be 0 if $A\not\in M(m|n,r)$. In particular, this is the case when $A$ has a negative entry or an entry $>1$ in the $m\times n$ or $n\times m$ block. Note that, since $\xi_{A+\la}$ has the same $\mathbb Z_2$-grading degree $\hat A$ (see \eqref{hatA}), $A(\bsj,r)$ is homogeneous of degree $\hat A$.

Let
$$\bsS(m|n,r)=\sS_\up(m|n,r)\otimes{\mathbb Q(\up)}.$$
We want to prove that $\bsS(m|n,r)$ is the span of
$\{A(\bsj,r)\mid A\in M(m|n)^\pm,\bsj\in\mathbb{Z}^n\}.$ The following
result shows that this span is closed under the multiplication by $O(\bsj,r)$.
\begin{proposition}\label{multiply by O}
For $\bsj,\bsj'\in\mathbb{Z}^{m+n}$ and $A=(a_{k,l})\in M(m|n)^{\pm}$, we have in $\bsS(m|n,r)$
\begin{itemize}
\item[(1)] $O(\bsj,r)A(\bsj',r)=\up^{\ro(A)\centerdot \bsj}A(\bsj+\bsj',r)$;
\item[(2)] $A(\bsj',r)O(\bsj,r)=\up^{\co(A)\centerdot \bsj}A(\bsj+\bsj',r)$.
\end{itemize}
\end{proposition}
\begin{proof}If $|A|>r$, then both sides are zero.
Assume now $|A|\leq r$. The proof of (1) follows from definition:
$$\aligned
\text{LHS}&=\sum_{\lambda\in{\Lambda(m|n,r)}}\up^{\lambda\centerdot\bsj}
\xib_{\lambda}\sum_{\mu\in\Lambda(m|n,r-|A|)}(-1)^{\overline{A+\mu}}\up^{\mu\centerdot\bsj'}\xib_{A+\mu}\\
&=\sum_{\mu\in\Lambda(m|n,r-|A|)}(-1)^{\overline{A+\mu}}\up^{(\mu+\ro(A))\centerdot
\bsj+\mu\centerdot\bsj'}\xib_{\mu+\ro(A)}\xib_{A+\mu}\text{ (by Lem~\ref{norm basis})}\\
&=\sum_{\mu\in\Lambda(m|n,r-|A|)}(-1)^{\overline{A+\mu}}\up^{\ro(A)\centerdot\bsj+\mu\centerdot(\bsj+\bsj')}\xib_{A+\mu}=\text{RHS}.
\endaligned$$
The proof of (2) is similar.
\end{proof}
Let $\bsS(m|n,0)=\mathbb Q(\up)$ and $O(\bsj,0)=1$ for all $\bsj\in \mathbb{Z}^{m+n}$.
By Proposition \ref{multiply by O}, similar to \cite[13.29]{DDPW}, we can
get the following result.
\begin{corollary} For all $r\geq0$, the set $\mathcal L_r=\{A(\bsj,r)\mid A\in M(m|n)^\pm,\bsj\in\mathbb{Z}^{m+n}\}$ spans
the $\up$-Schur superalgebra $\bsS(m|n,r)$ over $\mathbb Q(\up)$.
\end{corollary}
\begin{proof} Let $\mathcal{T}$ be the $\mathbb{Q}(\up)$-span of $\mathcal L_r$. Then, by Proposition \ref{multiply by O}, $\mathcal T$ is closed under multiplication by $O(\bsj,r)$ for all $\bsj\in\mathbb Z^{m+n}$. The first thing we need to prove is that $\xi_\la=\xi_{\diag(\la)}\in\mathcal T$ for all $\la\in\La(m|n,r)$. This is because
$$\xib_\lambda=\prod_{i=1}^{m+n}\left[O(\bse_i,r);0\atop
\lambda_i\right],$$
where $$\left[O(\bse_i,r);a\atop
\lambda_i\right]:=\prod_{j=1}^{\lambda_i}\frac{O(\bse_i,r)\up_i^{a-j+1}-O(-\bse_i,r)\up_i^{a+j-1}}{\up_i^j-\up_i^{-j}}\in
\mathcal T.$$
Next, for any $A\in M(m|n,r)$,
let $A^\pm$ be the matrix obtained by replacing all the diagonal
entries of $A$ by $0$. Then, putting $\lambda=\co(A)$ and
$\mu=\ro(A)$, we have $\xib_A=\xib_\mu
A^\pm(\mathbf{0},r)=A^\pm(\mathbf{0},r)\xib_\lambda\in\mathcal T$, by again Proposition \ref{multiply by O}. Hence,
$\mathcal{T}=\bsS(m|n,r)$.
For a more detailed argument, see the proof of
\cite[Lem.~3.29]{DDPW}.
\end{proof}

By these uniform spanning sets $\sL_r$, we define
\begin{equation}\label{A(j)}
\aligned
\bsS(m|n)&:=\prod_{r\geq0}\bsS(m|n,r)\quad\text{and } \\
\fA(m|n)&:=\text{span}\{A(\bsj)\mid A\in M(m|n)^\pm,\bsj\in\mathbb{Z}^{m+n}\},\\
\text{where}\qquad A(\bsj)&:=\sum_{r\geq0} A(\bsj,r)\in\bsS(m|n).\endaligned
\end{equation}
Note that, if we assign to $A(\bsj)$ the grading degree $\hat A$, then $\fA(m|n)$ is a superspace.
\begin{lemma}\label{Bruhat}
The set $\{A(\bsj)\mid A\in M(m|n)^\pm,\bsj\in\mathbb{Z}^{m+n}\}$ is linearly independent and forms a basis of homogeneous elements for the superspace $\fA(m|n)$.
\end{lemma}
\begin{proof}The proof is almost identical to the proof of \cite[Prop.~4.1(2)]{DF09} if we replace the vector $\boldsymbol 1=(1,1,\ldots,1)$ used in the last paragraph by
$\boldsymbol 1^\pm=(\underbrace{1,\ldots,1}_m,\underbrace{-1,\ldots,-1}_n).$
\end{proof}
We will eventually prove that the subspace $\fA(m|n)$ is a subalgebra isomorphic to the quantum supergroup $\bfU(\mathfrak{gl}_{m|n})$.
We now show that it is closed under the multiplication by ``generators''.

First, the formulas given in Proposition \ref{multiply by O} gives, by taking sum  over all $r\geq0$ (i.e.~removing $r$ throughout), the corresponding formulas in $\bsS(m|n)$.
\begin{proposition}\label{multiply O}
For $\bsj,\bsj'\in\mathbb{Z}^{m+n}$ and $A=(a_{k,l})\in M(m|n)^{\pm}$, we have in $\bsS(m|n)$
\begin{itemize}
\item[(1)] $O(\bsj)A(\bsj')=\up^{\ro(A)\centerdot \bsj}A(\bsj+\bsj')$;
\item[(2)] $A(\bsj')O(\bsj)=\up^{\co(A)\centerdot \bsj}A(\bsj+\bsj')$.
\end{itemize}
In particular, $\fA(m|n)$ is closed under the multiplication by $O(\bsj)$.
\end{proposition}

 Note that the sign $(-1)^{\overline{A+\lambda}}$ in \eqref{Ajr} disappears in the following cases:
$$E_{h,h+1}(\bsj,r)=\sum_{\lambda\in\Lambda(m|n,r-1)}\up^{\lambda\centerdot\bsj}\xib_{E_{h,h+1}+\lambda}
\mbox{ and }
E_{h+1,h}(\bsj,r)=\sum_{\lambda\in\Lambda(m|n,r-1)}\up^{\lambda\centerdot\bsj}\xib_{E_{h+1,h}+\lambda}.$$

 For given $1\leq h, k<m+n$ and $A\in M(m|n)$, let
\begin{equation}\label{$f(i)$ and $f'(i)$}
\aligned
\alpha_h&=\bse_h-\bse_{h+1},\\
\beta_h&=-\bse_h-\bse_{h+1},\\
f(k)&=f_h(\bse_k,A)=\sum_{j\geq k}a_{h,j}-(-1)^{\vep}\sum_{j>k}a_{h+1,j},\text{ and }\\
f'(k)&=f'_h(\bse_k,A)=\sum_{j\leq k}a_{h+1,j}-(-1)^{\vep}\sum_{j<k}a_{h,j},\quad\text{where}\\
\vep&=\vep_{h,h+1}=\widehat h+\widehat{h+1}\in\{0,1\}
\endaligned
\end{equation}
(see \eqref{beta_h1}, \eqref{gamma_h1}, and \eqref{beta_m gamma_m}).
 For $A\in M(m|n)$ and $k\in[1,m+n]$, define
\begin{equation}\label{sigma(k)}
\sigma(k)=\sigma(k,A)=\begin{cases}\sum_{i>m,j<k}a_{i,j},&\text{ if }k\leq m;\\
\sum_{i\leq
m,j>k}a_{i,j}+\sum_{i>m,j\leq
m}a_{i,j},&\text{ if }k\geq m+1.
\end{cases}
\end{equation}
So, for $k\leq m$, $\sigma(k)$ is the entry sum of the submatrix with the upper right corner $(m+1, k-1)$,
while, for $k\geq m+1$, $\sigma(k)$ is the entry sum of the left bottom $n\times m$ submatrix and the submatrix with the lower left corner $(m,k+1)$.

 Propositions
\ref{integral basis multiplication} for $p=1$ and \ref{integral basis multiplication $h=m$} can now be combined and extended as follows.

\begin{proposition}\label{h neq m}
Maintain the notation introduced above. For given  $h\in[1,m+n)$ and $A=(a_{k,l})\in M(m|n)^{\pm}$, the following multiplication formulas of homogeneous elements hold in $\bsS(m|n,r)$ for all $r\geq0$:
\begin{equation*}
\begin{aligned}
(1)\;\;E_{h,h+1}&(\mathbf{0},r)A(\bsj,r)\\
& =\sum_{k<h,a_{h+1,k}\geq
1}(-1)^{\vep\sigma(k)}\up_h^{f(k)}\overline{[\![a_{h,k}+1]\!]}_{\up_h^2}(A+E_{h,k}-E_{h+1,k})(\bsj+\alpha_h,r)\\
&\sum_{k>h+1,a_{h+1,k}\geq
1}(-1)^{\vep\sigma(k)}\up_h^{f(k)}\overline{[\![a_{h,k}+1]\!]}_{\up_h^2}(A+E_{h,k}-E_{h+1,k})(\bsj,r)\\
&+(-1)^{\vep\sigma(h+1)}\up_h^{f(h+1)+(-1)^{\vep}j_{h+1}}\overline{[\![a_{h,h+1}+1]\!]}_{\up_h^2}(A+E_{h,h+1})(\bsj,r)\\
&+(-1)^{\vep\sigma(h)}\up_h^{f(h)-j_h-1}\frac{(A-E_{h+1,h})(\bsj+\alpha_h,r)-(A-E_{h+1,h})(\bsj+\beta_h,r)}{1-\up_h^{-2}};\\\end{aligned}
\end{equation*}
\begin{equation*}
\begin{aligned}
(2)\;\;&E_{h+1,h}(\mathbf{0},r)A(\bsj,r)\\
&=\sum_{k<h,a_{h,k}\geq
1}(-1)^{\vep\sigma(k)}\up_{h+1}^{f'(k)}\overline{[\![a_{h+1,k}+1]\!]}_{\up_{h+1}^2}(A-E_{h,k}+E_{h+1,k})(\bsj,r)\\
&+\sum_{i>h+1,a_{h,k}\geq
1}(-1)^{\vep\sigma(k)}\up_{h+1}^{f'(k)}\overline{[\![a_{h+1,k}+1]\!]}_{\up_{h+1}^2}(A-E_{h,k}+E_{h+1,k})(\bsj-\alpha_h,r)\\
&+(-1)^{\vep\sigma(h)}\up_{h+1}^{f'(h)+(-1)^\vep j_h}\overline{[\![a_{h+1,h}+1]\!]}_{\up_{h+1}^2}(A+E_{h+1,h})(\bsj,r)\\
&+(-1)^{\vep\sigma(h+1)}\up_{h+1}^{f'(h+1)-j_{h+1}-1}\frac{(A-E_{h,h+1})(\bsj-\alpha_h,r)-(A-E_{h,h+1})(\bsj+\beta_h,r)}{1-\up_{h+1}^{-2}}.
\end{aligned}
\end{equation*}
Moreover, summing  over all $r\geq0$ (i.e.~removing $r$ throughout)  gives the corresponding formulas in $\bsS(m|n)$. In particular, $\fA(m|n)$ is closed under multiplication by $E_{h,h+1}(\mathbf{0})$ and $E_{h+1,h}(\mathbf{0})$.
\end{proposition}

\begin{proof} If $|A|>r$, then there is nothing to prove.\footnote{If $|A|=r+1$, then $|A-E_{h+1,h}|=r$ and so $(A-E_{h+1,h})(\bsj)=\xib_{A-E_{h+1,h}}$. Thus, the last summand in the RHS of (1) is also zero.}

We now only prove (1); the proof of (2) is symmetric. First, we prove the $h\neq m$ case. In this case, the sign $(-1)^{\vep\sigma(h)}$ and $(-1)^\vep$ become 1.

$$\aligned
E_{h,h+1}(&\mathbf{0},r)A(\bsj,r)=\sum_{\lambda\in\Lambda(m|n,r-1)}\xi_{E_{h,h+1}+\lambda}\sum_{\mu\in\Lambda(m|n,r-|A|)}(-1)^{\overline{A+\mu}}\up^{\mu\centerdot
\bsj}\xi_{A+\mu}\\
&=\sum_{\mu\in\Lambda(m|n,r-|A|)}(-1)^{\overline{A+\mu}}\up^{\mu\centerdot
\bsj}\xib_{E_{h,h+1}+\ro(A)+\mu-\bse_{h+1}}\xib_{A+\mu}
\text{ (by Lem~\ref{norm basis}(1))}\\
&=\sum_{\mu\in\Lambda(m|n,r-|A|)}\sum_{\substack{k\in[1,m+n]\\a^\mu_{h+1,k}\geq
1}} \up_h^{f_h(\bse_k,A+\mu)}\overline{[\![a^\mu_{h,k}+1]\!]}_{\up_h^2}
  (-1)^{\overline{A+\mu}}\up^{\mu\centerdot
\bsj}\xib_{A+E_{h,k}-E_{h+1,k}+\mu}\\
&\qquad(\text{by Proposition
\ref{integral basis multiplication}, where $A+\mu=(a^\mu_{i,j})$}).\\
\endaligned
$$
By Lemma \ref{Abar}, for $h\neq m$,
$\overline{A+E_{h,k}-E_{h+1,k}+\mu}=\overline{A+\mu}$. Note also that $\up_h^{\mu_h-\mu_{h+1}}=\up^{\mu\centerdot\alpha_h}$, $\up_h^{-\mu_h-\mu_{h+1}}=\up^{\mu\centerdot\beta_h}$, and $a_{j,k}^\mu=a_{j,k}$ whenever
$j\neq k$. Moreover, by definition,
$f_h(\bse_k,A+\mu)=\sum_{j\geq
k}a^\mu_{h,j}-\sum_{j>k}a^\mu_{h+1,j}$ for $h\neq m$. Thus,
\begin{equation}\label{beta_h}
f_h(\bse_k,A+\mu)=\begin{cases}f_h(\bse_k,A),&\text{ if }k>h;\\
f_h(\bse_k,A)+\mu_h-\mu_{h+1},&\text{ if }k\leq h.\end{cases}
\end{equation}
Hence,
\begin{equation*}
\begin{aligned}
\text{LHS}&=\sum_{k>h+1,a_{h+1,k}\geq
1}\up_h^{f(k)}\overline{[\![a_{h,k}+1]\!]}_{\up_h^2}(A+E_{h,k}-E_{h+1,k})(\bsj,r)\\
&\quad\,+\sum_{k<h,a_{h+1,k}\geq
1}\up_h^{f(k)}\overline{[\![a_{h,k}+1]\!]}_{\up_h^2}(A+E_{h,k}-E_{h+1,k})(\bsj+\alpha_h,r)\\
&\quad\,+Y_h+Y_{h+1},
\end{aligned}
\end{equation*}
where, noting $a_{i,i}=0$,
$$\aligned
Y_{h+1}&=\sum_{\substack{\mu\in\Lambda(m|n,r-|A|)\\\mu_{h+1}\geq1}}(-1)^{\overline{A+\mu}}\up^{\mu\centerdot\bsj}
\up_h^{f(h+1)}\overline{[\![a_{h,h+1}+1]\!]}_{\up_h^2}\xi_{A+E_{h,h+1}-E_{h+1,h+1}+\mu}\\
&=\up_h^{f(h+1)+j_{h+1}}\overline{[\![a_{h,h+1}+1]\!]}_{\up_h^2}
\sum_{\mu'\in\Lambda(m|n,r-|A|-1)}(-1)^{\overline{A+\mu}}\up^{\mu'\centerdot\bsj}
\xi_{A+E_{h,h+1}+\mu'}\\
&=\up_h^{f(h+1)+j_{h+1}}\overline{[\![a_{h,h+1}+1]\!]}_{\up_h^2}(A+E_{h,h+1})(\bsj,r),
\endaligned$$
and, if $a_{h+1,h}\geq1$,
$$\aligned
Y_h&=\sum_{\mu\in\Lambda(m|n,r-|A|)}(-1)^{\overline{A+\mu}}\up^{\mu\centerdot\bsj}\up_h^{f(h)+\mu_h-\mu_{h+1}}\overline{[\![a_{h,h}+\mu_h+1]\!]}_{\up_h^2}\xib_{A+E_{h,h}-E_{h+1,h}+\mu}\\
&=\sum_{\mu\in\Lambda(m|n,r-|A|)}(-1)^{\overline{A+\mu}}\up^{\mu\centerdot\bsj}\up_h^{f(h)+\mu_h-\mu_{h+1}}\overline{[\![\mu_h+1]\!]}_{\up_h^2}\xib_{A+E_{h,h}-E_{h+1,h}+\mu}\\
&=\up_h^{f(h)}
\sum_{\mu\in\Lambda(m|n,r-|A|)}(-1)^{\overline{A+\mu}}\up^{\mu\centerdot\bsj}
\frac{\up_h^{\mu_h-\mu_{h+1}}-\up_h^{-\mu_h-\mu_{h+1}-2}}{1-\up_h^{-2}}\xib_{A-E_{h+1,h}+\mu+\bse_h}\\
&=\up_h^{f(h)-j_h-1}\sum_{\mu\in\Lambda(m|n,r-|A|)}(-1)^{\overline{A+\mu}}\frac{\up^{(\mu+\bse_h)\centerdot(\bsj+\alpha_h)}-\up^{(\mu+\bse_h)\centerdot(\bsj+\beta_h)}}{1-\up_h^{-2}}\xib_{A-E_{h+1,h}+\mu+\bse_h}\\
&=\up^{f(h)-j_h-1}\frac{(A-E_{h+1,h})(\bsj+\alpha_h,r)-(A-E_{h+1,h})(\bsj+\beta_h,r)}{1-\up_h^{-2}},
\endaligned$$
proving (1) for $h\neq m$.

We now prove the $h=m$ case. By definition,
\begin{equation*}
\begin{aligned} E_{m,m+1}(&\mathbf{0},r)A(\bsj,r)=\sum_{\lambda\in\Lambda(m|n,r)}\xi_{E_{m,m+1}+\lambda}\sum_{\mu\in\Lambda(m|n,r-|A|)}(-1)^{\overline{A+\mu}}\up^{\mu\centerdot
\bsj}\xi_{A+\mu}\\
&=\sum_{\mu\in\Lambda(m|n,r-|A|)}(-1)^{\overline{A+\mu}}\up^{\mu\centerdot\bsj}\xib_{E_{m,m+1}+\ro(A)+\mu-\co(E_{m,m+1})}\xib_{A+\mu}\\
&=\sum_{\mu\in\Lambda(m|n,r-|A|)}\sum_{\substack{k\in[1,m+n]\\a^\mu_{m+1,k}\geq
1}}(-1)^{\overline{A+\mu}}\up^{\mu\centerdot\bsj}(-1)^{\sum_{i>m,j<k}a^\mu_{i,j}}\up^{f_m(\bse_k,A+\mu)}\\
&\quad\cdot\overline{[\![a^\mu_{m,k}+1]\!]}_{\up^2}\xib_{A+E_{m,k}-E_{m+1,k}+\mu}\; (\text{by Proposition \ref{integral basis multiplication $h=m$}, where $A+\mu=(a^\mu_{i,j})$}).
\end{aligned}
\end{equation*}
Now, by \eqref{beta_m gamma_m}, \eqref{beta_h} becomes, for $h=m$,
\begin{equation}\label{beta_m}
f_m(\bse_k,A+\mu)=\begin{cases}f_m(\bse_k,A)+\mu_m+\mu_{m+1},&\text{ if }k\leq m;\\
f_m(\bse_k,A),&\text{ if }k>m.\\
\end{cases}
\end{equation}

If $k<m$ and $a_{m+1,k}=a_{m+1,k}^\mu\geq 1$, then
$\overline{A+\mu}=\overline{A+E_{m,k}-E_{m+1,k}+\mu}$ by Lemma \ref{Abar} and, by \eqref{dot product}, $\mu\centerdot\alpha_{m}=\mu_m+\mu_{m+1}$. Thus, in this case, the term

\begin{equation*}
\begin{aligned}
&(-1)^{\overline{A+\mu}}\up^{\mu\centerdot\bsj}(-1)^{\sum_{i>m,j<k}a^\mu_{i,j}}\up^{f_m(\bse_k,A)}\overline{[\![a^\mu_{m,k}+1]\!]}_{\up^2}\xib_{A+E_{m,k}-E_{m+1,k}+\mu}\\
&=(-1)^{\overline{A+\mu}}\up^{\mu\centerdot\bsj}(-1)^{\sum_{i>m,j<k}a_{i,j}}\up^{f(k)+\mu_m+\mu_{m+1}}\overline{[\![a_{m,k}+1]\!]}_{\up^2}\xib_{A+E_{m,k}-E_{m+1,k}+\mu}\\
&=(-1)^{\sigma(k)}\up^{f(k)}\overline{[\![a_{m,k}+1]\!]}_{\up^2}(-1)^{\overline{A+E_{m,k}-E_{m+1,k}+\mu}}\up^{\mu\centerdot(\bsj+\alpha_m)}\xib_{A+E_{m,k}-E_{m+1,k}+\mu}.
\end{aligned}
\end{equation*}
Substituting gives the first sum in the formula.

If $k=m$ and $a_{m+1,m}\geq1$, then we have
\begin{equation*}
\begin{aligned}
&(-1)^{\overline{A+\mu}}\up^{\mu\centerdot\bsj}(-1)^{\sum_{i>m,j<m}a^\mu_{i,j}}\up^{f_m(\bse_m,A)}\overline{[\![a^\mu_{m,m}+1]\!]}_{\up^2}\xib_{A+E_{m,m}-E_{m+1,m}+\mu}\\
&=\up^{\mu\centerdot\bsj}(-1)^{\overline{A+\mu}}(-1)^{\sum_{i>m,j<m}a_{i,j}}\up^{f(m)+\mu_m+\mu_{m+1}}\overline{[\![\mu_m+1]\!]}_{\up^2}\xib_{A+E_{m,m}-E_{m+1,m}+\mu}\\
&=(-1)^{\sum_{i>m,j<m}a_{i,j}}(-1)^{\overline{A+\mu}}
\up^{\mu\centerdot\bsj}\up^{f(m)+\mu_m+\mu_{m+1}}\overline{[\![\mu_m+1]\!]}_{\up^2}\xib_{A+E_{m,m}-E_{m+1,m}+\mu}\\
&=(-1)^{\sigma(m)}\up^{f(m)-j_m-1}(-1)^{\overline{A+\mu}}\frac{\up^{(\mu+\bse_m)\centerdot(\bsj+\alpha_m)}-\up^{(\mu+\bse_m)\centerdot{(\bsj+\beta_m)}}}{1-\up^{-2}}\xib_{A-E_{m+1,m}+\mu+\bse_m}.
\end{aligned}
\end{equation*}
Applying Lemma \ref{Abar} and substituting give the last term in the formula.

If $k>m+1$ and $a_{m+1,k}\geq1$ then, by Lemma \ref{Abar}, $\overline{A+E_{m,k}-E_{m+1,k}+\mu}=\overline{A+\mu}+\sum_{i>m,m<j<k}a^\mu_{i,j}-\sum_{i\leq
m,j>k}a^\mu_{i,j}$, and

\begin{equation*}
\begin{aligned}
&\up^{\mu\centerdot\bsj}(-1)^{\overline{A+\mu}}(-1)^{\sum_{i>m,j<k}a^\mu_{i,j}}\up^{f_m(\bse_k,A)}\overline{[\![a^\mu_{m,k}+1]\!]}_{\up^2}\xib_{A+E_{m,k}-E_{m+1,k}+\mu}\\
&=\up^{\mu\centerdot\bsj}(-1)^{\sigma(k)}(-1)^{\overline{A+E_{m,k}-E_{m+1,k}+\mu}}\up^{f(k)}\overline{[\![a_{m,k}+1]\!]}_{\up^2}\xib_{A+E_{m,k}-E_{m+1,k}+\mu}\\
&=(-1)^{\sigma(k)}\up^{f(k)}\overline{[\![a_{m,k}+1]\!]}_{\up^2}(-1)^{\overline{A+E_{m,k}-E_{m+1,k}+\mu}}\up^{\mu\centerdot\bsj}\xib_{A+E_{m,k}-E_{m+1,k}+\mu}.
\end{aligned}
\end{equation*}
So substituting gives the second sum in the formula.

Finally, if $k=m+1$ and $\mu_{m+1}=a^\mu_{m+1,m+1}\geq1$, by Lemma \ref{Abar},
$\overline{A+\mu}=\overline{A+E_{m,m+1}+\mu-\bse_{m+1}}+\sum_{i\leq
m,j>m+1}a_{i,j}$ and the corresponding term becomes
\begin{equation*}
\begin{aligned}
&\up^{\mu\centerdot\bsj}(-1)^{\overline{A+\mu}}(-1)^{\sum_{i>m,j<m+1}a^\mu_{m+1,j}}\up^{f_m(\bse_{m+1},A)}\overline{[\![a^\mu_{m,m+1}+1]\!]}_{\up^2}\xib_{A+E_{m,m+1}-E_{m+1,m+1}+\mu}\\
=&(-1)^{\sigma(m+1)}\up^{f(m+1)-j_{m+1}}\overline{[\![a_{m,m+1}+1]\!]}_{\up^2}(-1)^{\overline{A+E_{m,m+1}+\mu-\bse_{m+1}}}\up^{(\mu-\bse_{m+1})\centerdot\bsj}\xib_{A+E_{m,m+1}+\mu-\bse_{m+1}}.
\end{aligned}
\end{equation*}
Now substituting gives the second last term.
\end{proof}

\section{Generators and relations for the quantum supergroup $\bfU(\mathfrak{gl}_{m|n})$}

We now use the formulas in $\bsS(m|n)$ given in Propositions \ref{multiply O} and \ref{h neq m} to derive the defining relations for the quantum supergroup $\bfU(\mathfrak{gl}_{m|n})$ inside $\bsS(m|n)$.
First, we recall the definition from \cite{Z}. Recall also the definition of the super commutator on homogeneous elements of a supernalgebra with parity function $\hat\  \,$:
$$[X,Y]=XY-(-1)^{\hat X\hat Y}YX.$$
\begin{definition} The quantum enveloping superalgebra
  $\bfU(\mathfrak{gl}_{m|n})$ is the algebra over $\mathbb Q(\up)$ generated by
$$\sfK_a,\sfK_a^{-1}, \sfE_{h,h+1},\sfE_{h+1,h},\quad(1\leq a\leq m+n,1\leq h<m+n)$$
which satisfy the following relations:
\begin{itemize}
\item[(QS1)] $\sfK_a\sfK_a^{-1}=1,\sfK_a\sfK_b=\sfK_b\sfK_a;$\vspace{.1cm}

\item[(QS2)] $\sfK_a
\sfE_{h,h+1}=\up_a^{\delta_{a,h}-\delta_{a,h+1}}\sfE_{h,h+1}\sfK_a,
\sfK_a\sfE_{h+1,h}=\up_a^{-\delta_{a,h}+\delta_{a,h+1}}\sfE_{h+1,h}\sfK_a;$\vspace{.1cm}

\item[(QS3)]
$[\sfE_{h,h+1},\sfE_{k+1,k}]=\delta_{h,k}\frac{\sfK_{h}\sfK_{h+1}^{-1}-\sfK_{h}^{-1}\sfK_{h+1}}{\up_h-\up_h^{-1}};$\vspace{.1cm}

\item[(QS4)] $\sfE_{h,h+1}\sfE_{k,k+1}=\sfE_{k,k+1}\sfE_{h,h+1},\;\sfE_{h+1,h}\sfE_{k+1,k}=\sfE_{k+1,k}\sfE_{h+1,h},$
 where $|k-h|>1;$

\item[(QS5)] For $h\neq m$,
$$\aligned&\sfE_{h,h+1}^2\sfE_{h+1,h+2}-(\up_h+\up_h^{-1})\sfE_{h,h+1}\sfE_{h+1,h+2}\sfE_{h,h+1}+\sfE_{h+1,h+2}\sfE_{h,h+1}^2=0,\\
&\sfE^2_{h,h+1}\sfE_{h-1,h}-(\up_h+\up_h^{-1})\sfE_{h,h+1}\sfE_{h-1,h}\sfE_{h,h+1}+\sfE_{h-1,h}\sfE^2_{h,h+1}=0,\\
&\sfE_{h+1,h}^2\sfE_{h+2,h+1}-(\up_h+\up_h^{-1})\sfE_{h+1,h}\sfE_{h+2,h+1}\sfE_{h+1,h}+\sfE_{h+2,h+1}\sfE_{h+1,h}^2=0,\\
&\sfE_{h+1,h}^2\sfE_{h,h-1}-(\up_h+\up_h^{-1})\sfE_{h+1,h}\sfE_{h,h-1}\sfE_{h+1,h}+\sfE_{h,h-1}\sfE_{h+1,h}^2=0;
\endaligned$$
\item[(QS6)] $\sfE_{m,m+1}^2=\sfE_{m+1,m}^2=0$ and
$[\sfE_{m,m+1},\sfE_{m-1,m+2}]=[\sfE_{m+1,m},\sfE_{m+2,m-1}]=0.$
\end{itemize}
Here, only $\sfE_{m,m+1},\sfE_{m+1,m}$ are odd generators and others are even generators.
Moreover, the quantum root vectors $\sfE_{a,b}$,
for $a,b\in[1,m+n]$ with $|a-b|>1$, is defined recursively as follows
\begin{equation}\label{qbrackets}
\sfE_{a,b}=\left\{\begin{aligned}
&\sfE_{a,c}\sfE_{c,b}-\up_c\sfE_{c,b}\sfE_{a,c}, \mbox{  if } a>b;\\
&\sfE_{a,c}\sfE_{c,b}-\up^{-1}_c\sfE_{c,b}\sfE_{a,c}, \mbox{  if } a<b,
\end{aligned}
\right.
\end{equation}
where $c$ can be taken to be an arbitrary index strictly between $a$ and $b$, and $\sfE_{a,b}$ is homogeneous of degree $\hat E_{a,b}:=\hat{a}+\hat{b}$ .
\end{definition}
By the quantum Schur-Weyl duality (\cite[Thm.~4.4]{Mi}, \cite[\S8]{DR}, \cite[\S4]{TK}),
it is known that  there are algebra homomorphisms from $\bfU(\mathfrak{gl}_{m|n})$ to $\bsS(m|n,r)$ for all $r\geq0$. These homomorphisms induce an algebra homomorphism  from $\bfU(\mathfrak{gl}_{m|n})$ to the direct product $\bsS(m|n)$.
We now use the multiplication formulas given in Propositions \ref{multiply O} and \ref{h neq m} to explicitly define a $\mathbb Q(\up)$-algebra homomorphism
$\eta:\bfU(\mathfrak{gl}_{m|n})\longrightarrow \bsS(m|n)$.

Let $M(m|n)^+$(resp. $M(m|n)^-$) be the subset of $M(m|n)$
consisting of those matrices $(a_{i,j})$ with $a_{i,j}=0$, for all
$i\geq j$ (resp. $i\leq j$). 
Recall also the symmetric Gaussian numbers
$[k]_{\up_h}=\frac{\up_h^{k}-\up_h^{-k}}{\up_h-\up_h^{-1}}$.

For the convenience of computation below, we list some special cases of Proposition \ref{h neq m}.

\begin{lemma}\label{integral generators} For $1\leq h<m+n$, let $\vep=\widehat h+\widehat{h+1}$ and, for $A^+=(a_{i,j})\in M(m|n)^+$ and $A^-=(a_{i,j})\in M(m|n)^-$, let (see \eqref{sigma(k)})
$$\sigma^+(k):=\sigma(k,A^+)=\sum_{i\leq m,j>k}a_{i,j}\;\text{ and }\;\sigma^-(k)=\sigma(k,A^-)=\sum_{i\geq m+1,j<k}a_{i,j}.$$
Then the following multiplication formulas hold in $\bsS(m|n)$:
\begin{equation*}
\begin{aligned}
\!\!(1)\; E_{h,h+1}(\mathbf{0})A^+(\mathbf{0})&=(-1)^{\vep\sigma^+(h+1)}\up_h^{f(h+1)}\overline{[\![a_{h,h+1}+1]\!]}_{\up^2_h}(A^++E_{h,h+1})(\mathbf{0})\\
&\!\!\!\!\!+\sum_{k>h+1,a_{h+1,k}\geq
1}(-1)^{\vep\sigma^+(k)}\up_h^{f(k)}\overline{[\![a_{h,k}+1]\!]}_{\up^2_h}(A^++E_{h,k}-E_{h+1,k})(\mathbf{0});\\
(2)\;E_{h+1,h}(\mathbf{0})A^-(\mathbf{0})&=(-1)^{\vep\sigma^-(h)}\up_{h+1}^{f'(h)}\overline{[\![a_{h+1,h}+1]\!]}_{\up^2_{h+1}}(A^-+E_{h+1,h})(\mathbf{0})\\
&\sum_{k<h,a_{h,k}\geq
1}(-1)^{\vep\sigma^-(k)}\up_{h+1}^{f'(k)}\overline{[\![a_{h+1,k}+1]\!]}_{\up^2_{h+1}}(A^--E_{h,k}+E_{h+1,k})(\mathbf{0}).
\end{aligned}
\end{equation*}
In particular,
for $k\geq0$, we have
\begin{equation*}
\begin{aligned}
(3)\;&E_{h,h+1}(\mathbf{0})^k=[k]^!_{\up_h}(kE_{h,h+1})(\mathbf{0})\;\text{ and } \;E_{h+1,h}(\mathbf{0})^k=[k]^!_{\up_{h+1}}(kE_{h+1,h})(\mathbf{0}),
&\mbox{ if } h\neq m;\\
(4)\;&E_{m,m+1}(\mathbf{0})^k=0\;\text{ and }\; E_{m+1,m}(\mathbf{0})^k=0,&\!\!\!\mbox{ if } k>1.\\
\end{aligned}
\end{equation*}
\end{lemma}
\begin{proof} Formulas (1) and (2) follows from Proposition \ref{h neq m} and (4) was seen from Theorem \ref{U_i and L_i}.
We prove (3) by induction on $k$. By (1),
\begin{equation*}
E_{h,h+1}(\mathbf{0})E_{h,h+1}(\mathbf{0})=\up_h\overline{[\![2]\!]}_{\up^2_h}(2E_{h,h+1})(\mathbf{0})=[2]_{\up_h}(2E_{h,h+1})(\mathbf{0}).
\end{equation*}
By induction and (1),
\begin{equation*}
\begin{aligned}
(E_{h,h+1}(\mathbf{0}))^{k+1}&=[k]^!_{\up_h}E_{h,h+1}(\mathbf{0})(kE_{h,h+1})(\mathbf{0})\\
&=[k]^!_{\up_h}\up_h^{k}\overline{[\![k+1]\!]}_{\up_h^2}((k+1)E_{h,h+1})(\mathbf{0})\\\
&=[k+1]^!_{\up_h}((k+1)E_{h,h+1})(\mathbf{0}).
\end{aligned}
\end{equation*}
The second formula in (3) can be proved similarly by applying (2).
\end{proof}

Let
$$\ugk_i=O(\bse_i),\quad \ugk_i^{-1}=O(-\bse_i),\quad
E_{h}=E_{h,h+1}(\mathbf{0}),\quad \text{and}\quad F_h=E_{h+1,h}(\mathbf{0}).$$

\begin{theorem}\label{quantum serre relations} There is a $\mathbb Q(\up)$-algebra homomorphism
$$\eta:\bfU(\mathfrak{gl}_{m|n})\longrightarrow \bsS(m|n)$$
sending $\sfE_{h,h+1},\sfE_{h+1,h}$ and $\sfK_a^{\pm1}$ to $E_h,F_h$ and $\ugk_a^{\pm1}$, respectively.
\end{theorem}

\begin{proof} It is enough to show that $\eta$ preserves all relations (QS1)--(QS6).

\noindent
(QS1) By Proposition \ref{multiply O}, we have
$$\ugk_a\ugk_a^{-1}=O(\bse_a)O(-\bse_a)=O(\mathbf{0})=1,$$
and
$$\ugk_a\ugk_b=O(\bse_a)O(\bse_b)=O(\bse_{a+b})=O(\bse_b)O(\bse_a)=\ugk_b\ugk_a.$$

\noindent
(QS2) For $a\in [1,m+n]$, by Proposition \ref{multiply O},
we have
$$\aligned
\ugk_aE_h&=O(\bse_a)E_{h,h+1}(\mathbf{0})=\up^{\ro(E_{h,h+1})\centerdot
\bse_a}E_{h,h+1}(\bse_a)\\
E_h\ugk_a&=E_{h,h+1}(\mathbf{0})O(\bse_a)=\up^{\co(E_{h,h+1})\centerdot
\bse_a}E_{h,h+1}(\bse_a).
\endaligned$$ Since
$\co(E_{h,h+1})=\bse_{h+1}$ and
$\ro(E_{h,h+1})=\bse_h$, by \eqref{dot product},
$$\ugk_aE_h=\up^{(\bse_h-\bse_{h+1})\centerdot\bse_a}E_h\ugk_a=\up_a^{\delta_{a,h}-\delta_{a,h+1}}E_h\ugk_a.$$
Similarly, one proves the second formula.

\noindent
(QS3) Let $\hat{O}=0,\widehat{E}_{i,j}=\hat{i}+\hat{j}$.
By definition, for $1\leq h,k\leq m+n-1$,
$$[E_h,F_k]=E_hF_k-(-1)^{\widehat{E}_{h,h+1}\widehat{E}_{k+1,k}}F_kE_h.$$
(i) If $h=k\neq m$, by Proposition \ref{h neq m}, then
$$E_hF_h=(E_{h,h+1}+E_{h+1,h})(\mathbf{0})+\frac{O(\alpha_h)-O(\beta_h)}{\up_h-\up_h^{-1}}$$
and
$$F_hE_h=(E_{h,h+1}+E_{h+1,h})(\mathbf{0})+\frac{O(-\alpha_h)-O(\beta_h)}{\up_h-\up_h^{-1}}.$$

Since $\widehat{E}_{h,h+1}=0$,
$$[E_h,F_h]=\frac{O(\alpha_h)-O(-\alpha_h)}{\up_h-\up_h^{-1}}=
\frac{\ugk_h\ugk_{h+1}^{-1}-\ugk_h^{-1}\ugk_{h+1}}{\up_h-\up_h^{-1}}.$$
(ii) If $h=k=m$, then $\widehat{E}_{m,m+1}=\widehat{E}_{m+1,m}=1$.
By  Proposition \ref{h neq m},
$$E_mF_m=-(E_{m,m+1}+E_{m+1,m})(\mathbf{0})+\frac{O(\alpha_m)-O(\beta_m)}{\up-\up^{-1}}$$
and
$$F_m E_m=(E_{m,m+1}+E_{m+1,m})(\mathbf{0})-\frac{O(-\alpha_m)-O(\beta_m)}{\up-\up^{-1}}.$$
Hence,
$$\aligned[]
[ E_m,F_m]&= E_mF_m+F_m E_m\\
&=\frac{O(\alpha_m)-O(-\alpha_m)}{\up-\up^{-1}}=\frac{\ugk_m\ugk^{-1}_{m+1}-\ugk_m^{-1}\ugk_{m+1}}{\up-\up^{-1}}.
\endaligned$$
(iii) If $h\neq k$, then
$\widehat{E}_{h,h+1}\widehat{E}_{k+1,k}=0.$
By  Proposition \ref{h neq m},
$$E_hF_k=(E_{h,h+1}+E_{k+1,k})(\mathbf{0})=F_kE_h.$$
Thus
$$[E_h,F_k]=E_hF_k-F_kE_h=0.$$

\noindent
(QS4)
Similarly, for $|h-k|>1$, we have $\widehat{E}_{h+1,h}\widehat{E}_{k+1,k}=0$ and
$$F_hF_k=(E_{h+1,h}+E_{k+1,k})(\mathbf{0})=F_kF_h$$
and
$$E_h E_k=(E_{h,h+1}+E_{k,k+1})(\mathbf{0})= E_kE_h.$$

\noindent
(QS5) These are the quantum Serre relations when $h\neq m$. By applying Lemma \ref{integral generators}, the argument in the proof of \cite[Lem.~5.6]{BLM} carries over; see also the proof of \cite[Th.~13.33]{DDPW}. Note that, when $h=m-1$, there is no sign affecting the computation of $E_{m-1}E_mE_{m-1}$ or $E_mE_{m-1}^2$.

\noindent
(QS6) The relation $ E_m^2=0=F_m^2$ is given in Lemma \ref{integral generators}(4).
Finally, we prove
$$[\eta(\sfE_{m+2,m-1}),F_m]=\eta(\sfE_{m+2,m-1})F_m+F_m\eta(\sfE_{m+2,m-1})=0,$$
where, by \eqref{qbrackets},
$$\eta(\sfE_{m+2,m-1})= F_{m+1}F_m F_{m-1}-
\up^{-1}F_m F_{m+1} F_{m-1}-
\up F_{m-1} F_{m+1}F_m+ F_{m-1}F_m F_{m+1}.$$

Applying Lemma \ref{integral generators}(2) repeatedly yields
\begin{equation*}\label{(1)}
\aligned
(1) \qquad F_{m+1}F_m F_{m-1}F_m&= F_{m+1}F_m(E_{m,m-1}+E_{m+1,m})(\bf0)\\
&= F_{m+1}(E_{m+1,m-1}+E_{m+1,m})(\bf0)\\
&=(E_{m+2,m-1}+E_{m+1,m})(\mathbf{0})+
\up(E_{m+1,m-1}+E_{m+2,m})(\mathbf{0})+\\
&\quad\;
\up^2(E_{m+1,m-1}+E_{m+1,m}+E_{m+2,m+1})(\mathbf{0}),
\endaligned\end{equation*}
\begin{equation*}\label{(2)}
\aligned
(2)\;\;\;F_m F_{m+1} &F_{m-1}F_m\\
&=F_m F_{m+1}(E_{m,m-1}+E_{m+1,m})(\bf0)\\
&=F_m((E_{m,m-1}+E_{m+2,m})({\bf0})+\up((E_{m,m-1}+E_{m+1,m}+E_{m+2,m+1})(\bf0)))\\
&=(E_{m+1,m-1}+E_{m+2,m})({\bf0})+\up^{-1}(E_{m,m-1}+E_{m+1,m}+E_{m+2,m})({\bf0})\\
&\quad+\up(E_{m+1,m-1}+E_{m+1,m}+E_{m+2,m+1})(\bf0),\\
\endaligned\end{equation*}
\noindent
(3) $ F_{m-1} F_{m+1}F_mF_m=0,$ and

$$\aligned
(4)\quad  F_{m-1}F_m F_{m+1}&F_m= F_{m-1}F_m((E_{m+2,m})({\bf0})+\up(E_{m+1,m}+E_{m+2,m+1})({\bf0}))\\
&= F_{m-1}(E_{m+2,m}+E_{m+1,m})({\bf0})=(E_{m,m-1}+E_{m+1,m}+E_{m+2,m})({\bf0}).
\endaligned$$
Thus,
$$\aligned
\eta(\sfE_{m+2,m-1})F_m&=(1)-\up^{-1}(2)-\up(3)+(4)\\
&=(E_{m+2,m-1}+E_{m+1,m})({\bf0})+(\up-\up^{-1})(E_{m+1,m-1}+E_{m+2,m})({\bf0})\\
&\quad\;+(\up^2-1)(E_{m+1,m-1}+E_{m+1,m}+E_{m+2,m+1})({\bf0})\\
&\quad\;+(1-\up^{-2})(E_{m,m-1}+E_{m+1,m}+E_{m+2,m})({\bf0}).
\endaligned$$

We now compute the summands of
$F_m\eta(\sfE_{m+2,m-1})$. By applying Lemma \ref{integral generators}(2) repeatedly, we have
$$\aligned
(1)'\;\; F_m F_{m+1}&F_m F_{m-1}=F_m F_{m+1}((E_{m+1,m-1})({\bf0})+\up^{-1}(E_{m+1,m}+E_{m,m-1})({\bf0}))\\
&=F_m((E_{m+2,m-1})({\bf0})+\up(E_{m+1,m-1}+E_{m+2,m+1})({\bf0})+\\
&\quad\,\up^{-1}(E_{m+2,m}+E_{m,m-1})({\bf0})+(E_{m+1,m}+E_{m+2,m+1}+E_{m,m-1})({\bf0}))\\
&=-(E_{m+2,m-1}+E_{m+1,m})({\bf0})+\up^{-1}(E_{m+1,m-1}+E_{m+2,m})({\bf0})+\\
&\quad\,\up^{-2}(E_{m,m-1}+E_{m+1,m}+E_{m+2,m})({\bf0}),
\endaligned$$

(2)$'$ \,\,$\up^{-1}F_mF_m F_{m+1} F_{m-1}=0$,
$$\aligned
(3)'\;\,F_m F_{m-1} &F_{m+1}F_m=F_m F_{m-1}((E_{m+2,m})({\bf0})+\up(E_{m+1,m}+E_{m+2,m+1})({\bf0}))\\
&=(E_{m+1,m-1}+E_{m+2,m})({\bf0})+\up^{-1}(E_{m+2,m}+E_{m+1,m}+E_{m,m-1})({\bf0})+\\
&\quad\;\up(E_{m+1,m-1}+E_{m+1,m}+E_{m+2,m+1})({\bf0}),\\
\endaligned$$
and
$$\aligned
(4)'\quad F_m F_{m-1}F_m F_{m+1}&=F_m F_{m-1}((E_{m+2,m+1}+E_{m+1,m})({\bf0}))\\
&=F_m((E_{m+2,m+1}+E_{m+1,m}+E_{m,m-1})({\bf0}))\\
&=(E_{m+1,m-1}+E_{m+1,m}+E_{m+2,m+1})({\bf0}).
\endaligned$$
Thus,
$$\aligned F_m\eta(\sfE_{m+2,m-1})&=(1)'-\up^{-1}(2)'-\up(3)'+(4)'\\
&=-(E_{m+2,m-1}+E_{m+1,m})({\bf0})+(\up^{-1}-\up)(E_{m+1,m-1}+E_{m+2,m})({\bf0})\\
&\quad\;+(1-\up^2)(E_{m+1,m-1}+E_{m+1,m}+E_{m+2,m+1})({\bf0})\\
&\quad\;+(\up^{-2}-1)(E_{m,m-1}+E_{m+1,m}+E_{m+2,m})({\bf0}).
\endaligned$$
Hence, $[\eta(\sfE_{m+2,m-1}),F_m]=0.$
Similarly, we can prove $[ E_m,\eta(\sfE_{m-1,m+2})]=0.$ This completes the proof of
the theorem.
\end{proof}
\begin{corollary}There is a $\mathbb Q(\up)$-superalgebra homomorphism
$$\eta_r:\bfU(\mathfrak{gl}_{m|n})\longrightarrow \bsS(m|n,r)$$
sending $\sfE_{h,h+1},\sfE_{h+1,h}$ and $\sfK_a^{\pm1}$ to $E_{h,h+1}({\bf0},r),F_{h+1,h}({\bf0},r)$ and $O(\pm\bse_a,r)$, respectively.
\end{corollary}
In the last two section, we  will prove that $\eta$ induces an algebra isomorphism from $\bfU(\mathfrak{gl}_{m|n})$ to $\fA(m|n)$. In other words, we want to show that $\eta$ is injective and its image is $\fA(m|n)$. For the latter, we prove that $\fA(m|n)$ is generated by $\ugk_a^\pm$, $E_h$, and $F_h$; while, for the former, we show that $\eta$ sends a basis to a linearly independent set. We will see a monomial basis plays a key role in the proof.

\section{A super triangular relation}

We now aim to construct a monomial basis and a triangular relation from the monomial basis to the basis $\{A(\bsj)\}_{A,\bsj}$. In this section, we first make the construction at the $\up$-Schur superalgebra level.

We first recall \eqref{jmath} and the Bruhat order on $M(m|n,r)$ from \cite{DR}: if $A=\jmath(\la,y,\mu)$ and $B=\jmath(\la,w,\mu)$, then we define
$$A\leq B \iff y\leq w\;(\text{the Bruhat order on $W$})\iff W_\la yW_\mu\leq W_\la wW_\mu.$$
We first observe the following super version of \cite[Prop. 7.38]{DDPW} and \cite[Lem.~9.7]{DDPW} continue to hold. For a double coset $D=W_\la d W_\mu$ ($d\in\sD_{\la\mu}^\circ$), let 
$$T_D=\sum_{z\in\sD_{\la d\cap\mu}\cap W_\mu}(-\bsq)^{-\ell(z_1)}x_{\la^{(0)}}y_{\la^{(1)}}T_dT_z$$ be defined as in \cite[(5.3.2)]{DR}, where $x_{\la^{(0)}}y_{\la^{(1)}}=\sum_{w\in W_\la}(-\bsq)^{\ell(w_1)}T_w$.

\begin{lemma} For $d\in\sD_{\la\mu}^\circ$ and $d'\in\sD_{\mu\nu}^\circ$, let $D=W_\la d W_\mu$, $D'=W_\mu d W_\nu$, $A=\jmath(\la,d,\mu)$ and $A'=\jmath(\mu,d',\nu)$.
\begin{enumerate}
\item If $T_DT_{D'}=\sum_{D''\in W_\la \sD^\circ_{\la\nu}W_\nu}f_{D,D',D''}T_{D''}$, then there exists
a $D_0$ with $f_{D,D',D_0}\neq0$ and $D''\leq D_0$ whenever $f_{D,D',D''}\neq0$.
\item If $\xi_A\xi_{A'}=\sum_{A''\in M(m|n,r)}g_{A,A',A''}\xi_{A''}$, then there exists
a $A_0$ with $g_{A,A',A_0}\neq0$ and $A''\leq A_0$ whenever $g_{A,A',A''}\neq0$.
\end{enumerate}
\end{lemma}
\begin{proof} For (1), the proof of \cite[Prop.~7.38]{DDPW} carries over to the super case.  A key observation is that the leading term in $T_DT_{D'}$ is $T_{w_D*w_{D'}}$ which must occurs in some $T_{D''}$ of the RHS.
By \cite[8.3,8.4]{DGW}, $\xi_A=\pm\up^*\phi_A$ where $\phi_A$ is defined in \cite[(5.7.1)]{DR}.
Thus, $f_{D,D',D''}=d_\mu g_{A,A',A''}$ for some $d_\mu\in\sZ$, (2) follows from (1).
\end{proof}

The following lemma is the super version of \cite[Lem.~3.8]{BLM} (see also \cite[Lem.~13.22]{DDPW}). The ``lower terms" in an expression of the form ``$\xib_A+\mbox{
lower
 terms}$" means a $\mathbb{Q}(\up)$-linear combination of elements
 $\xib_B$ with $B\in M(m|n,r)$ and $B<A$ (the Bruhat ordering on $M(m+n,r)$ defined above.

\begin{lemma}\label{lemma of BLM triangular relations}
Assume that $1\leq k\leq m+n$, $1\leq h<m+n$, and $p\in \mathbb{N}.$ Let
$A\in M(m|n,r)$, $U_p=U_{p}(h,\ro(A))$, and $L_p=L_{p}(h,\ro(A))$ as in (\ref{$U_p,L_p$}).
\begin{itemize}
\item[(1)] If $A\in M(m|n,r)$ has the form
\begin{equation*}
A=\left(
\begin{array}{cccccc}
\cdots&\cdots&\cdots&&\cdots&\\
\cdots&a_{h,k-1}&0&0&\cdots&0\\
\cdots&a_{h+1,k-1}&a_{h+1,k}&0&\cdots&0\\
\cdots&\cdots&\cdots&&\cdots&
\end{array}
\right).
\end{equation*}
with $a_{h,j}=0$, for $k\leq j\leq m+n$,
$a_{h+1,j}=0$, for $k+1\leq j\leq m+n$, and $a_{h+1,k}\geq p$, then
$$\aligned
\text{\rm(a)}\quad\xib_{U_p}\xib_A&=\xib_{A_p}+\mbox{lower terms,}\text{ for $h\neq m$, and}\\
\text{\rm(b)}\quad\xib_{U_1}\xib_A&=(-1)^{\sum_{i>m,j<k}a_{i,j}}\xib_{A_1}+\mbox{lower
terms, for $h=m$.}\\
\endaligned$$
 where $A_p=A+pE_{h,k}-pE_{h+1,k}$.
\item[(2)] If $A\in M(m|n,r)$ has the form
\begin{equation*}
A=\left(
\begin{array}{cccccc}
&\cdots&&\cdots&\cdots&\cdots\\
0&\cdots&0&a_{h,k}&a_{h,k+1}&\cdots\\
0&\cdots&0&0&a_{h+1,k+1}&\cdots\\
&\cdots&&\cdots&\cdots&\cdots
\end{array}
\right).
\end{equation*}
with $a_{h,j}=0$, for $1\leq j\leq k-1$,
$a_{h+1,j}=0$, for $1\leq j\leq k$, and $a_{h,k}\geq p$, then
$$\aligned
\text{\rm(a)}\quad\xib_{L_p}\xib_A&=\xib_{A_p}+\mbox{lower terms, for $h\neq m$, and}\\
\text{\rm(b)}\quad\xib_{L_1}\xib_A&=(-1)^{\sum_{i>m,j<k}a_{i,j}}\xib_{A_1}+\mbox{lower
terms, for $h=m$},\endaligned$$ where $A_p=A-pE_{h,k}+pE_{h+1,k}$.
\end{itemize}
\end{lemma}
\begin{proof}
When $h\neq m$, comparing Proposition \ref{integral basis
multiplication} with \cite[theorem 13.18]{DDPW}($1'$) and ($2'$), we
find that if we replace $\up$ in \cite[theorem 13.18]{DDPW}($1'$) and
($2'$) with $\up_h$, we get Proposition \ref{integral basis
multiplication}. Similar to the proof of \cite[13.22]{DDPW}, we have
$\xib_{U_p}\xib_A=\xib_{A_p}+\mbox{lower terms}$ and
$\xib_{L_p}\xib_A=\xib_{A_p}+\mbox{lower terms}$ respectively.

When $h=m$, by Lemma \ref{main lemma},
if $\xib_{U_p}\xib_A\neq 0$ or $\xib_{L_p}\xib_A\neq 0$, then
$p=1$. Applying Proposition \ref{integral basis multiplication $h=m$} to the matrix  $A$ in (1) or (2)
gives the required formulas.
\end{proof}

Let
$$\mathscr{T}=\mathscr T(m+n)=\{(i,h,j)\mid 1\leq i\leq h<j\leq m+n\}.$$
There are two total ordering $\leq_i(i=1,2)$ on $\mathscr{T}$ which are
defined as follows (see \cite[p.562]{DDPW}):
\begin{equation}\label{triangular relations total ordering}
\left\{
\begin{aligned}
(i,h,j)&\leq_1(i',h',j')\\
\iff &\mbox{ one of the following three conditions is
satisfied:}\\
&(\verb"i")j>j',(\verb"ii")j=j',i>i', \mbox{ and }(\verb"iii")j=j',i=i',h\leq h';\\
(i,h,j)&\leq_2(i',h',j')\\
\iff &\mbox{ one of the following three conditions is
satisfied:}\\
&(\verb"i")i<i',(\verb"ii")i=i',j<j', \mbox{ and
}(\verb"iii")j=j',i=i',h\geq h'.
\end{aligned}
\right.
\end{equation}
For example, $\mathscr{T}(3)=\{(2,2,3),(1,1,3),(1,2,3),(1,1,2)\}$ under $\leq_1$. In general,
\begin{equation}\label{two orders}
(\mathscr T,\leq_1)=(\scrT_{m+n},\ldots,\scrT_3,\scrT_2)\quad\text{and}\quad
(\mathscr T,\leq_2)=(\scrT_1',\scrT_2',\ldots,\scrT_{m+n-1}')
\end{equation}
 where, for every $m+n\geq j\geq2$ and $1\leq l< m+n$, $\scrT_j$ is the sequence (formed by inserting top row entries into the $*$ entry below):
\begin{equation}\label{T_j}
\underbrace{j-1}_{(j-1,*,j)\;\,\;}\underbrace{j-2,j-1}_{(j-2,*,j)\;\;\;}\;\ldots\;\underbrace{i,i+1,\ldots,j-1}_{(i,*,j)}\;\ldots\;\underbrace{1,2,\ldots,j-1}_{(1,*,j)}
\end{equation}
while $\scrT'_l$ is the sequence:
\begin{equation}\label{T'_k}
\underbrace{l}_{(l,*,l+1)}\;\ldots\;\underbrace{l+k-1,\dots,l+1,l}_{(l,*,l+k)}\;\ldots\;\underbrace{m+n-1,\ldots,l+1,l}_{(l,*,m+n)}.
\end{equation}

 The following definition is taken from
 \cite[Def.~13.23]{DDPW} which modifies the definition given in \cite{BLM}.

\begin{definition}\label{definition $E^A,F^A$}
Let $A=(a_{i,j})\in {M}(m|n)$ and define recursively almost
diagonal matrices $E^{(A)}_{i,h,j}$, for $(i,h,j)\in(\mathscr{T},\leq_1)$ as follows:
\begin{enumerate}
\item $E_{1,1,2}^{(A)}$ is the matrix defined by the conditions that
$\co(E_{1,1,2}^{(A)})=\co(A)$ and $E^{(A)}_{1,1,2}-a_{1,2}E_{1,2}$
is diagonal;
\item If $(i,h,j)$ is the immediate predecessor of $(i',h',j')$, then
 $E_{i,h,j}^{(A)}$ is defined by the conditions that
 $\co(E_{i,h,j}^{(A)})=\ro(E_{i',h',j'}^{(A)})$ and
 $E_{i,h,j}^{(A)}-a_{i,j}E_{h,h+1}$ is diagonal.
\end{enumerate}
 Similarly, define recursively almost
diagonal matrices $F^{(A)}_{i,h,j}$, for $(i,h,j)\in(\mathscr{T},\leq_2)$,
with respect to $\leq_2$ as follows:
\begin{enumerate}
\item $F_{N-1,N-1,N}^{(A)}$ ($N=m+n$) is the matrix defined by the
conditions that
$F_{N-1,N-1,N}^{(A)}-a_{N,N-1}E_{N,N-1}$ is diagonal
and
$\co(F_{N-1,N-1,N}^{(A)})=\ro(E_{N-1,N-1,N}^{(A)})$;
\item If $(i,h,j)$ is the immediate predecessor of $(i',h',j')$, then
 $F_{i,h,j}^{(A)}$ is defined by the conditions that
 $\co(F_{i,h,j}^{(A)})=\ro(F_{i',h',j'}^{(A)})$ and
 $F_{i,h,j}^{(A)}-a_{j,i}E_{h+1,h}$ is diagonal.
 \end{enumerate}
\end{definition}

It is easy to see from the definition that, if  $A\in M(m|n,r)$,
then all $E_{i,h,j}^{(A)}$ and $F_{i,h,j}^{(A)}$ are in $M(m|n,r)$.

For $A\in M(m|n)$, define $\bar A$ as in \eqref{signAbar}.

\begin{theorem}\label{BLM triangular}
Maintain the notation introduced above and let $A=(a_{i,j})\in
M(m|n,r)$. The following triangular relation holds in
$\sS_\up(m|n,r)$:
\begin{equation}\label{BLM triangular relations formula}
\Psi_A:=\prod_{(i,h,j)\in(\mathscr T,\leq_2)}\xib_{F_{i,h,j}^{(A)}}\cdot\prod_{(i,h,j)\in(\mathscr T,\leq_1)}\xib_{E_{i,h,j}^{(A)}}=(-1)^{\overline{A}}\xib_A+\mbox{lower
terms},
\end{equation}
where the products are taken over the total orderings listed in \eqref{two orders}. In particular, the set $\{\Psi_A\}_{A\in M(m|n,r)}$ forms a basis for $\sS_\up(m|n,r)$.

\end{theorem}
\begin{proof} We follow the proof of \cite[13.24]{DDPW} by repeatedly applying Lemma \ref{lemma of BLM triangular relations} and carefully manipulating the sign occurred in Lemma \ref{lemma of BLM triangular relations}(1b)\&(2b).

We start with the largest element $(1,1,2)$ in $(\scrT,\leq_1)$ and let $A_{1,1,2}=E_{1,1,2}^{(A)}$.
Repeatedly applying Lemma \ref{lemma of BLM triangular relations}(1) (and noting Lemma \ref{Bruhat}) yields, for $(i,h,j)\in(\scrT,\leq_1)$
$$\prod_{\substack{(i',h',j')\in(\mathscr T,\leq_1)\\(i,h,j)\leq_1(i',h',j')}}\xib_{E_{i',h',j'}^{(A)}}=(-1)^{N_{i,h,j}}\xi_{A_{i,h,j}}+\text{lower terms},$$
where $A_{i,h,j}$ is the upper triangular matrix
\begin{equation}\label{$A,i<j$}
\left(\begin{array}{cccccccc}
\lambda_1 &a_{1,2}&\cdots&a_{1,j-1}&a_{1,j}&0&\cdots\\
0&\lambda_2-a_{1,2}&\cdots&a_{2,j-1}&a_{2,j}&0&\cdots\\
\vdots&\vdots&&\vdots&\vdots&\vdots&\\
0&0&\cdots&a_{i-1,j-1}&a_{i-1,j}&0&\cdots\\
0&0&\cdots&a_{i,j-1}&0&0&\cdots\\
\vdots&\vdots&&\vdots&\vdots&\vdots&\\
0&0&\cdots&a_{h,j-1}&a_{i,j}&0&\cdots\\
\vdots&\vdots&&\vdots&\vdots&\vdots&\\
0&0&\cdots&\lambda_{j-1}-\sum_{l=1}^{j-2}a_{l,j-1}&0&0&\cdots\\
0&0&\cdots&0&\lambda_{j}-\sum_{l=1}^{i}a_{l,j}&0&\cdots\\
0&0&\cdots&0&0&\lambda_{j+1}&\cdots\\
\vdots&\vdots&&\vdots&\vdots&\vdots&
\end{array}
\right)
\end{equation}
Note that, if $j\leq m$ and $(i',h',j')\geq_1(i,h,j)$, then $h'<m$. So only Lemma \ref{lemma of BLM triangular relations}(1a) applies. Hence, all $N_{i,h,j}=0$ unless $j\geq m+1$.
Similarly, if $j\geq m+1$ and $(i,h,j)$ is an immediate predecessor or an immediate successor of $(i',h',j')$ with $h\neq m$, then  $N_{i,h,j}=N_{i',h',j'}$. Thus, if $j>m+1$, then (see \eqref{T_j})
\begin{equation}\label{N_{1,m+1,j}}
\begin{aligned}
(1)&\;N_{1,m+1,j}=N_{1,m+2,j}=\cdots=N_{1,j-1,j}=N_{j-2,j-2,j-1}=\ldots=N_{m,m,j-1},&\\
(2)&\;N_{i+1,m+1,j}=\cdots=N_{i+1,j-1,j}=N_{i,i,j}=\cdots=N_{i,m-1,j}= N_{i,m,j}, \qquad\quad\;\,\mbox{ for } i\leq m,\\
(3)&\;N_{j-1,j-1,j}=\cdots=N_{i+1,i+1,j}=\cdots=N_{i+1,j-1,j}=N_{i,i,j}=\cdots=N_{i,j-1,j},\mbox{ for } i>m.
\end{aligned}
\end{equation}
Hence, all $N$-values are determined by $N_{i,m,j}$ with $1\leq i\leq m$ and $m+1\leq j\leq m+n$.

We first claim the following recursive formula:
$$N_{m,m,j}=\begin{cases} 0,&\text{ if }j=m+1;\\
N_{m,m,j-1}+\displaystyle\sum_{\substack{m+n\geq
i'>m\geq i\geq1,\\m<j'<j\leq m+n}}a_{i',j'}a_{i,j},&\text{ if }j\geq m+2.\\
\end{cases}
$$
Indeed, $N_{m,m,m+1}=0$ is clear since the $(k,l)$ entry $(A_{i,h,m+1})_{k,l}=0$ whenever $k>m$ and $l<m+1$. Assume now $j\geq m+2$. For $i=1$, by \eqref{N_{1,m+1,j}}(1), $N_{1,m+1,j}=N_{m,m,j-1}$.  Then, by \eqref{N_{1,m+1,j}}(2) and noting $a_{i,j}=0$ or $1$ for all $i\leq m$,
$$\aligned
N_{1,m,j}&=N_{1,m+1,j}+\sum_{i'>m,m<j'<j}a_{i',j'}a_{1,j}=N_{m,m,j-1}+\sum_{i'>m,m<j'<j}a_{i',j'}a_{1,j},\\
N_{2,m,j}&=N_{1,m,j}+\sum_{i'>m,m<j'<j}a_{i',j'}a_{2,j}=N_{m,m,j-1}+\sum_{i'>m,m<j'<j}a_{i',j'}(a_{1,j}+a_{2,j}),\\
\cdots&\quad\,\cdots,\\
N_{m,m,j}&=N_{m-1,m,j}+\sum_{i'>m,m<j'<j}a_{i',j'}a_{m,j}=N_{m,m,j-1}+\sum_{\substack{i'>m\geq i\geq1\\m<j'<j}}a_{i',j'}a_{i,j},\endaligned$$
proving the claim.

By the claim, we obtain a close formula:
$$N_{m,m,m+n}=\sum_{\substack{m+n\geq i>m\geq k\geq
1,\\m<j<l\leq m+n}}a_{i,j}a_{k,l}=\bar{A}.$$
Thus, by \eqref{N_{1,m+1,j}}(3), $N_{m,m,m+1}=N_{m+n-1,m+n-1,m+n}$ and so
$$\xi(E):=\prod_{(i,h,j)\in(\scrT,\leq_1)}\xib_{E^{(A)}_{i,h,j}}=(-1)^{\bar{A}}\xib_{A_{m+n-1,m+n,m+n}}+\mbox{ lower
terms}.$$

Now, for $(i,h,j)\in(\scrT,\leq_2)$,
$$\prod_{\substack{(i',h',j')\in(\mathscr T,\leq_2)\\(i,h,j)\leq_1(i',h',j')}}\xib_{F_{i',h',j'}^{(A)}}\cdot \xi(E)=(-1)^{N_{i,h,j}'+\bar A}\xi_{A'_{i,h,j}}+\text{lower terms},$$
where $A'_{i,h,j}$ is a matrix of the form
\begin{equation*}\label{$A,i>j$}
\left(\begin{array}{cccccccc}
\lambda_1 &a_{1,2}&\cdots&a_{1,i-1}&a_{1,i}&a_{1,i+1}&\cdots\\
0&\lambda_2-a_{1,2}&\cdots&a_{2,i-1}&a_{2,i}&a_{2,i+1}&\cdots\\
\vdots&\vdots&&\vdots&\vdots&\vdots&\\
0&0&\cdots&\lambda_{i-1}-\sum_{l=1}^{i-2}a_{l,i-1}&a_{i-1,i}&a_{i-1,i+1}&\cdots\\
0&0&\cdots&0&\lambda_{i}-\sum_{l<i}a_{l,i}-\sum_{l\geq j}a_{l,i}&a_{i,i+1}&\cdots\\
\vdots&\vdots&&\vdots&\vdots&\vdots&\\
0&0&\cdots&0&0&a_{h,i+1}&\cdots\\
0&0&\cdots&0&a_{j,i}&a_{h+1,i+1}&\cdots\\
0&0&\cdots&0&0&a_{h+2,i+1}&\cdots\\
\vdots&\vdots&&\vdots&\vdots&\vdots&\\
0&0&\cdots&0&0&a_{j,i+1}&\cdots\\
0&0&\cdots&0&a_{j+1,i}&a_{j+1,i+1}&\cdots\\
\vdots&\vdots&&\vdots&\vdots&\vdots&\\
0&0&\cdots&0&a_{m+n,i}&a_{m+n,i+1}&\cdots
\end{array}
\right)
\end{equation*}
(In particular, $A'_{1,1,2}=A$.) From the matrix above, we see that
$(A'_{i,m+1,j})_{j',i'}=0$ for $j'>m,i'<i$.
Hence, by Lemma \ref{lemma of BLM triangular relations}, all $N'_{i,h,j}=0$ and
 $$\prod_{(i,h,j)\in(\scrT,\leq_2)}\xib_{F_{i,h,j}^{(A)}}\cdot\xi(E)=(-1)^{\overline{A}}\xib_A+\mbox{lower
terms},$$
as required.
\end{proof}



\section{Realisation of the quantum general linear supergroups}

We are now in position to solve the realisation problem for $\bfU(\mathfrak{gl}_{m|n})$ by first determining the image of the homomorphism $\eta$ and then proving that $\eta$ is injective. This requires another triangular relation with respect the preorder $\preceq$ on $M(m|n)$ (or more precisely on $M(m|n)^\pm$) which we define now. This order has already been implicitly used in the proof of Lemma \ref{lemma of BLM triangular relations}.

For
$A=(a_{i,j}),A'=(a'_{i,j})\in M(m|n)$, define
 \begin{equation}\label{prec}
 A'\preceq A\iff
\begin{cases}
(1)\quad \sum_{i\leq s,j\geq t}a'_{i,j}\leq \sum_{i\leq
s,j\geq t}a_{i,j},&\text{for all $s<t$};\\
(2)\quad \sum_{i\geq s,j\leq t}a'_{i,j}\leq \sum_{i\geq
s,j\leq t}a_{i,j},&\text{for all $s>t$}\end{cases}
\end{equation}
and
$$\|A\|=\sum_{i<j}{j-i+1\choose2}(a_{i,j}+a_{j,i}).$$
These definitions are independent of the diagonal entries of a matrix.
Moreover, the following is taken from
\cite[Lem.~3.6(1)]{BLM} (see also \cite[Lem.~13.20,13.21]{DDPW}).
\begin{equation}\label{prec order}
A<B\,(\text{the Bruhat order})\implies A^\pm \prec B^\pm\implies \|A^\pm\|<\|B^\pm\|,
\end{equation}
where $A,B\in M(m|n,r)$ and $X^\pm$ is the matrix obtained by replacing the diagonal entires by 0.

\begin{theorem}\label{BLM triangular relations in V} The (super) subspace
$\mathfrak A(m|n)$ of the algebra $\bsS(m|n)$ defined as in \eqref{A(j)} is the (super) subalgebra generated by
$$E_h=E_{h,h+1}(\mathbf{0}),
F_h=E_{h+1,h}(\mathbf{0}),\ugk_i^{\pm1}=O(\pm\bse_i),$$
for
all $1\leq h<m+n$ and $1\leq i\leq m+n$.
\end{theorem}
\begin{proof} Let $\fA(m|n)'$ be the subalgebra of $\bsS(m|n)$ generated by $E_h,
F_h,\ugk_i^{\pm1}.$
 By Proposition \ref{h neq m}, we have $\fA(m|n)'\subseteq \fA(m|n)$. We now prove $\fA(m|n)\subseteq \fA(m|n)'$ by induction on $\|A\|$. Clealy,  if $\|A\|=0$, $A(\bsj)=O(\bsj)=\prod_{i=1}^{m+n}\ugk_i^{j_i}\in\fA(m|n)'$ for all $\bsj\in\mathbb Z^{m+n}$. Assume now $\|A\|>0$ and $B(\bsj)\in\fA(m|n)'$ for all $\bsj\in\mathbb Z^{m+n}$, whenever $\|B\|<\|A\|$. We need a triangular relation to complete the proof.

  By Lemma \ref{integral generators}, we have
$$(a_{i,j}E_{h,h+1})({\bf0})=\frac{E_{h,h+1}^{a_{i,j}}({\bf0})}{[a_{i,j}]_{\up_h}^!}\,\text{ and }\,
(a_{j,i}E_{h+1,h})({\bf0})=\frac{E_{h+1,h}^{a_{j,i}}({\bf0})}{[a_{j,i}]_{\up_{h+1}}^!}.
$$
Thus, repeatedly applying Proposition \ref{h neq m} yields
$$\prod^{(\leq_2)}_{(i,h,j)\in\scrT}(a_{j,i}E_{h+1,h})({\bf0})\prod^{(\leq_1)}_{(i,h,j)\in\scrT}(a_{i,j}E_{h,h+1})({\bf0})=\sum_{\substack{B\in
M(m|n)^\pm\\ \bsj\in\mathbb Z^{m+n}}}g_{A,B,\bsj}B(\bsj).$$
We now prove that $g_{A,A,{\bf0}}=1$ and $B\prec A$ whenever $g_{A,B,\bsj}\neq0$ and $B\neq A$. In other words, we show that the triangular relation established in Theorem \ref{BLM triangular} induces a triangular relation of the form
\begin{equation}\label{triangular2}
\aligned
\prod^{(\leq_2)}_{(i,h,j)\in\scrT}(a_{j,i}E_{h+1,h})(\mathbf{0})&\prod^{(\leq_1)}_{(i,h,j)\in\scrT}(a_{i,j}E_{h,h+1})(\mathbf{0})
=A(\mathbf{0})+\sum_{\substack{B\in
M(m|n)^\pm,\bsj\in\mathbb Z^{m+n}\\ B\prec
A}}g_{B,A,\bsj}B(\bsj),
\endaligned
\end{equation}
or equivalently, for all $r\geq0$,
$$\prod_{(i,h,j)\in\scrT}^{(\leq_2)}(a_{j,i}E_{h+1,h})(\mathbf{0},r)\prod^{(\leq_1)}_{(i,h,j)\in\scrT}(a_{i,j}E_{h,h+1})({\bf0},r)=A(\mathbf{0},r)+\sum_{\substack{B\in
M(m|n)^\pm\\\bsj\in\mathbb Z^{m+n}\\ B\prec
A}}g_{B,A,\bsj}B(\bsj,r).$$

Observe, for $\lambda,\mu\in\Lambda(m|n,r-|A|)$ and $\lambda\neq \mu$, the orthogonality relations
$$\xib_{E^{(A+\lambda)}_{i,h,j}}\xib_{E^{(A+\mu)}_{i,h,j}}=0,\quad
\xib_{F^{(A+\lambda)}_{i,h,j}}\xib_{F^{(A+\mu)}_{i,h,j}}=0,\;\text{ and }\;\xib_{F^{(A+\lambda)}_{N-1,N-1,N}}\xib_{E^{(A+\mu)}_{N-1,N-1,N}}=0,$$
where $N=m+n$.
Thus,
$$\aligned\prod^{(\leq_2)}_{(i,h,j)\in\scrT}(&a_{j,i}E_{h+1,h})({\bf0},r)\prod^{(\leq_1)}_{(i,h,j)\in\scrT}(a_{i,j}E_{h,h+1})({\bf0},r)\\
&=\sum_{\lambda\in\Lambda(m|n,r-|A|)}\prod^{(\leq_2)}_{(i,h,j)\in\scrT}(-1)^{\overline{F^{(A+\lambda)}_{i,h,j}}}\xib_{F^{(A+\lambda)}_{i,h,j}}\prod^{(\leq_1)}_{(i,h,j)\in\scrT}(-1)^{\overline{E^{(A+\lambda)}_{i,h,j}}}\xib_{E^{(A+\lambda)}_{i,h,j}}.
\endaligned$$

Since $\overline{E^{(A+\lambda)}_{i,h,j}}=0$ and
$\overline{F^{(A+\lambda)}_{i,h,j}}=0$, by Theorem \ref{BLM triangular},
$$\aligned&\prod^{(\leq_2)}_{(i,h,j)\in\scrT}
(a_{j,i}E_{h+1,h})({\bf0},r)\prod^{(\leq_1)}_{(i,h,j)\in\scrT}
(a_{i,j}E_{h,h+1})({\bf0},r)\\
&=\sum_{\lambda\in\Lambda(m|n,r-|A|)}\prod^{(\leq_2)}_{(i,h,j)\in\scrT}
\xib_{F^{(A+\lambda)}_{i,h,j}}\prod^{(\leq_1)}_{(i,h,j)\in\scrT}
\xib_{E^{(A+\lambda)}_{i,h,j}}\\
&=\sum_{\lambda\in\Lambda(m|n,r-|A|)}((-1)^{\overline{A+\lambda}}\xib_{A+\lambda}+\mbox{lower
terms})\\
&=A(\mathbf{0},r)+ \sum_{\substack{B\in M(m|n)^\pm,B\prec A\\\bsj\in\mathbb Z^{m+n}}}g_{B,A,\bsj}B(\bsj,r),\text{ by \eqref{prec}},
\endaligned$$
where ``lower terms" is a linear combination of $\xib_{B}$ with
$B\in M(m|n)$ and $B<A$ under Chevalley-Bruhat ordering in $M(m|n)$.
Hence, \eqref{triangular2} follows.

To complete the proof, by \eqref{prec order} and induction,
\eqref{triangular2} implies that $A({\bf0})\in\fA(m|n)'$. Finally,
by Proposition \ref{multiply O}, for any
$\bsj=(j_1,j_2,\cdots,j_{m+n})\in\mathbb{Z}^{m+n}$,
$A(\bsj)=\up^{-\co(A)\cdot\bsj} A(\mathbf{0})O(\bsj)\in \fA(m|n)'$.
\end{proof}

For any $A=(a_{i,j})\in M(m|n)$ and $\bsj\in\mathbb{Z}^{m+n}$, let
$${\rm M}_{A,\bsj}:=\prod^{(\leq_2)}_{(i,h,j)\in\scrT}(a_{j,i}E_{h+1,h})(\mathbf{0})\cdot O(\bsj)\cdot\prod^{(\leq_1)}_{(i,h,j)\in\scrT}(a_{i,j}E_{h,h+1})(\mathbf{0}).$$
Now, Lemma \ref{Bruhat} and \eqref{triangular2} implies immediately the first part of the following.

\begin{corollary} (1) The set
$\overline{\mathfrak M}=\{{\rm M}_{A,\bsj}\mid A\in M(m|n)^\pm, \bsj\in\mathbb Z^{m+n}\}$
forms a basis of homogeneous elements for $\fA(m|n)$.

(2) The image of the algebra homomorphism $\eta$ established in Theorem \ref{quantum serre relations} is the subalgebra $\fA(m|n)$.
\end{corollary}

We now use this basis to derive a basis, {\it a monomial basis}, for the quantum supergroup $\bfU=\bfU(\mathfrak{gl}_{m|n})$.

For any $A\in M(m|n)^\pm, \bsj\in\mathbb Z^{m+n}$, define
\begin{equation}\label{monomial}
\sfM^{A,\bsj}=\sfM'_1\sfM'_2\cdots\sfM'_{m+n-1}\sfK_1^{j_1}\cdots\sfK_{m+n}^{j_{m+n}}\sfM_{m+n}\cdots\sfM_3\sfM_2,
\end{equation}
where, for $2\leq j\leq {m+n}$ and $1\leq k\leq {m+n}-1$,
$$\aligned
\sfM_j&=\sfE_{j-1,j}^{(a_{j-1,j})}(\sfE_{j-2,j-1}^{(a_{j-1,j})}\sfE_{j-1,j}^{(a_{j-2,j})})\cdots (\sfE_{1,2}^{(a_{1,j})}\sfE_{2,3}^{(a_{1,j})}\cdots \sfE_{j-1,j}^{(a_{1,j})}),\\
\sfM'_k&=\sfE_{k+1,k}^{(a_{k+1,k})}(\sfE_{k+2,k+1}^{(a_{k+2,k})}\sfE_{k+1,k}^{(a_{k+2,k})})
\cdots(\sfE_{{m+n},{m+n}-1}^{(a_{{m+n},k})}\cdots \sfE_{k+2,k+1}^{(a_{{m+n},k})}\sfE_{k+1,k}^{(a_{{m+n},k})}),
\endaligned
$$
following the orders given in \eqref{T_j} and \eqref{T'_k}.
\begin{corollary}  The set
${\mathfrak M}=\{\sfM^{A,\bsj}\mid A\in M(m|n)^\pm, \bsj\in\mathbb Z^{m+n}\}$
forms a basis of homogeneous elements for $\bfU(\mathfrak{gl}_{m|n})$.
\end{corollary}
\begin{proof} Let $Q^+=\sum_{i=1}^{m+n}\mathbb N\alpha_{i}$ be the $+$-part of the root lattice, where $\alpha_i=\alpha_{i,i+1}$ are simple roots and let $\bfU^\pm$ be the $\pm$-part of $\bfU$. Then $\bfU^+=\oplus_{\alpha\in Q^+}\bfU_\alpha^+$ is $Q^+$-graded. By inspecting a PBW type basis,\footnote{For example, replacing $\sfM_j$ (resp., $\sfM'_k$) by $\sfE_{j-1,j}^{a_{j-1,j}}\sfE_{j-2,j}^{a_{j-2,j}}\cdots \sfE_{1,j}^{a_{1,j}}$ (reps., $\sfE_{k+1,k}^{a_{k+1,k}}\sfE_{k+2,k}^{a_{k+2,k}}\cdots\sfE_{N,k}^{a_{N,k}}$) yields such a basis.} we see that
$$\dim\bfU_\al^+=\#\{A\in M(m|n)^+\mid \alpha=\textstyle\sum_{i<j}a_{i,j}\alpha_{i,j}\},$$
where, for $i<j$, $\alpha_{i,j}=\alpha_i+\cdots+\alpha_{j-1}$.
Since $\eta(\mathfrak M)=\overline{\mathfrak M}$ is linearly independent, $\mathfrak M$ is linearly independent. In particular,
 $\mathfrak M^+:=\{\sfM^{A,\bf0}\mid A\in M(m|n)^+\}$ is linearly independent. A dimensional comparison shows that $\mathfrak M^+$ forms a basis for $\bfU^+$. Hence, $\mathfrak M$ forms a basis for $\bfU$ since $\bfU\cong\bfU^-\otimes\bfU^0\otimes\bfU^+$.
\end{proof}
 Now, $\eta$ sends $\mathfrak M$ to a basis. Hence $\eta$ is injective and so we have established the following realisation.
\begin{theorem} The quantum enveloping superalgebra $\bfU(\mathfrak{gl}_{m|n})$ is isomorphic to
the superalgebra $\fA(m|n)$. In particular, $\bfU(\mathfrak{gl}_{m|n})$ can be regarded as the $\mathbb Q(\up)$-superalgebra whose underlying superspace is spanned by
$$\{A(\bsj)\mid A\in M(m|n)^\pm, \bsj\in\mathbb Z^{m+n}\}$$ and whose multiplication is given by the formulas in Propositions \ref{multiply O} and \ref{h neq m}.
\end{theorem}


\begin{thebibliography}{99}
\bibitem {BLM} A.A. Beilinson, G. Lusztig and R. MacPherson, \textit{A geometric
 setting for the quantum deformation of $GL_n$}, Duke
 Math.J. {\bf 61} (1990), 655-677.
\bibitem{Br} T. Bridgeland, \emph{Quantum groups via Hall algebras of complexes}, Ann. Math. 2013 (available on-line).
\bibitem{DDF}
B. Deng, J. Du and Q. Fu, {\em A Double Hall Algebra Approach to Affine Quantum Schur--Weyl Theory},
LMS Lect. Note Ser. Vol. {\bf  401}, Cambridge University Press, Cambridge, 2012.
\bibitem {DDPW}
B. M. Deng, J. Du, B. Parshall, J. P. Wang,  \emph{Finite
Dimensional Alegebras and Quantum Groups}, Mathematical Surveys and
Monographs, Vol. 150, Amer. Math. Soc, Providence R. I. (2008).
\bibitem{Du91-2}J. Du, \emph{The modular representation theory of q-Schur algebras}, Trans. Am. Math. Soc. 329 (1992), 253--271; \emph{II}, Math. Z. 208 (1991), 503--536.
\bibitem{Du96} J. Du, \emph{Cells in certain sets of matrices}, T\^ohoku Math. J. {\bf 48}(1996), 417--427.
\bibitem{DGW} J. Du, H. Gu and J. Wang, \emph{Irreducible representations of $q$-Schur superalgebras at a root of unity},
arXiv:13010161.
\bibitem{DF09}
J. Du and Q. Fu, {\em A modified BLM approach to quantum affine
$\frak{gl}_n$}, Math. Z. {\bf 266} (2010), 747--781.
\bibitem{DR}
J. Du and H. Rui,\emph{Quantum schur superalgebras and Kazhdan-Lusztig
combinatrics}, Journal of Pure and Applied algebra, {\bf215} (2011),
2715--2737.
\bibitem{Gr95} J.~A. Green, {\em Hall algebras, hereditary algebras and
quantum groups}, Invent. Math. {\bf 120} (1995), 361--377.
\bibitem{J}L. Jones, \emph{Centers of generic algebras}, Trans. Am. Math. Soc. {\bf 317} (1990), 361--392.
\bibitem{Kac2} V. Kac, {\em Infinite dimensional Lie algebras}, Third edition,
Cambridge University Press, 1990.
\bibitem {Lu91} G. Lusztig, {\em Quivers, perverse sheaves, and quantized
enveloping algebras,} J. Amer. Math. Soc. {\bf 4} (1991), 365--421.
\bibitem{Lu93} G. Lusztig, {\em Introduction to quantum groups},
Progress in Math. {\bf 110}, Birkh\"auser, 1993.
\bibitem{Lu} G. Lusztig, {\em Aperiodicity in quantum affine $\mathfrak{gl}_n$},
Asian J. Math. {\bf 3} (1999), 147--177.
\bibitem{M}
Y. Manin, \emph{Multiparameteric quantum deformation of the general
linear supergroup}, Comm. Math. Physics {\bf123} (1989),163--175.
\bibitem{Mi} H. Mitsuhashi, \emph{SchurÐWeyl reciprocity between the quantum superalgebra and the IwahoriÐHecke algebra}, Algebr. Represent. Theory 9 (2006)
309--322.
\bibitem{PW}
B. Parshall, J.-p. Wang, \emph{Quantum linear groups}, Mem. Amer.
Math. Soc.Vol. 89, No. 439 (1991).
\bibitem{PX} L. Peng and J. Xiao, {\em Triangulated categories and
Kac--Moody algebras}, Invent. Math. {\bf 140}(2000), 563--603.
\bibitem{R90} C.~M. Ringel, {\em Hall algebras and quantum groups},
Invent. Math. {\bf 101} (1990), 583--592.
\bibitem{TK}
H. El Turkey and J. Kujawa, \emph{Presenting Schur superalgebras},
Pacific J. Math., to appear.
\bibitem{Y}
H. Yamane, \emph{Quantized enveloping algebras associated with
simple Lie superalgebras and their universal R-matrices}, Publ
RIMS,Kyoto. Univ. 30(1994),15--87.
\bibitem{Z}
R. Zhang, \emph{Finite dimensional irreducible representations of
the quantum supergroup $U_q(gl(m|n))$}, J. Math. Phys. Vol.34,
No.3 (1993), 1236--1254.
\end{thebibliography}
\end{document}